\documentclass[reqno,11pt,a4paper]{amsart}
\usepackage{exscale,amsmath,amssymb,amsthm,color,graphicx,accents,mathrsfs,mathtools,hyperref}
\usepackage[margin=1in]{geometry}
\usepackage{enumitem}
\setlist[enumerate,1]{leftmargin=1cm}

\newcommand{\cev}[1]{\accentset{\leftharpoonup}{#1}}

\theoremstyle{plain}

\newtheorem{theorem}{Theorem}

\newtheorem{proposition}[theorem]{Proposition}
\newtheorem{corollary}[theorem]{Corollary}
\newtheorem{lemma}[theorem]{Lemma}

\theoremstyle{definition}
\newtheorem{definition}[theorem]{Definition}

\theoremstyle{remark}

\newcommand{\abs}[1]{\left\lvert #1 \right\rvert}

\makeatletter
 \DeclareRobustCommand{\checkarg}{\@ifnextchar[{\@witharg}{}}
 \DeclareRobustCommand{\@witharg}[1][]{\ensuremath{\left(#1\right)}}
 
 \DeclareRobustCommand{\scaleGen}[1]{\@ifnextchar[{\@scalewithargs{#1}}{\odot^{}_{#1}}}
 \def\@scalewithargs#1[#2][#3]{#2 \odot^{}_{#1} #3}
\makeatother

\def\IPspace{\mathcal{I}}
\def\dI{d_{\IPspace}}

\def\cJ{\mathcal{J}}

\def\Exc{\mathcal{E}}

\def\mBxc{\nu_{\textnormal{BES}}}	

\def\cC{\mathcal{C}}	

\def\cN{\mathcal{N}}

\def\cT{\mathcal{T}}

\def\life{\zeta}					
\def\dis{\textnormal{dis}}			
\def\IPmag#1{\left\|\vphantom{I}#1\right\|}		

\def\skewer{\textsc{skewer}}		

\def\scaleB{\scaleGen{\textnormal{BES}}}		
\def\scaleH{\scaleGen{\textnormal{cld}}}		

\def\ShiftRestrict#1#2{#1\big|^{\from}_{#2}} 
\def\shiftrestrict#1#2{#1|^{\from}_{#2}}
\def\RestrictShift#1#2{#1\big|^{\to}_{#2}} 
\def\restrictshift#1#2{#1|^{\to}_{#2}}
\def\Restrict#1#2{#1\big|_{#2}}
\def\restrict#1#2{#1|_{#2}}

\def\Concat{ \mathop{ \raisebox{-2pt}{\Huge$\star$} } }
\def\ConcatIL{ \mbox{\huge $\star$} }
\def\concat{\star}

\newcommand{\IPLT}{\mathscr{D}}

\newcommand{\td}[1]{\widetilde{#1}}

\def\BR{\mathbb{R}}				
\def\Leb{\textnormal{Leb}}		

\def\to{\rightarrow}
\def\downto{\downarrow}

\def\from{\leftarrow}

\def\cf{\mathbf{1}}				
\def\Pr{\mathbf{P}}				
\def\cF{\mathcal{F}}			

\def\distribfont#1{\texttt{\upshape #1}}

\def\ExpDist{\distribfont{Exponential}\checkarg}
\def\GammaDist{\distribfont{Gamma}\checkarg}

\def\DirDist{\distribfont{Dir}\checkarg}

\def\PRM{\distribfont{PRM}\checkarg}

\def\Stable{\distribfont{Stable}\checkarg}
\def\BESQ{\distribfont{BESQ}\checkarg}
\def\PDIP{\distribfont{PDIP}\checkarg}

\DeclareRobustCommand{\CRP}{\texttt{CRP}\checkarg}
\DeclareRobustCommand{\oCRP}{\texttt{oCRP}\checkarg}
\def\CRPAT{\CRP[\alpha,\theta]}
\def\oCRPAT{\oCRP[\alpha,\theta]}

\makeatletter
\let\save@mathaccent\mathaccent
\newcommand*\if@single[3]{%
  \setbox0\hbox{${\mathaccent"0362{#1}}^H$}%
  \setbox2\hbox{${\mathaccent"0362{\kern0pt#1}}^H$}%
  \ifdim\ht0=\ht2 #3\else #2\fi
  }
\newcommand*\rel@kern[1]{\kern#1\dimexpr\macc@kerna}
\newcommand{\widebar}{}
\DeclareRobustCommand*\widebar[1]{\@ifnextchar^{\wide@bar{#1}{0}}{\wide@bar{#1}{1}}}
\newcommand*\wide@bar[2]{\if@single{#1}{\wide@bar@{#1}{#2}{1}}{\wide@bar@{#1}{#2}{2}}}
\newcommand*\wide@bar@[3]{%
  \begingroup
  \def\mathaccent##1##2{%
    \let\mathaccent\save@mathaccent
    \if#32 \let\macc@nucleus\first@char \fi
    \setbox\z@\hbox{$\macc@style{\macc@nucleus}_{}$}%
    \setbox\tw@\hbox{$\macc@style{\macc@nucleus}{}_{}$}%
    \dimen@\wd\tw@
    \advance\dimen@-\wd\z@
    \divide\dimen@ 3
    \@tempdima\wd\tw@
    \advance\@tempdima-\scriptspace
    \divide\@tempdima 10
    \advance\dimen@-\@tempdima
    \ifdim\dimen@>\z@ \dimen@0pt\fi
    \rel@kern{0.6}\kern-\dimen@
    \if#31
      \overline{\rel@kern{-0.6}\kern\dimen@\macc@nucleus\rel@kern{0.4}\kern\dimen@}%
      \advance\dimen@0.4\dimexpr\macc@kerna
      \let\final@kern#2%
      \ifdim\dimen@<\z@ \let\final@kern1\fi
      \if\final@kern1 \kern-\dimen@\fi
    \else
      \overline{\rel@kern{-0.6}\kern\dimen@#1}%
    \fi
  }%
  \macc@depth\@ne
  \let\math@bgroup\@empty \let\math@egroup\macc@set@skewchar
  \mathsurround\z@ \frozen@everymath{\mathgroup\macc@group\relax}%
  \macc@set@skewchar\relax
  \let\mathaccentV\macc@nested@a
  \if#31
    \macc@nested@a\relax111{#1}%
  \else
    \def\gobble@till@marker##1\endmarker{}%
    \futurelet\first@char\gobble@till@marker#1\endmarker
    \ifcat\noexpand\first@char A\else
      \def\first@char{}%
    \fi
    \macc@nested@a\relax111{\first@char}%
  \fi
  \endgroup
}
\makeatother

\newcommand{\bD}{\mathbb{D}}
\newcommand{\bE}{\mathbb{E}}

\newcommand{\bP}{\mathbb{P}}
\newcommand{\bR}{\mathbb{R}}

\newcommand{\cE}{\mathcal{E}}
\newcommand{\cI}{\mathcal{I}}

\newcommand{\sD}{\mathscr{D}}

\newcommand{\ff}{\mathbf{f}}

\newcommand{\fE}{\mathbf{E}}

\newcommand{\fN}{\mathbf{N}}
\newcommand{\fn}{\mathbf{n}}
\newcommand{\fm}{\mathbf{m}}
\newcommand{\fX}{\mathbf{X}}
\newcommand{\fP}{\mathbf{P}}
\newcommand{\fT}{\mathbf{T}}

\newcommand{\besq}{{\tt BESQ}}

\newcommand{\skewerbar}{\ensuremath{\overline{\normalfont\textsc{skewer}}}}
\newcommand{\clade}{\ensuremath{\normalfont\textsc{clade}}}

\title[Interval partition evolutions related to the Aldous diffusion]{Interval partition evolutions with emigration \\ related to the Aldous diffusion}

\author{N\MakeLowercase{\sc oah} F\MakeLowercase{\sc orman,} S\MakeLowercase{\sc oumik} P\MakeLowercase{\sc al,} D\MakeLowercase{\sc ouglas} R\MakeLowercase{\sc izzolo, and}  M\MakeLowercase{\sc atthias} W\MakeLowercase{\sc inkel}}

\address{\hspace{-0.42cm}N.~Forman\\ Department of Mathematics\\ University of Washington\\ Seattle WA 98195\\ USA\\ Email: noah.forman@gmail.com}             

\address{\hspace{-0.42cm}S.~Pal\\ Department of Mathematics\\ University of Washington\\ Seattle WA 98195\\ USA\\ Email: soumikpal@gmail.com}

\address{\hspace{-0.42cm}D.~Rizzolo\\ 531 Ewing Hall\\ Department of Mathematical Sciences\\ University of Delaware\\ Newark DE 19716\\ USA\\ Email: drizzolo@udel.edu}

\address{\hspace{-0.42cm}M.~Winkel\\ Department of Statistics\\ University of Oxford\\ 24--29 St Giles'\\ Oxford OX1 3LB\\ UK\\ Email: winkel@stats.ox.ac.uk}



\date{\today}

\thanks{This research is partially supported by NSF grants DMS-1204840, DMS-1444084, DMS-1612483, EPSRC grant EP/K029797/1, and the University of Delaware Research Foundation}

\begin{document}

\begin{abstract}
 We construct a stationary Markov process corresponding to the evolution of masses and distances of subtrees along the spine from the root to a branch point in a conjectured stationary, continuum random tree-valued diffusion that was proposed by David Aldous. As a corollary this Markov process induces a recurrent extension, with Dirichlet stationary distribution, of a Wright--Fisher diffusion for which zero is an exit boundary of the coordinate processes. This extends previous work of Pal who argued a Wright--Fisher limit for the three-mass process under the conjectured Aldous diffusion until the disappearance of the branch point. In particular, the construction here yields the first stationary, Markovian projection of the conjectured diffusion. Our construction follows from that of a pair of interval partition-valued diffusions that were previously introduced by the current authors as continuum analogues of down-up chains on ordered Chinese restaurants with parameters $\big(\frac{1}{2}, \frac12\big)$ and $\big(\frac{1}{2},0\big)$. These two diffusions are given by an underlying Crump--Mode--Jagers branching process, respectively with or without immigration. In particular, we adapt the previous construction to build a continuum analogue of a down-up ordered Chinese restaurant process with the unusual parameters $\left(\frac{1}{2},  -\frac{1}{2} \right)$, for which the underlying branching process has emigration.
\end{abstract}

\maketitle

\vspace{-0.5cm}

\section{Introduction}

\noindent The Aldous chain is a Markov chain on the space of rooted binary trees with $n$ labeled leaves. Each transition of the {Aldous chain}, called a down-up move, has two steps. In the down-move a uniform random leaf is deleted and its parent branch point is contracted away. In the up-move a uniform random edge is selected, a branch point is inserted into the middle of the edge, and the leaf is reattached at that point. See Figure \ref{fig:AC_move}. 
David Aldous \cite{Aldous00} studied the analogue of this chain on unrooted trees.

The unique stationary distribution of the Aldous chain on rooted $n$-leaf labeled binary trees is the uniform distribution. Consider an $n$-leaf binary tree as a metric space where each edge has a length of $1/\sqrt{n}$. Then the scaling limit of the sequence of uniform $n$-leaf binary trees, as $n$ tends to infinity, is the Brownian Continuum Random Tree (CRT) \cite{AldousCRT1}. This fundamental limiting random metric space can alternatively be described as being encoded by a Brownian excursion. Aldous \cite{Ald-web,Aldous00} conjectured a ``diffusion on continuum trees'' that can be thought of as a continuum analogue of the Aldous Markov chain. 

In order to understand this difficult and abstract conjectured diffusion, it is natural to search for simpler ``finite-dimensional projections'' that are also Markovian and easier to analyze. Such a projection was suggested by Aldous for the Markov chain and later analyzed by Pal \cite{Pal13}. Specifically, suppose $(T_n(j),\,j\geq 0)$ is the Aldous chain on trees with $n$ leaves. 
Any branch point naturally partitions the tree $T_n(0)$ into three components. As the Aldous chain runs, leaves move among components until the branch point disappears, i.e.\ a component becomes empty.  Until that time, let $m_i(j)$, $i\in \{1,2,3\}$, be the 
proportions of leaves in these components, with $m_3$ referring to the root component. Then
\begin{equation}\label{eq:negativewf}
 \left(\left( m_1^{(n)}(\lfloor n^2u\rfloor), m_2^{(n)}(\lfloor n^2u\rfloor), m_3^{(n)}(\lfloor n^2u\rfloor)\right),\,u\geq 0\right) \overset{d}{\underset{n\to\infty}{\longrightarrow}} ((X_1(u),X_2(u), X_3(u)),\,u\geq 0),
\end{equation}
where the right hand side is a generalized Wright--Fisher diffusion with mutation rate parameters $(-\frac{1}{2},-\frac{1}{2}, \frac{1}{2})$, stopped when one of the first two coordinates vanishes. Since zero is an exit boundary for the coordinates of a Wright--Fisher diffusion that have negative mutation rates, the limiting process does not shed light on how to continue beyond the disappearance of a branch point. 

\begin{figure}[t!]\centering
 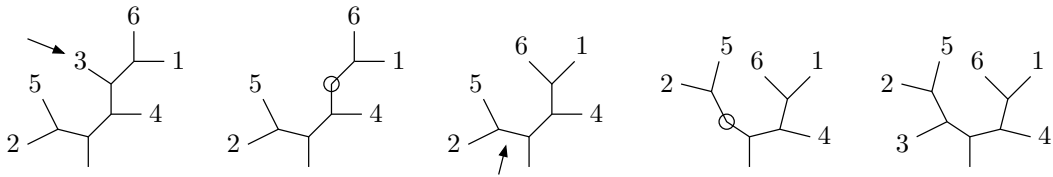
 \caption{From left to right, one Aldous down-up move.\label{fig:AC_move}}
\end{figure}

In our previous work on the discrete Aldous chain \cite{Paper2} we have provided a natural mechanism for selecting a new branch point when the old one disappears, in such a way that the projected mass evolutions remain Markovian. The primary purpose of this paper is to construct a diffusion analogue of this strategy for the case of one branch point. 
To this end, we construct a process on a space of interval partitions as in \cite{Paper1}, which can be projected down onto a three-mass process, and which has the added benefit of describing certain lengths in the conjectured CRT-valued diffusion. 
For the discrete chain, this idea was described in \cite[Appendix A]{Paper2}.

An \emph{interval partition} (IP) in the sense of \cite{Aldousexch,CSP} is a set $\beta$ of disjoint, open subintervals of some interval $[0,M]$, that cover $[0,M]$ up to a Lebesgue null set. We refer to $M\ge0$ as the \emph{mass} of $\beta$ and generally use notation $\|\beta\|$ for $M$. We refer to the subintervals comprising the interval partition as its \emph{blocks}. We denote by $\cI_H$ the set of all interval partitions and by $\cI$ the subset of ``Brownian-like'' interval partitions $\beta$ \emph{with diversity} \cite{Pitman03}, i.e.\ for which the limit 
\begin{equation}\label{eq:diversity}
 \sD_\beta(t)=\sqrt{\pi}\lim_{h\downarrow 0}\sqrt{h}\#\{(a,b)\in\beta\colon |b-a|>h,b\le t\}
\end{equation}
exists for all $t\in[0,\|\beta\|]$. In the context of a rooted Brownian CRT $(\cT,d,\rho,\mu)$ with root $\rho$, a ``2-tree'' with an associated interval partition can be extracted as follows. We independently sample two leaves $\Sigma_1,\Sigma_2\sim\mu$. Consider the geodesic paths $[[\rho,\Sigma_1]],[[\rho,\Sigma_2]]\subset\cT$, their intersection 
$[[\rho,b_{1,2}]]=[[\rho,\Sigma_1]]\cap[[\rho,\Sigma_2]]$, which defines a branch point $b_{1,2}\in\cT$, and the masses $\mu(\cC)$ of all connected components $\cC$ of $\cT\setminus[[\rho,b_{1,2}]]$.
Record $(X_1,X_2,\beta)$, where 
 $X_1$ and $X_2$ are the $\mu$-masses of the connected components containing $\Sigma_1$ and $\Sigma_2$, respectively, 
 and $\beta$ is the \emph{spinal interval partition} of total mass $1-X_1-X_2$ that captures in its interval lengths the $\mu$-masses of the remaining components, in the order of decreasing distance from $\rho$.
It is well-known \cite{Aldous94,PitmWink09} that $(X_1,X_2,1-X_1-X_2)$ has law $\DirDist[\frac{1}{2},\frac{1}{2},\frac{1}{2}]$, and $\beta/(1-X_1-X_2)$ is independent with law $\PDIP\left(\frac{1}{2},\frac{1}{2}\right)$. Here $\PDIP\left(\frac12,\frac12\right)$, which stands for Poisson--Dirichlet Interval Partition, is law of the random interval partition of the unit interval obtained from the excursion intervals of a standard Brownian bridge  \cite{PermPitmYor92, PitmYor92, GnedPitm05}. Furthermore, the total diversity $\sD_\beta(\|\beta\|)$, from \eqref{eq:diversity}, is also equal to $d(\rho,b_{1,2})$, the length of the spine from $\rho$ to
$b_{1,2}$ in $\cT$. Our aim is to construct a Markov process on such 2-tree structures, triplets of an interval partition and two top masses, with total mass one, that is stationary with respect to the law of $(X_1, X_2, \beta)$ described above. 

In \cite{Paper1}, the present authors introduced a related IP-valued process called type-1 evolution, or $\left(\frac{1}{2},0\right)$-IP evolution. We recall its definition in Section \ref{prel}, but for now we recall three properties.
\begin{enumerate}[label=(\roman*), ref=(\roman*)]
 \item It is a path-continuous Hunt process on a space $(\cI,d_\cI)$ with continuously evolving diversities \cite[Theorem 1.4]{Paper1}. The metric $d_{\cI}$ is defined in Definition \ref{def:IP:metric}.
 \item The total mass of the interval partition evolves as a \BESQ[0], the squared-Bessel diffusion of dimension 0 \cite[Theorem 1.5]{Paper1}. In particular, the type-1 evolution is eventually absorbed (we say it dies) at the empty interval partition state, $\emptyset$.\label{item:type1:total_mass}
 \item At Lebesgue almost every time prior to its death, the evolving interval partition has a leftmost block \cite[Proposition 4.30, Lemma 5.1]{Paper1}.\label{item:type1:LMB}
\end{enumerate}


In this paper we find it convenient to represent a type-1 evolution by a pair, $((m^y,\beta^y),\,y\ge0)$, rather than just an evolving interval partition, with $m^y$ denoting the mass of the leftmost block and $\beta^y$ denoting the remaining interval partition, shifted down so that its left end lines up with zero. We take the convention that $m^y = 0$ at the exceptional times $y$ at which there is no leftmost block and after the death of the process. 


A type-2 evolution, or $(\frac{1}{2},-\frac{1}{2})$-IP evolution, is a process that has two leftmost blocks at almost every time. We can represent the two leftmost blocks by just their masses and consider such a process on either of the following state spaces:
\begin{equation}\label{type2spaces}
\begin{split}
 \cJ^\circ &:=\left\{ (a,b,\beta)\in [0,\infty)^2 \times \cI,\;  a+b >0  \right\} \cup \{(0,0,\emptyset)\},\\
 \cI^\circ &:= \{\beta\in\cI\colon \exists a>0\text{ s.t.\ }(0,a)\in\beta\}\cup\{\emptyset\} = \big\{ (0,a)\concat (0,b)\concat\beta\colon (a,b,\beta)\in\cJ^\circ\big\}.
\end{split}
\end{equation} 
Here $\concat$ means a natural concatenation of blocks. Let $d_{\cJ}$ denote the metric on $\cJ^\circ$ given by 
\[
d_{\cJ}\left( (a_1, b_1, \beta_1), (a_2, b_2, \beta_2)  \right)= \abs{a_1-a_2} + \abs{b_1 - b_2} + d_{\cI}(\beta_1, \beta_2). 
\]

Let $\besq_a(-1)$ denote the squared Bessel diffusion of dimension $-1$ starting from $a\ge 0$. This process is killed upon hitting zero. If $\ff\sim \besq_a(-1)$, for some $a\ge 0$, let $\zeta(\ff)$ denote the \emph{lifetime} of the process $\ff$.

\begin{definition}\label{def:type2:v1}
 Let $(a,b,\beta)\in\cJ^\circ$. A type-2 evolution starting from $(a,b,\beta)$ is a $\cJ^\circ$-valued process of the form 
$((m_1^y,m_2^y,\alpha^y),y\ge 0)$, with $(m_1^0,m_2^0,\alpha^0)=(a,b,\beta)$. Its IP-valued variant is a process on state space $\cI^\circ$ that starts from $(0,a)\concat (0,b) \concat \beta$. The distributions of these processes are specified by the following construction.
 
 Let $\left(\fm^{(0)}, \gamma^{(0)} \right)$ be a type-1 evolution starting with the initial condition $( b, \beta)$ and independent of $\ff^{(0)}\sim \besq_a(-1)$. Let $Y_1=\zeta(\ff^{(0)})$. For $0\le y \le Y_1$, define the type-2 evolution as
 \[
 \left( m_1^y, m_2^y, \alpha^y  \right):=\left(\ff^{(0)}(y), \fm^{(0)}(y), \gamma^{(0)}(y)  \right), \quad 0\le y \le Y_1,
 \]
 while its IP-valued variant is the process 
 \[
 (0, \ff^{(0)}(y)) \concat (0, \fm^{(0)}(y))\concat \gamma^{(0)}(y), \quad 0\le y \le Y_1.
 \]
 Now proceed inductively. Suppose, for some $n\ge 1$, these processes have been constructed until time $Y_n$ with $m_1^{Y_n}+m_2^{Y_n}>0$. Conditionally given this history, consider a type-1 evolution $(\fm^{(n)}, \gamma^{(n)})$ with initial condition $(0,\alpha^{Y_n}) = (0,\gamma^{Y_{n-1}}(Y_n-Y_{n-1}))$ that is independent of $\ff^{(n)}$, a $\besq(-1)$ diffusion with initial value $\fm^{(n-1)}(Y_n-Y_{n-1})$. The latter equals $m_2^{Y_n}$ if $n$ is odd or $m_1^{Y_n}$ if $n$ is even.
Set $Y_{n+1}=Y_n+\zeta(\ff^{(n)})$. For $y\in(0,Y_{n+1}-Y_n]$, define
 \begin{align*}
  &\left(m_1^{Y_n+y},m_2^{Y_n+y},\alpha^{Y_n+y}\right):= \begin{dcases}
  (\fm^{(n)}(y),\ff^{(n)}(y),\gamma^{(n)}(y)), &\mbox{if $n$ is odd},\\
  (\ff^{(n)}(y),\fm^{(n)}(y),\gamma^{(n)}(y)), &\mbox{if $n$ is even.}
  \end{dcases}
 \end{align*}
 The IP-valued variant of the process does not switch between the top two masses and is always defined as
 \[
 (0, \ff^{(n)}(y-Y_n)) \concat (0, \fm^{(n)}(y-Y_n)) \concat \gamma^{(n)}(y-Y_n), \quad Y_n < y \le Y_{n+1}. 
 \]
 If, for some $n\ge 1$, $m_1^{Y_n}+m_2^{Y_n}=0$, set $(m_1^y,m_2^y,\alpha^y):=(0,0,\emptyset)$ for all $y>Y_n$ and $Y_{n+1} := \infty$.
\end{definition}

The difference between the two variants of type-2 evolutions is that in one the top two masses are labeled by $1$ and $2$ which jump as a mass hits zero, while in the other the top two masses are unlabeled and simply drop out of the interval partition as empty blocks when they hit zero. The former allows a stationary construction, while the latter is necessary for continuity.

\begin{theorem}\label{thm:diffusion}
 Type-2 evolutions are Borel right Markov processes on $(\cJ^\circ,d_{\cJ})$. IP-valued type-2 evolutions are path-continuous Hunt processes on $(\cI^\circ,d_{\cI})$.
\end{theorem}

\begin{theorem}\label{thm:total_mass}
 For a type-2 evolution $((m_1^y,m_2^y,\alpha^y),\,y\ge0)$, the total mass process $(m_1^y+m_2^y+\|\alpha^y\|,\,y\ge0)$ is a \BESQ[-1] process.
\end{theorem}

Since a $\besq(-1)$ process eventually gets killed at zero, a type-2 evolution is not stationary. However, we obtain a stationary variant 
by modifying the process in two ways: de-Poissonization and resampling. De-Poissonization means that we normalize so that the total mass remains constant at one, and then we apply a time-change. De-Poissonization was used in \cite{Paper1} to obtain a stationary variant of type-1 evolution and has previously been applied in related settings in \cite{Pal11,Pal13,WarrYor98}. 
Resampling is a new idea in this context. 
We will see that the type-2 evolution eventually \emph{degenerates}, entering a state of only having a single block: either $m_1^y = \|\alpha^y\| = 0 < m_2^y$ or $m_2^y = \|\alpha^y\| = 0 < m_1^y$. 
At that time we will have the process jump into an independent state sampled from the law described above as a 2-tree projection of the Brownian CRT; see Definition \ref{def:resampling}. The state spaces of the resampling de-Poissonized processes are 
\begin{equation}\label{eq:IP_spaces}
\begin{split}
 \cJ^*_1 &:= \{(a,b,\beta)\in\cJ^\circ\colon a+b+\|\beta\| = 1;\,a,b,\|\beta\|<1\},\\
 \cI^*_1 &:= \big\{\beta\in\cI^\circ\colon \|\beta\| = 1,\,\beta\neq\{(0,1)\}\big\} = \big\{ (0,a)\concat (0,b)\concat\beta\colon (a,b,\beta)\in\cJ^*_1\big\}.
\end{split}
\end{equation}

\begin{theorem}\label{thm:stationary}
The resampling, de-Poissonized type-2 evolution (which we also call a \emph{2-tree evolution}) is a Borel right Markov process on $(\cJ_1^*,d_{\cJ})$. The IP-valued variant is a Borel right Markov process on $(\cI^*_1,d_{\cI})$ and is path-continuous except on a discrete set of resampling times.
 
 Consider $(A_1,A_2,A_3)\sim \DirDist[\frac12,\frac12,\frac12]$ and an independent interval partition $\widebar\beta\sim\PDIP\left(\frac12,\frac12\right)$. The law of $(A_1,A_2,A_3\widebar\beta)$ is the unique stationary distribution for the 2-tree evolution on $\cJ_1^*$.
\end{theorem}



Consider the map $\pi^\bullet_2$ on $\cJ_1^*$ given by $(a,b,\beta)\mapsto (a,b,\|\beta\|)$. The range of this map is the set $\Delta:=\{ (p_1, p_2, p_3) \in [0,1)^3, \sum_{i=1}^3 p_i=1   \}$. Also consider the stochastic kernel $\Lambda$ from $\Delta$ to $\cJ_1^*$ that maps $(p_1,p_2,p_3)$ to the law of $(p_1,p_2,p_3\widebar\beta)$, where $\widebar\beta\sim\PDIP[\frac12,\frac12]$. Given $(p_1,p_2,p_3)\in\Delta$, run a resampling, de-Poissonized type-2 evolution $(T^u,\,u\ge0)$ with initial condition $\Lambda(p_1,p_2,p_3)$. The \emph{induced $3$-mass process} is then $(X_1(u),X_2(u),X_3(u)) := \pi^\bullet_2(T^u)$, $u\ge0$.

\begin{theorem}\label{thm:wright-fisher}
 The induced $3$-mass process is a recurrent Markovian extension of the Wright--Fisher$(-\frac{1}{2}, -\frac{1}{2}, \frac{1}{2})$ diffusion, described in \eqref{eq:negativewf}, in the following sense. Let $U$ be the first time $u$ when either $X_1(u)=0$ or $X_2(u)=0$. Then the process killed at $U$ is the killed Wright--Fisher diffusion. The $3$-mass process is intertwined with the resampling, de-Poissonized type-2 evolution, and it converges to its unique stationary law $\DirDist\left( \frac{1}{2}, \frac{1}{2}, \frac{1}{2} \right)$.
\end{theorem}

Notice that the 3-mass process jumps back into the interior of the simplex immediately after either of the first two coordinates vanish. This extension of the generalized Wright--Fisher diffusion is natural from the perspective of the Aldous chain and is the continuum analogue of the construction in \cite{Paper2}.

\subsection{From the Aldous chain to $\left(\frac12,-\frac12\right)$-Chinese restaurants}

In this and the next subsection, we informally discuss a discrete counterpart to the 2-tree evolutions, giving a preliminary overview of the construction of type-1 evolution and its connection to 2-trees.

Consider the following decomposition of a rooted binary tree, in analogy
with the decomposition of the BCRT described below equation
\eqref{eq:diversity}. Select a branch point. We decompose the tree into
two \emph{top subtrees} above the branch point and a
sequence of \emph{spinal subtrees} branching off of the path, called the
\emph{spine}, from the branch point to the root. We represent partial information about the tree via the
masses (leaf counts) in these subtrees: a pair of top masses 
$(m_1,m_2)$, followed by a finite sequence of spinal masses 
$(b_1,\ldots,b_{k-2})$, ordered by decreasing distance from the root. We 
call this representation a \emph{discrete 2-tree}.

\begin{figure}\centering
 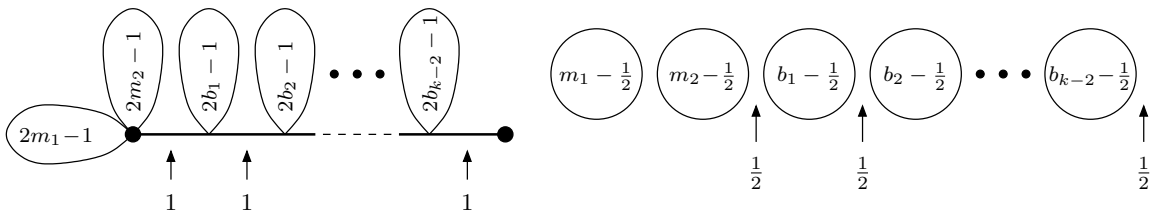
 \caption{Up-move weights for 2-tree projection of Aldous chain correspond to seating rule for \oCRP[\frac12,-\frac12].\label{fig:2tree_CRP}}
\end{figure}

The Aldous down-up moves act on the discrete 2-tree as follows. In the
down-move, we make a size-biased pick among the masses and reduce that mass
by one. If the mass is reduced to zero, it is removed from the list. For
up-moves, we choose a mass $m$ with probability proportional to $2m - 1$,
or choose any edge along the spine with probability proportional to 1; see
Figure \ref{fig:2tree_CRP}. If a mass is chosen, it is incremented by 1; if
a spinal edge is chosen, a `1' is inserted into the sequence of spinal
masses at that point, representing the appearance of a new spinal subtree.
We adopt the rule that if, after a down-move, one of the two top subtree
masses is reduced to zero, then the first mass along the spine replaces it
as a new top mass. 
A generalization of this projected Aldous chain is studied in \cite[Appendix A]{Paper0}.

The up-move weights of Figure \ref{fig:2tree_CRP} are very close to the seating rule for an ordered Chinese restaurant process (\oCRP) \cite{PitmWink09,Paper1}. The \oCRPAT\ begins with a single customer sitting alone at a table. New customers enter one by one. Upon entering, the $n+1^{\text{st}}$ customer chooses to join a table that already has $m$ customers with probability $(m-\alpha)/(n+\theta)$; sits alone at a new table inserted at the far left end of the restaurant with probability $\theta/(n+\theta)$; or sits alone at a new table, inserted to the right of any particular table already present, with weight $\alpha/(n+\theta)$, so that the total probability to sit alone is $(k\alpha+\theta)/(n+\theta)$, where $k$ is the number of tables already present. If we ignore the left-to-right order of these tables, then this is the well-known (unordered) \CRPAT\ due to Dubins and Pitman \cite[\S 3.2]{CSP}. The distribution of an \oCRP\ after $n$ customers have arrived is the discrete analogue to the \PDIP.

 If we take $(\alpha,\theta) = \big(\frac12,0\big)$, then this seating rule differs from the up-move probabilities in Figure \ref{fig:2tree_CRP} only in that, in the \oCRP, a new table can be introduced between the two leftmost tables, whereas in the 2-tree no new mass can be inserted in between the two leftmost masses, representing the two top subtrees, which are not separated by an edge but only by a branch point. We refer to the probabilities in Figure \ref{fig:2tree_CRP} as the seating rule for the \oCRP[\frac12,-\frac12], as, under this rule, if there are a total of $k$ masses (2 top masses and $k-2$ spinal masses) then the probability for insertion of a new mass `1' is $(k-1)/(2n-1) = \big(k\frac12 - \frac12\big)/\big(n-\frac12\big)$.

This is outside of the usual parameter range considered for the \CRP. Indeed, if we start a \CRP[\frac12,-\frac12] with a single customer, as described above, then all subsequent customers will be forced to join the first at a single table, as the probability to sit alone will be zero. However, if we start with two customers sitting separately, then the \oCRP[\frac12,-\frac12] seating rule produces a non-trivial configuration distributed as the 2-tree projection of a uniform random rooted binary tree with labeled leaves. We remark that Poisson--Dirichlet distributions with ``forbidden'' parameters have been considered before, e.g.\ in the context of $\sigma$-finite dislocation measures of fragmentation processes and related discrete splitting probabilities \cite{Mie-03,Bas-06,MPW,HPW}.

In the setting of the Chinese restaurant analogy Aldous's down-up moves become \emph{re-seating}: a uniform random customer leaves their seat; their table is removed if empty; and they choose a new seat according to the seating rule, as if entering for the first time.

\subsection{Discrete scaffolding, spindles, and skewer}\label{sec:intro:scaff_spind}

We simplify matters by \emph{Poissonizing} the Aldous chain. In the Poissonized Aldous chain, each leaf is removed in a down-move after an independent exponential time with rate 1. That leaf is not immediately re-inserted into the tree. Rather, up-moves occur at each edge after an exponential time with rate $\frac12$ (since there are roughly twice as many edges as leaves). This allows the total number of leaves in the tree to fluctuate, but it results in a process in which disjoint subtrees evolve independently. 
Scaling limits of some statistics of this Poissonized chain have been rigorously connected to the Aldous chain via de-Poissonization in \cite{Pal13}.

When we project to the discrete 2-tree as before, the sequence of subtree masses evolves as a \emph{Poissonized down-up} \oCRP[\frac12,-\frac12]. In this process, each table population $m$ decreases by 1 with rate $m$, or increases by 1 with rate $m-\frac12$. To the right of any table except for the leftmost (i.e.\ not between the two leftmost), a new table of population 1 appears with rate $\frac12$.

Due to Poissonization, the table populations evolve independently of each other. Each one is a birth-and-death chain, having deaths with rate $m$ and births with rate $m-\frac12$ when the population is $m$, until absorption at population 0. Let $\mu$ denote the distribution of the lifetime of this birth-and-death chain, started from population 1.

This Poissonized down-up \oCRP\ admits a surprising representation, which was introduced in \cite{Paper1} to describe continuum analogues of the Poissonized down-up \oCRP[\frac12,0] and \oCRP[\frac12,\frac12]. We discuss here the $\big(\frac12,0\big)$ case and then the new extension to $\big(\frac12,-\frac12\big)$.

We think of the tables that appear and vanish in this evolving \oCRP\ as members of a family: when a new table is born, the table immediately to its left at that time is its parent. The number of tables is then evolving over time as a Crump--Mode--Jagers (CMJ) branching process \cite{Jagers69}. The genealogy among these tables, and their lifetimes, can be represented in a splitting tree \cite{GeigKers97}. For our purposes, this can be formalized as a rooted plane tree with edge lengths.

Figure \ref{fig:splitting} depicts the construction of a splitting tree representation of the Poissonized down-up \oCRP[\frac12,0], started with a single customer.
\begin{enumerate}[label=(\arabic*), ref=(\arabic*)]
 \item Draw a line of random length, sampled from $\mu$ (the lifetime distribution of a table started with population 1); this represents the first table.
 \item Now, mark that line with Poisson points along its length, with rate $\frac12$.
 \item At each marked point, attach a new ``child'' line, branching off to the right from its parent, with length independently sampled from $\mu$. Each such line represents a table born, at some point, immediately to the right of the first table.
 \item Repeat steps (2), (3), and (4) on each of the newly drawn lines, if any.
\end{enumerate}
It is not immediately obvious, but this procedure almost surely terminates for this choice of $\mu$.

\begin{figure}
 \centering
 \scalebox{.8}{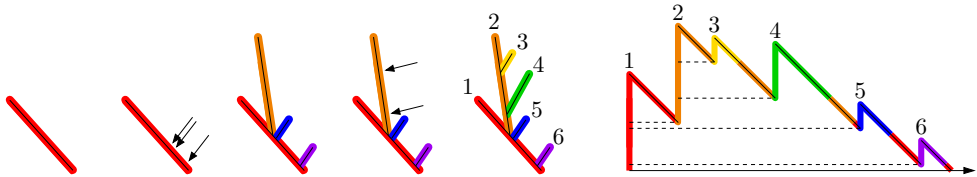}
 \caption{Iterative construction of splitting tree representation of tables in Poissonized down-up \oCRP[\frac12,0], started with one customer, and JCCP.\label{fig:splitting}}
\end{figure}

This tree can be represented by a \emph{jumping chronological contour process} (JCCP) \cite{Geiger95,GeigKers97}, shown in Figure \ref{fig:splitting}. Imagine a flea traveling around the splitting tree. It begins to the left of the root, and immediate jumps up to the top of the leftmost branch, representing the first table. It then slides down the right hand side of that branch at unit speed until its path is blocked by a branch sticking out to the right. When that happens, it jumps to the top of the new branch, and carries on in the same manner, until it finally reaches the root. The JCCP records the distance from the flea to the root, as a function of time.

The tables that arise in the evolving \oCRP\ are in bijective correspondence with the jumps of the JCCP, with the levels of the bottom and top of each jump equaling the birth and death times of the corresponding table. The genealogy among tables can be recovered by looking to the bottom of each jump (a child), and drawing a horizontal line to the left from that point, seeing where it crosses another jump (its parent). 

JCCP representations of splitting trees like ours are L\'evy processes of positive jumps and negative drift \cite{Lambert10}. Our particular JCCP has drift $-1$ and L\'evy measure $\frac12\mu$. \emph{Levels} in the JCCP correspond to \emph{times} in the evolving \oCRP. On the other hand, times in the JCCP have no simple meaning in the \oCRP, and serve mainly to record the left-to-right order of tables.

What is missing from this JCCP picture is the evolving table populations. Recall that each table population evolves as a birth-and-death chain with lifetime distribution $\mu$. This is also the law of jump heights in our JCCP. We incorporate both the genealogy among tables and the evolving table populations into a single formal object by marking each jump with such a birth-and-death chain, with lifetime equal to the height of the jump.

We depict this object by representing each birth-and-death chain as a laterally symmetric ``spindle'' shape, beginning at the bottom of the jump and evolving towards its top, with width at each level describing the value of the chain at the corresponding time. In the context of this construction, we refer to the JCCP as \emph{scaffolding} and the markings as \emph{spindles}. See Figure \ref{fig:scaff_skewer}.

\begin{figure}[t]\centering
  \includegraphics[scale=.7]{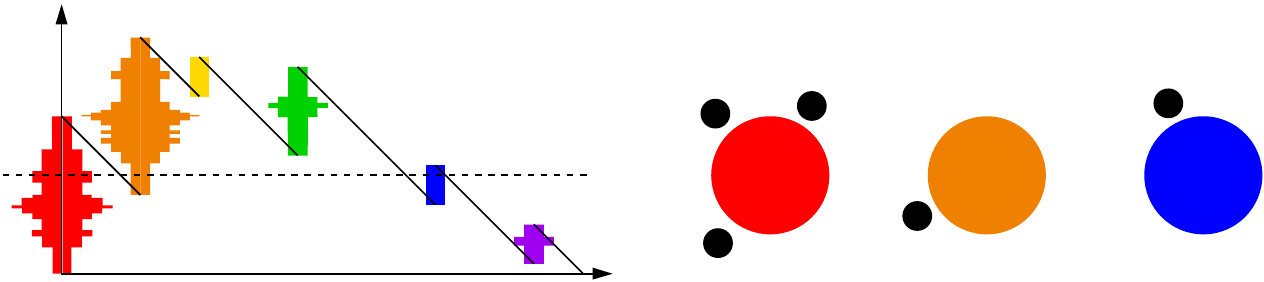}
  \caption{Scaffolding with spindles and \oCRP\ arrangement described by the skewer.}
  \label{fig:scaff_skewer}
\end{figure}

Then, to recover the Poissonized down-up \oCRP[\frac12,0] from the scaffolding and spindles representation, we apply a \emph{skewer map}: for any $y\ge0$, we draw a horizontal line through the picture at that level, and look at the cross-sections of spindles pierced by the line. The widths of these cross-sections represent populations of tables, and their left-to-right order corresponds to that in the \oCRP. If we slide this horizontal line up continuously, then the cross-sections gradually change in width, with some dying out as the horizontal line passes the top of a jump, and new ones appearing as it reaches the bottom of a jump.


In scaling limits, the scaffolding converges to a \Stable[\frac32] L\'evy process, and the law of the birth-and-death chain spindles converges to a $\sigma$-finite excursion measure associated with squared Bessel processes with parameter $-1$, abbreviated as \BESQ[-1], studied in \cite{PitmYor82}. This motivated the construction, in \cite{Paper1}, of \emph{type-1 evolutions} by applying the skewer map to \Stable[\frac32] scaffolding marked by \BESQ[-1] excursion spindles. A simulation of Poissonized down-up \oCRP\ approximating the type-1 evolution is shown in Figure \ref{fig:scaff_spind_skewer}.

\begin{figure}
 \centering
 \begin{minipage}{.49\textwidth}
  \centering
  (a) \raisebox{-.5\height}{\includegraphics[width = 1.2in,height=.72in]{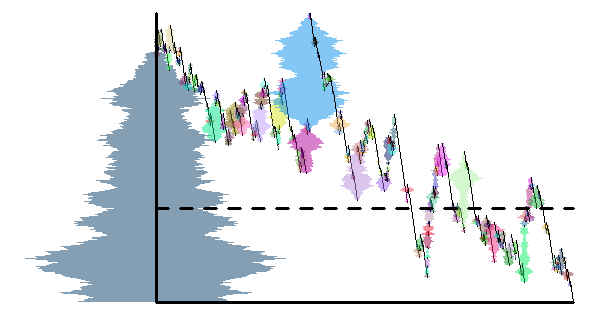}} \ \ 
  (b) \raisebox{-.5\height}{\includegraphics[width = 1.2in,height=.72in]{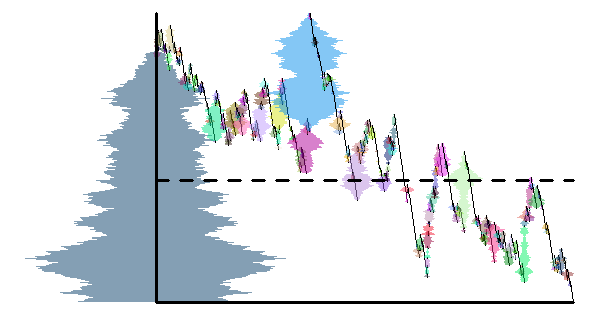}}\\[-.1in]
  \hphantom{(a)} \hspace{.2in}\includegraphics[width = 0.55in,height=.05in]{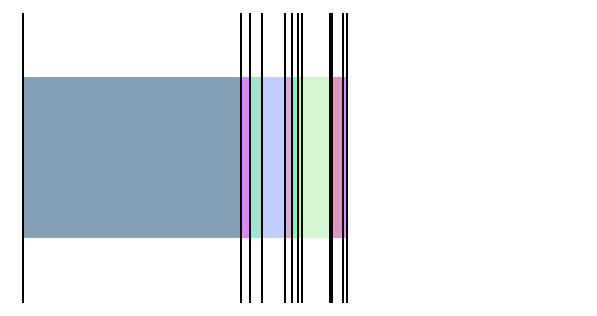}\hspace{.5in} \quad 
  \hphantom{(b)} \hspace{.1in}\includegraphics[width = 0.55in,height=.05in]{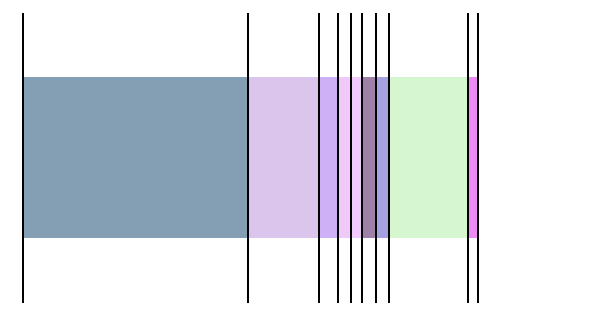}\hspace{.5in}\\[.1in]
  (c) \raisebox{-.5\height}{\includegraphics[width = 1.2in,height=.72in]{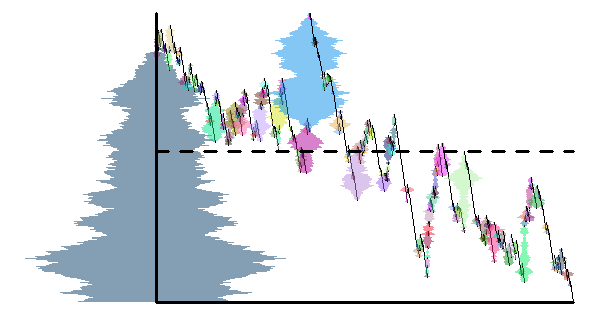}} \ \  
  (d) \raisebox{-.5\height}{\includegraphics[width = 1.2in,height=.72in]{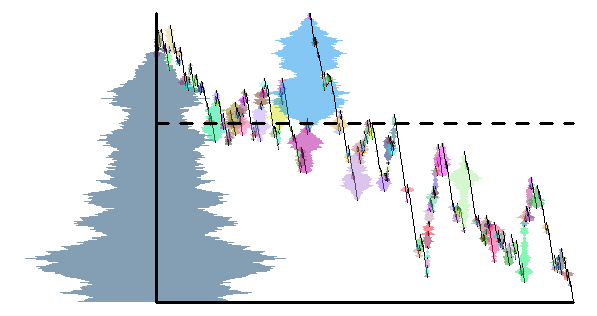}}\\[-.1in]
  \hphantom{(c)} \hspace{.2in}\includegraphics[width = 0.55in,height=.05in]{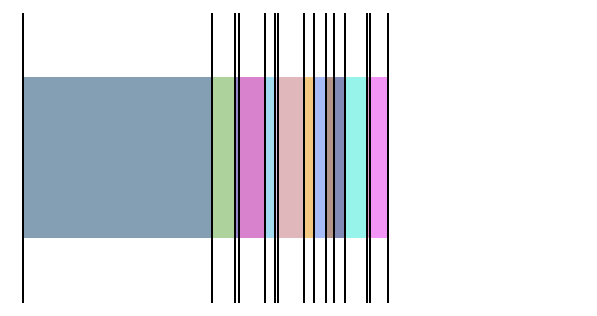}\hspace{.5in} \quad 
  \hphantom{(d)} \hspace{.1in}\includegraphics[width = 0.55in,height=.05in]{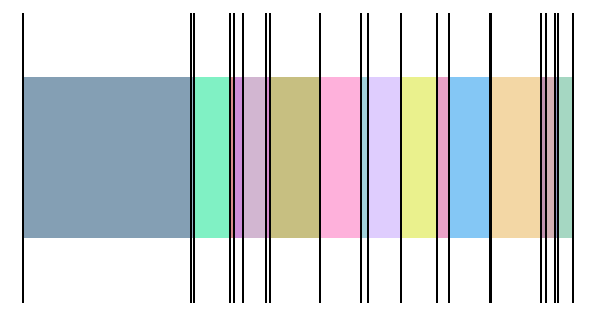}\hspace{.5in}\\
 \end{minipage}
 \begin{minipage}{.49\textwidth}
  \centering
  (e)\raisebox{-.5\height}{\includegraphics[width = 2.7in,height=1.62in]{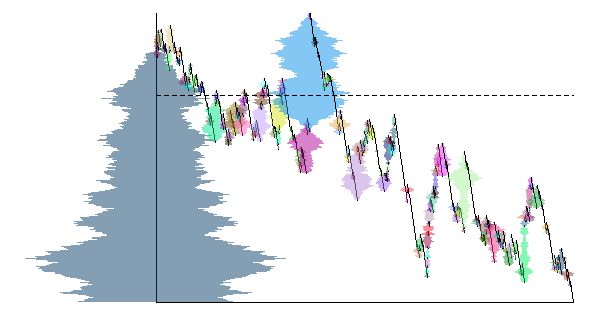}}\\[-.1in]
  \hphantom{(e)} \hspace{.72in}\includegraphics[width = 1.2in,height=.1in]{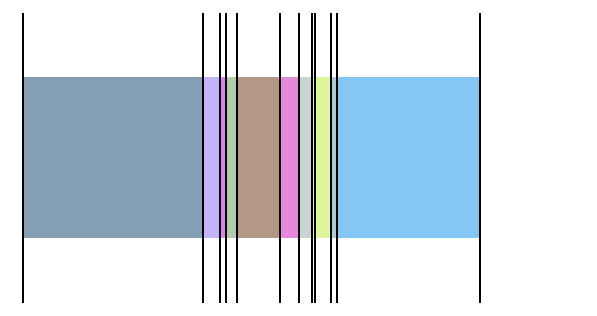}\hspace{.7in}
 \end{minipage}
 \caption{Type-1 interval partition evolution, constructed from scaffolding (slanted black lines), spindles (laterally symmetric colored shapes), and the skewer map. Simulation from \cite{WXMLCRP}.\label{fig:scaff_spind_skewer}}
\end{figure}

\subsection{Poissonized discrete 2-tree evolution}\label{sec:discrete_clocking}

The difference between the Poissonized down-up \oCRP[\frac12,0] and the Poissonized down-up \oCRP[\frac12,-\frac12], corresponding to the 2-tree evolution, is that in the latter process no new tables can be born in between the two leftmost tables. 
This corresponds to no new subtrees appearing between the two top subtrees in the 2-tree evolution. 
At different times, different tables may become the leftmost. Such a table may have had children prior to becoming leftmost, but subsequently, it ceases to do so.

To construct such a process via scaffolding and spindles, we begin with a scaffolding-and-spindles construction of the Poissonized down-up \oCRP[\frac12,0] with two initial tables. We then find all instances in which a child was born to a parent spindle at a level at which the parent was the overall leftmost spindle, and we delete all such children and their offspring; see Figure \ref{fig:d_clk_1}. We refer to these leftmost spindles as \emph{clock spindles} and the transformation of deleting their descendants as \emph{deletion clocking} or \emph{emigration}.  We think of these spindles, which correspond to the $\ff^{(n)}$s in Definition \ref{def:type2:v1}, as timers. When they reach zero, at $Y_{n+1}$, we pass to the $n+1^{\text{st}}$ stage in the construction. 
We formalize deletion clocking in the continuum analogue in Definition \ref{deftype2}.

\begin{figure}
 \centering
 \includegraphics[scale=.75]{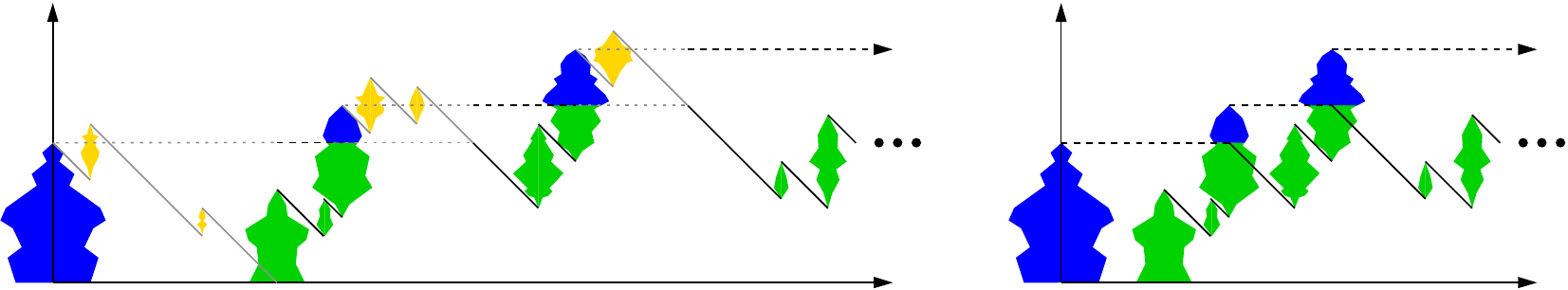}
 \caption{Deletion-clocking, to transform scaffolding-and-spindles representation of Poissonized down-up \oCRP[\frac12,0] to representation of \oCRP[\frac12,-\frac12].\label{fig:d_clk_1}}
\end{figure}

\section{Type-1 evolutions: preliminaries and representation as pairs}\label{prel}


\noindent In this section, we recall from \cite{Paper1} the ``scaffolding, spindles, and skewer'' constructions of type-1 and type-0 interval partition evolutions. We also recall the main results of \cite{Paper1} and record some further consequences. Here, for brevity, we will construct type-0 and type-1 evolutions on a probability space; in \cite{Paper1}, all of this work is carried out in terms of probability distributions and filtrations on a canonical space of counting measures.


\subsection{Preliminaries on type-1 and type-0 evolutions}

Recall the definition of the set $\cI$ of interval partitions with diversity from the introduction. This space can be metrized by the Hausdorff metric between complements of interval partitions, but we prefer a stronger metric $d_{\cI}$ that accounts for diversity. We formally define this metric later, in Definition \ref{def:IP:metric}. We define two probability distributions on this space, as in \cite{Paper1}: \PDIP[\frac12,0], which is the reversal of the interval partition formed by excursion intervals of Brownian motion on time $[0,1]$, including the incomplete final excursion; and \PDIP[\frac12,\frac12], which is the partition formed by excursion intervals of Brownian bridge. The names are in recognition of the facts that the ranked excursion lengths are Poisson--Dirichlet distributed with respective parameters, ${\tt PD}(\frac{1}{2},\frac{1}{2})$ and ${\tt PD}(\frac{1}{2},0)$, see e.g.\ \cite{CSP}. 

We denote by $\cE$ the space of c\`adl\`ag excursions away from zero. In the context of the following construction, we refer to continuous excursions $f\in\cE$ as \emph{spindles} and to excursions with a c\`adl\`ag jump at 0 and/or at their time of absorption $\zeta(f)$ as \emph{cut-off spindles}.

Recall $\besq(-1)$ has an exit boundary at zero \cite{GoinYor03}. Despite this, methods of \cite{PitmYor82} allow the construction of a $\sigma$-finite excursion measure $\mBxc$ associated with $\besq(-1)$. We choose the normalization constant so that
\begin{equation}
 \mBxc(\zeta\in dy)=\frac{3}{2\pi\sqrt{2}}y^{-5/2}dy.
\end{equation}

Let $\fN$ be a Poisson random measure on $[0,\infty)\times\cE$ with intensity measure ${\rm Leb}\times\nu_{\rm BES}$. For excursions $f$ arising in this point process, we take their lifetimes $\zeta(f)$ to be jump heights for a L\'evy process constructed from jumps and compensation:
\begin{equation}
 \fX(t) := \xi_{\fN}(t) := \lim_{z\downto 0}\left(\int_{[0,t]\times\{g\in\Exc\colon\zeta(g) > z\}}\life(f)d\fN(s,f) - \frac{3tz^{-1/2}}{\pi\sqrt{2}}\right),\quad t\ge0.\label{eq:JCCP_def}
\end{equation}
We abbreviate $\xi(\fN) := (\xi_{\fN}(t),\,t\ge 0)$. This is a spectrally positive \Stable[\frac32]\ L\'evy process, called \emph{scaffolding}. The \emph{aggregate mass process} and \emph{skewer} of $\fN$ at level $y\in\bR$ are
\begin{equation}
 \begin{array}{c}
 \displaystyle M^y_{\fN}(t) := \int_{(-\infty,t]\times\Exc} \max\Big\{\ f\big((y-\xi_{\fN}(u-))-\big),\ f\big(y-\xi_{\fN}(u-)\big)\ \Big\}d\fN(u,f),\quad t\ge 0,\\[14pt]
 \displaystyle \skewer(y,\fN) := \left\{\left(M^y_{\fN}(t-),M^y_{\fN}(t)\right)\colon t\in\bR,\ M^y_{\fN}(t-) < M^y_{\fN}(t)\right\}.
 \end{array}
 \label{eq:skewer_def}
\end{equation}
We abbreviate $\skewerbar(\fN) := \skewerbar(\fN,\xi(\fN)) := \big(\skewer(y,\fN),\ y\geq 0\big)$. A simulation of this construction is depicted in Figure \ref{fig:scaff_spind_skewer}. Note that \eqref{eq:skewer_def} is set up to allow counting measures $N$ on $\BR\times\Exc$, and not just on $[0,\infty)\times\Exc$, in anticipation of our construction of type-0 point measures. 

We denote by $\cN$ a space of point measures on $[0,\infty)\times\cE$, supported on bounded time intervals $[0,T]\times\cE$, in which $\skewerbar(N)$ is well-defined and $d_\cI$-continuous for each $N\in\cN$; this was denoted by $\cN_{\rm fin}^{\rm sp,*}$ in \cite[Definitions 3.16, 4.16]{Paper1}. We also define $\cev\cN$ analogously, but for measures supported on $(-\infty,0)\times\cE$, not necessarily on a bounded time interval. Around \eqref{eq:JCCP_def_type0}, we will introduce one such random element of $\cev\cN$, a random measure supported on unbounded time, for which the skewer at each level remains bounded and evolves continuously.

Now, consider $\ff\sim\besq_{x}(-1)$ a \BESQ[-1] started from $x >0$. A \emph{clade of initial mass $x$} is then a random counting measure $\fn\in\cN$, distributed as
\begin{equation}
 \clade(\ff,\fN) := \delta(0,\ff)+\Restrict{\fN}{(0,T_{-\zeta(\ff)}(\fN)]\times\cE},\quad \text{where} \quad T_{-y}(\fN) := \inf\{t\ge 0\colon\xi_{\fN}(t)=-y\}.
\end{equation}

In the following clade construction and elsewhere, the notion of ``concatenation,'' denoted by $\star$, is in the sense of excursion theory: concatenating a sequence of excursions means running one after the other. This easily generalizes to totally ordered collections with summable excursion lengths. Concatenation of excursions induces a notion of concatenation of point measures of jumps and hence a notion for point measures of (jumps marked by) spindles. See \cite{Paper1} for details.

For any ``initial'' interval partition $\beta$, we denote by $\fP_\beta^1$ the law of a random \emph{type-1 point measure} $\fN_{\beta}\in\cN$ obtained by concatenating independent clades with initial masses equal to the lengths of the intervals in $\beta$. 

Now, consider the point measure $\cev\fN$ on $(-\infty,0)\!\times\!\cE$ formed by concatenating a sequence of independent copies of $\restrict{\fN}{(0,T_{-1}(\fN)]\times\cE}$, with each copy being concatenated to the left of the previous copies. 
We slightly modify \eqref{eq:JCCP_def} in this setting:
\begin{equation}
 \cev\fX(t) := \xi_{\cev\fN}(t) := \lim_{z\downto 0}\left(-\int_{(t,0)\times\{g\in\Exc\colon\zeta(g) > z\}}\life(f)d\cev\fN(s,f) + \frac{3|t|z^{-1/2}}{\pi\sqrt{2}}\right),\quad t\le0.\label{eq:JCCP_def_type0}
\end{equation}
As before, $\xi\big(\cev\fN\big) := (\xi_{\cev\fN}(t),\,t\le 0)$. Informally, this is a spectrally positive \Stable[\frac32]\ first-passage descent from $\infty$ down to 0, arranged to arrive at 0 at time zero. This construction was discussed in \cite[Remark 5.15]{Paper1}.

If $\fN_{\beta}$ is as above and independent of $\cev\fN$, then $\cev\fN + \fN_{\beta}$ is a \emph{type-0 point measure} with initial state $\beta$. In the sequel, we find it convenient to represent this as a \emph{type-0 data pair} $(\cev\fN,\fN_{\beta})\in {\cev\cN}\times\cN$. 
We denote the law of this pair by $\fP_\beta^0$. We take the convention that $\xi(\cev\fN + \fN_{\beta})$ equals $\xi(\cev\fN)$ on $(-\infty,0)$ and equals $\xi(\fN_{\beta})$ on $[0,\infty)$.

We showed in \cite[Theorem 1.4]{Paper1} that the associated \emph{type-1} and \emph{type-0 evolutions}, respectively  
$\skewerbar(\fN_\beta)$ and $\skewerbar(\cev\fN+\fN_\beta)$, are path-continuous strong Markov processes on $(\cI,d_\cI)$. 

We define the \emph{shifted restrictions} of a point measure, denoted by $\shiftrestrict{N}{[a,b]\times\Exc}$ and $\restrictshift{N}{[a,b]\times\Exc}$ to be point measures obtained by first restricting support to the indicated region, and then shifting the resulting point measure to be supported on $[0,b-a]\times\Exc$ or $[a-b,0]\times\Exc$, respectively. We denote by $\cev{T}_y\big(\cev{\fN}\big)=\inf\{t\le 0\colon\xi_{\cev{\fN}}(t)=y\}$, $y\ge 0$, the pre-0 downward first passage times of $\xi(\cev{\fN})$. 
Just as we used $\fN$ on $(0,T_{-\zeta(\ff)}(\fN)]$ to define $\clade(\ff,\fN)$, we can use $\cev{\fN}$ on the time interval $[\cev{T}_{\zeta(\ff)}(\cev{\fN}),0)$ to extend a spindle $\ff$ to a clade $\clade(\ff,\cev{\fN}):=\delta(0,\ff)+\ShiftRestrict{\cev{\fN}}{[T_{\zeta(\ff)}(\cev{\fN}),0)\times\cE}$.

\begin{lemma}\label{type1fromtype0}
 Let $\Psi := (\ff,\cev{\fN},\fN_\gamma)\sim\fP_{x,\gamma}^1:=\besq_x(-1)\otimes\fP_\gamma^0$ for some $x\in[0,\infty)$ and $\gamma\in\cI$. Then the measure $\fN_* = \clade(\ff,\cev{\fN})\star\fN_\beta\sim\fP_{(0,x)\star\gamma}^1$ is a type-1 point measure. 
\end{lemma}

In light of this lemma, we refer to $\Psi$ as a \emph{type-1 data triple}. This construction may seem superfluous: why include $\fN_{\gamma}$ as a member of a triple of objects, just to set up another point measure $\fN_*$ of the same type? However, this sets up a parallel with type-0 data pairs leading to the definition of type-2 data quadruples. This parallel will be useful in forthcoming work on the Aldous diffusion \cite{Paper4,Paper5} involving all three processes.

To exhibit the Markovian nature of the skewer processes and underlying clade constructions, 
for $y\ge 0$ we decompose $\fN_*$ into a point process $\fN_*^y$ of spindles or cut-off spindles above level $y$ and a point process $\fN_*^{\le y}$ of spindles or cut-off spindles below level $y$, as in Figure \ref{fig:biclade_decomp}.

\begin{figure}
 \centering
 \includegraphics{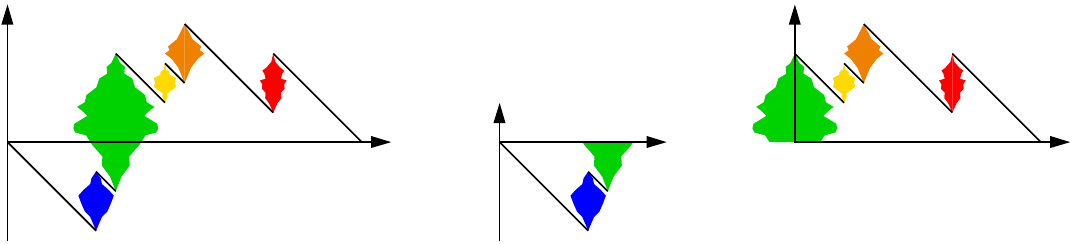}
 \caption{Decomposition of (discrete approximation of) a typical excursion of \Stable[\frac32]\ scaffolding with spindles into an initial component below the level and a subsequent component above. The latter is a \emph{clade}.\label{fig:biclade_decomp}}
\end{figure}

\begin{figure}
 \centering
 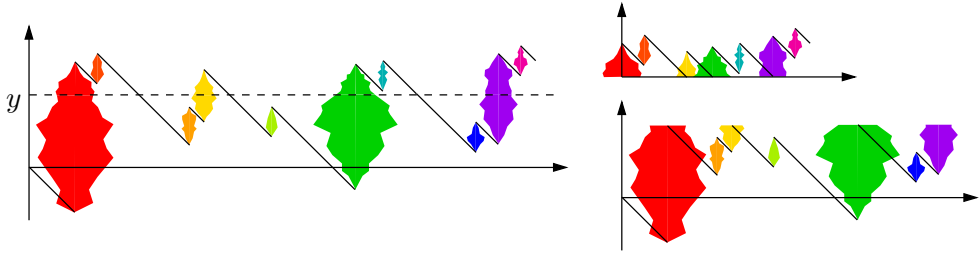
 \caption{Illustration \cite[Figure 3.3]{Paper1} of spindles cut at level $y$. Left: $N$. Right: $N^y$ and $N^{\le y}$.\label{fig:cutoff}}
\end{figure}

More formally, for a type-0 data pair $\Psi = (\cev\fN,\fN_{\beta})$, define $\cev{\fN}^y_{\Psi}:=\restrictshift{\cev{\fN}}{(-\infty,T_y(\cev{\fN}))\times\cE}$. This captures only spindles above level $y$. Beyond time $T_y(\cev\fN)$, each spindle $f$ that crosses level $y$ corresponds to the unique jump across $y$ in an excursion of $\xi(\cev\fN+\fN_{\beta})$ about level $y$. The cut-off spindle $\hat{f}^y$ together with the spindles following this jump in the excursion forms a clade; see Figure \ref{fig:biclade_decomp}. We denote the concatenation of these subsequent clades by $\fN^y_{\Psi}$. We also concatenate the point measures of the remaining spindles and cut-off spindles $\check f$ -- the initial components, below level $y$, of each excursion of scaffolding, as in Figure \ref{fig:biclade_decomp} -- into a point measure $\fN^{\le y}_{\Psi}$, and denote by $(\cF^y,y\ge 0)$ the filtration generated by $(\fN^{\le y}_{\Psi},y\ge 0)$.

By \cite[Proposition 5.17 and its proof]{Paper1}, type-0 data pairs have a Markov-like property.

\begin{lemma}\label{lem:type0Markov}
 Let $\Psi = (\cev{\fN},\fN_\beta)\sim\fP_\beta^0$ for some $\beta\in\cI$. For all $y\ge 0$, conditionally given $\cF^y$, 
  $$(\cev{\fN}^y_{\Psi},\fN^y_{\Psi})\text{, defined above, has conditional distribution }\fP_{\alpha^y}^0$$
  where $(\alpha^z,\,z\ge 0)=\skewerbar(\cev{\fN}+\fN_\beta)$ is the associated type-0 evolution.
\end{lemma}


Recall from the type-1 setting of Lemma \ref{type1fromtype0} that $\fN_*$ is a function of a type-1 data triple $\Psi$. For $y\ge 0$, define
\begin{equation}
\begin{split}
 m^y(\fN_*) &:= M_{\fN_*}^y(\inf\{t\ge 0\colon M_{\fN_*}^y(t)>0\})\\
 \text{and}\quad
\ff^y_{\Psi} &:= \big(m^{y+z}(\fN_*),\, 0\le z\le\inf\{w\ge 0\colon m^{y+w-}(\fN_*)=0\} \big),
\end{split}
\end{equation}
respectively the mass of the leftmost block at level $y$ and the leftmost spindle, evolving up from that level. These both vanish for $y\ge\sup\xi(\fN_*)$ and for a Lebesgue null set of levels that are in the range of the running supremum process of $\xi(\fN_*)$. For any other level, $\ff^y_{\Psi}$ is associated with a jump of $\xi(\fN_*)$ across level $y$, and 
along with the following spindles until $\xi(\fN_*)$ first hits level $y$, it forms a clade $\delta(0,\ff^y_{\Psi})\star\cev{\fn}^y$ for some point measure $\cev{\fn}^y$. 
Beyond $\cev{\fn}^y$, we collect clades above level $y$ as for type 0 and concatenate these to form a point measure $\fN^y_{\Psi}$. Let 
$\cev{\fN}^y_{\Psi} = \restrictshift{\cev{\fN}}{(-\infty,T_{\zeta(\ff^y_{\Psi})}(\cev{\fN}))\times\cE}\star(\cev{\fn}^y)^{\to}$.

We can alternatively represent this decomposition about level $y$ via $\fN_*^y := \clade\big(\ff_\Psi^y,\cev\fN_\Psi^y\big)\concat\fN_\Psi^y$. Point processes $\fN_*^y$ and $\fN_*^{\le y}$ can also be directly defined as above, from spindles in the excursions of $\xi(\fN_*)$ about $y$. For type 1, augment the type-0 filtration so that $\ff$ is also adapted.
More specifically, we define $(\cF^y,y\ge 0)$ as the natural filtration of $(\ff(y),\fN^{\le y}_{\Psi})$. Then $\fN_*^{\le y}$ is adapted. 
By \cite[Proposition 5.6 (and Lemma 3.41)]{Paper1}, we have the following.

\begin{lemma}\label{lem:type1Markov}
 In the setting of Lemma \ref{type1fromtype0}, for all $y\ge 0$, conditionally given $\cF^y$,
 $$(\ff^y_{\Psi},\cev{\fN}^y_{\Psi},\fN^y_{\Psi})\qquad\mbox{has conditional distribution }\besq_{m^y}(-1)\otimes\fP_{\alpha^y}^0,$$ 
 where $m^y:=m^y(\fN_*)$ and $(\alpha^z,\,z\ge 0)$ is such that $\big((0,m^z)\star\alpha^z,\,z\ge0\big) = \skewerbar(\fN_*)$ is the associated type-1 evolution. This includes the degenerate case $m^y=0$ and $\alpha^y=\emptyset$.  
\end{lemma}

In the sequel, we will abuse terminology and also refer to $((m^y,\alpha^y),y\ge 0)$ as a type-1 evolution. Note that for type 1, the part $(\alpha^y,y\ge 0)$ is only a
type-0 evolution up to the random level $y=\zeta(\ff)$. Above this level, the point measure $\restrict{\cev{\fN}}{(-\infty,T_{\zeta(\ff)}(\cev{\fN}))\times\cE}$ is redundant for the type-1 evolution, while it provides further blocks for the type-0 evolution. We recall some more facts about type-0 and type-1 evolutions from \cite{Paper1}. The following is an immediate consequence of the definitions.

\begin{proposition}\label{type01plustype1}
 Let $\fN_\beta$ and $\fN^\prime_{\beta^\prime}$ be two independent type-1 point measures. Then $\fN_\beta\star\fN^\prime_{\beta^\prime}$ is also
 a type-1 point measure. In particular, $\skewerbar(\fN_\beta\star\fN^\prime_{\beta^\prime})$ is a type-1 evolution starting from
 $\beta\star\beta^\prime$. 
 Similarly, $(\cev{\fN},\fN_\beta\star\fN^\prime_{\beta^\prime})$ is a type-0 data pair.
\end{proposition}

For type-0 and type-1 evolutions, the associated \emph{total mass evolutions}, $(\|\skewer(y,\cev\fN+\fN_{\beta})\|,\,y\ge 0)$ and $(\|\skewer(y,\fN_*)\|,\,y\ge 0)$ respectively, are as follows. 

\begin{proposition}[Theorem 1.5 of \cite{Paper1}]\label{type1totalmass}
 For any $\beta\in\cI$, the total mass evolution under $\fP_\beta^0$ is $\besq_{\|\beta\|}(1)$, while the total mass evolution under $\fP_\beta^1$ or $\fP_{x,\beta}^1$ is $\besq_{\|\beta\|}(0)$ or $\besq_{x+\|\beta\|}(0)$. In particular, the type-1 evolution a.s.\ is absorbed at $\emptyset$ in finite time.
\end{proposition}

The transition kernels of type-1 and type-0 evolutions are given in \cite[Propositions 4.30, 5.4 and 5.16]{Paper1}. Most relevant for us is the following.

\begin{proposition}[Pseudo-stationarity for type-0 and type-1 evolutions; Theorem 6.1 of \cite{Paper1}]\label{prop:type1:pseudo}
 Consider $\widebar{\beta}\sim{\tt PDIP}(\frac{1}{2},\frac{1}{2})$ and an independent random variable $M\in[0,\infty)$. Let $B\sim\besq_M(1)$ and $\alpha$ a type-0 evolution starting from $\beta=M\widebar{\beta}$. Then $\alpha^y$ has the same distribution $B(y)\widebar{\beta}$.
 
 The same result holds for type-1 evolutions $\alpha$, if we take $\widebar{\beta}\sim{\tt PDIP}(\frac{1}{2},0)$ and $B\sim\besq_M(0)$.
\end{proposition}

We refer to the respective distributions of $M\widebar{\beta}$ as type-0 and type-1 \emph{pseudo-stationary distributions}. We can also integrate this result over $y$ to extend this to independent random times, and in the type-1 case, we can rephrase this as a result conditionally given that the process survives to level $y$, since this conditioning only involves the total mass; see \cite[Theorem 6.9]{Paper1}. Specifically, we have the following.

\begin{proposition}[Proposition 6.2 of \cite{Paper1}]\label{prop:type1:pseudo_g}  
 Consider a type-1 evolution $(\alpha^y,y\ge 0)$ starting from $M\widebar{\beta}$, where $\widebar{\beta}\sim{\tt PDIP}(\frac{1}{2},0)$ is independent of $M\sim\ExpDist[\gamma]$ for some rate parameter $\gamma\in(0,\infty)$. Then the conditional distribution of $\alpha^y$ given $\alpha^y\neq\emptyset$ is the same as the (unconditional) distribution of $(2y\gamma+1)\alpha^0$.
 
 If $\widebar{\beta}\sim{\tt PDIP}(\frac{1}{2},\frac{1}{2})$, $M\sim\GammaDist[\frac{1}{2},\gamma]$ for type 0, then $\alpha^y\sim(2y\gamma+1)\alpha^0$ for all $y\ge 0$.
\end{proposition}

\begin{corollary}\label{corbesq2}
 For a type-1 evolution $(\alpha^y,y\ge 0)$ starting from the pseudo-stationary distribution of random mass $M$, given $\alpha^y\neq\emptyset$, the mass $\|\alpha^y\|$ is conditionally independent of $\alpha^y/\|\alpha^y\|$. The former is conditionally distributed as $B(y)$, where $(B(z),\,z\ge0)$ is a $\besq_M(0)$ conditioned to survive to time $y$, and the latter has conditional law $\PDIP[\frac{1}{2},0]$.
\end{corollary}

Note that conditioning the total mass, $\besq(0)$, more strongly to never become extinct gives rise to a $\besq(4)$ process; 
see e.g.\ \cite[p.\ 451]{PitmYor82}.

\subsection{Type-1 evolutions as pairs of leftmost blocks and remaining interval partition}

Recall Sharpe's definition \cite{Sharpe} (see also \cite[Definition A.18]{Li2011}) of Borel right Markov processes:
\begin{enumerate}[leftmargin=.82cm]
  \item[1.] Lusin state space (homeomorphic to a Borel subset of a compact metric space),
  \item[2.] right-continuous sample paths,
  \item[3.] Borel measurable semi-group and strong Markov property.
\end{enumerate}
It is additionally a Hunt process if it is quasi-left-continuous, i.e.
\begin{enumerate}[leftmargin=.82cm]
  \item[4.] left-continuous along all increasing sequences of stopping times.
\end{enumerate}

In preparation for a discussion of continuity, we recall the formal definition of $d_{\cI}$ from \cite{Paper1}.

\begin{definition} \label{def:IP:metric}
 We adopt the notation $[n] := \{1,2,\ldots,n\}$.
 For $\beta,\gamma\in \IPspace$, a \emph{correspondence} from $\beta$ to $\gamma$ is a finite sequence of ordered pairs of intervals $(U_1,V_1),\ldots,(U_n,V_n) \in \beta\times\gamma$, $n\geq 0$, where the sequences $(U_j)_{j\in [n]}$ and $(V_j)_{j\in [n]}$ are each strictly increasing in the left-to-right ordering of the interval partitions.

 The \emph{distortion} of a correspondence $(U_j,V_j)_{j\in [n]}$ from $\beta$ to $\gamma$, denoted by $\dis(\beta,\gamma,(U_j,V_j)_{j\in [n]})$, is defined to be the maximum of the following four quantities:
 \begin{enumerate}[label=(\roman*), ref=(\roman*)]
  \item $\sup_{j\in [n]}|\IPLT_{\beta}(U_j) - \IPLT_{\gamma}(V_j)|$,
  \item $|\IPLT_{\beta}(\infty) - \IPLT_{\gamma}(\infty)|$,
  \item $\sum_{j\in [n]}|\Leb(U_j)-\Leb(V_j)| + \IPmag{\beta} - \sum_{j\in [n]}\Leb(U_j)$, \label{item:IP_m:mass_1}
  \item $\sum_{j\in [n]}|\Leb(U_j)-\Leb(V_j)| + \IPmag{\gamma} - \sum_{j\in [n]}\Leb(V_j)$. \label{item:IP_m:mass_2}
 \end{enumerate}
 Note that the second of these quantities depends only on the partitions $\beta$ and $\gamma$ and not on the correspondence. 
 
 For $\beta,\gamma\in\IPspace$ we define
 \begin{equation}\label{eq:IP:metric_def}
  \dI(\beta,\gamma) := \inf_{n\ge 0,\,(U_j,V_j)_{j\in [n]}}\dis\big(\beta,\gamma,(U_j,V_j)_{j\in [n]}\big),
 \end{equation}
 where the infimum is over all correspondences from $\beta$ to $\gamma$.
\end{definition}
\begin{proposition}[Theorem 1.4 of \cite{Paper1}]\label{intro:thm1} Type-1 and type-0 evolutions are path-continuous Hunt processes
  in $(\cI,d_{\cI})$ and are continuous in the initial condition.
\end{proposition}
As noted above, a type-1 evolution has a leftmost block $m^y:=m^y(\fN_*)>0$ at Leb-a.e.\ level a.s.. 
Let $\gamma^y$ satisfy $(0,m^y)\star\gamma^y=\skewer(y,\fN_*)$. Consider
the continuous bijection $\varphi(m,\gamma)=(0,m)\concat\gamma$ from 
$\cJ^\bullet:=\{(m,\gamma)\!\in\![0,\infty)\!\times\!\cI\colon m\!>\!0\mbox{ or }\gamma\!\in\!\cI\setminus\cI^\circ\}\cup\{(0,\emptyset)\}$
to $\cI$, which has a (discontinuous) measurable inverse. Then Proposition \ref{intro:thm1} has the following corollary.
\begin{corollary}\label{corpairtype1}
 Representations $((m^y,\gamma^y),y\ge 0):=(\varphi^{-1}(\beta^y),y\ge 0)$ of type-1 evolutions $(\beta^y,y\ge 0)$ are $\cJ^\bullet$-valued Borel right Markov processes, but not Hunt. 
\end{corollary}
\begin{proof} 
 1. The space $\cJ^\bullet$, equipped with the metric $d_\bullet((m_1,\gamma_1),(m_2,\gamma_2))=|m_1-m_2|+d_\cI(\gamma_1,\gamma_2)$, 
    is as a Borel subset of a product of Lusin spaces and is therefore Lusin 
    (see \cite[Theorem 2.7]{Paper1} for the Lusin property of $(\cI,d_\cI)$).

 2. Consider $\fN_*\sim\fP^1_{(0,m)\times\gamma}$. It is a consequence of the clade construction and properties of \Stable[\frac32] processes 
    that $y\mapsto m^y(\fN^*)$ is c\`adl\`ag and  
    the only jumps are up from zero, one at the starting level of each excursion of $\xi(\fN_*)$ below the supremum. It is a.s.\ the case that no
    two such excursions share an endpoint. See \cite{BertoinLevy} for 
    details on fluctuation theory. It is not difficult to show that $(m^y,\gamma^y)$ is also c\`adl\`ag since for $m_n\rightarrow m_0$ and 
    $(m_n,\gamma_n)\in\cJ_\bullet$ for all $n\ge 0$, we have 
    $$d_\cI(\gamma_n,\gamma_0)\rightarrow 0\quad\mbox{if and only if}\quad d_\cI((0,m_n)\star\gamma_n,(0,m_0)\star\gamma_0)\rightarrow 0.$$ 
  
 3. Since $\varphi$ and $\varphi^{-1}$ are measurable bijections, the measurability of the semi-group and the strong Markov property follow 
    from Proposition \ref{intro:thm1}. 
  
 4. Consider two independent type-1 evolutions $(\beta^y,y\ge 0)$ and $(\gamma^y,y\ge 0)$. By Proposition \ref{type01plustype1}, the concatenation  $\beta^y\concat\gamma^y$ defines
    a type-1 evolution. Consider $\eta_n=\inf\{y\ge 0\colon \|\beta^y\|<1/n\}$. Then $\eta_n$ increases to $\eta=\inf\{y\ge 0\colon\beta^y=\emptyset\}$. Then
    the top block at level $\eta_n$ converges to 0, but the leftmost block of $\gamma^\eta$ is non-zero with positive probability.
\end{proof}

\section{Type-2 evolutions}\label{clocking}

\noindent We will derive properties like c\`adl\`ag sample paths, strong Markov property and $\besq(-1)$ total mass directly from Definition \ref{def:type2:v1}, in Sections \ref{secbrm} and \ref{sectm}.  However, it will
extend our toolkit to rephrase this definition in the context of the scaffolding and spindles construction of type-1 evolutions. Indeed, the rephrasing also simplifies establishing some basic symmetry and non-accumulation properties of type-2 evolutions, which will be our starting point. 

\subsection{Alternative definition of type-2 evolutions: deletion clocking}


Definition \ref{def:type2:v1} constructs a type-2 evolution from sequences of $\besq(-1)$ processes and type-1 evolutions. In fact, we can construct a process with the same distribution using only a single $\ff_1\sim\besq_a(-1)$ and a single type-1 data triple $\Psi_0=(\ff_2,\cev{\fN},\fN_\beta)\sim\fP_{b,\beta}^1$. This is a continuum analogue of the construction described in Section \ref{sec:discrete_clocking}.



\begin{figure}
 \centering
 \scalebox{.85}{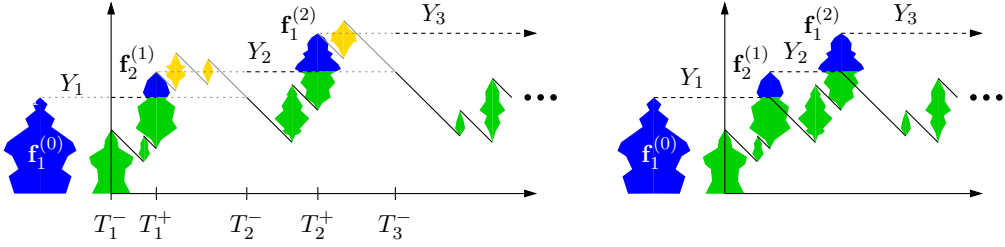}
 \caption{The effect of deletion clocking is to ignore intervals of spindles. Here, the clock spindles are dark blue, the ignored spindles are yellow, and other spindles are green.\label{fig:d_clk}}
\end{figure}


\begin{definition}\label{deftype2}
 For $(a,b,\beta)\in\cJ^\circ$, consider a \emph{type-2 data quadruple with initial state $(a,b,\beta)$}
 $$\Psi := (\ff_1,\ff_2,\cev\fN,\fN_{\beta}) \sim \besq_a(-1)\otimes\besq_b(-1)\otimes\fP^0_{\beta} =: \fP^2_{a,b,\beta}.$$
 Let $\fN_* := \textsc{clade}(\ff_2,\cev\fN)\concat\fN_{\beta}$ and $\fX_* := \xi(\fN_*)$. We define
 $((m_1^y,m_2^y,\beta^y),\,y\ge0)$ in four steps.
 
 \textbf{Step 1}. We define \emph{clock levels} $(Y_n)$ and \emph{clock change times} $(T_n^{\pm})$ for $\fX_*$ recursively. These quantities appear labeled in Figure \ref{fig:d_clk}. Set $Y_0 = 0$, $T_0^+ = T_1^- = 0$, $Y_1 := \zeta(\ff_1)$, and for $n\ge 1$,
 \begin{equation}
   T^+_n:=\inf\{t \ge T^-_n\colon \fX_*(t) > Y_n\},\quad 
   Y_{n+1}:=\fX_*(T^+_n),\quad
   T^-_{n+1}:=\inf\{t > T^+_n\colon \fX_*(t) \leq Y_n\}
  \label{eq:d_clk:levels}
 \end{equation}
 with the conventions $\inf\emptyset\!=\!\infty$ and $\fX_*(\infty)\!=\!\infty$. Though we omit it from our notation, we view each of the preceding quantities as a function of $\Psi$.
 
 \textbf{Step 2}. We define clock spindles. Let $\ff^{(0)}=\ff_1$. For $j\ge 1$, let $\ff^{(n)}$ denote the cut-off top part $\hat f^{Y_{n}}$ of the spindle $f$ that occurs at time $T^+_{n}$ in $\fN_*$. Each $\ff^{(n)}$ will be the clock spindle during the interval $[Y_{n},Y_{n+1})$.
 
 \textbf{Step 3}. We define type-1 data. Let $\Psi_0 := (\ff_{\Psi_0},\cev\fN_{\Psi_0},\fN_{\Psi_0}):=(\ff_2,\cev\fN,\fN_{\beta})$. For $n\ge 1$, let
 \begin{equation}\label{eq:d_clk:shift}
   \Psi_n := \left(\ff_{\Psi_n},\cev\fN_{\Psi_n},\fN_{\Psi_n}\right):=\left(0,\RestrictShift{\cev\fN}{(-\infty,T_{Y_n}(\cev\fN)]\times\cE},\left(\ShiftRestrict{\fN_*}{(T_{n+1}^-,\infty)\times\cE}\right)^0\right) 
	\qquad \text{for }n\ge1.
 \end{equation}
 The superscript $0$ on the rightmost term above is in the sense of the cutoff processes $\fN^y$ described around Lemma \ref{lem:type0Markov}, in which spindles below a given level are removed or cut off. Each $\Psi_n$ is a type-1 data triple for the non-clock top mass and spinal masses during the interval $[Y_{n},Y_{n+1})$.
 
 \textbf{Step 4}. We define the evolution. For $n\ge 0$ even,
 \begin{equation}
  m_1^y := \ff^{(n)}(y-Y_{n}), \quad (0,m_{2}^y)\concat\alpha^y := \skewer\big(y-Y_{n},\ShiftRestrict{\fN_*}{(T_{n+1}^-,\infty)\times\cE}\big) \quad \text{for }y\in [Y_n,Y_{n+1}),\label{eq:type_2_def}
 \end{equation}  
 where $m_2^y = 0$ if and only if the skewer in the last expression has no leftmost block. For $n\ge 1$ odd, the definition is the same, but with $m_1^y$ and $m_2^y$ swapping roles.
\end{definition}
\medskip

The effect of this construction is to skip over intervals of spindles from $\fN_*$, ensuring that they never contribute blocks to the skewer: for each $n\ge 1$, the process $\restrict{\fN_*}{(T_n^+,T_{n+1}^-]\times\cE}$ is redundant. We therefore refer to this construction as \em deletion 
clocking\em. This is illustrated in Figure \ref{fig:d_clk}. The time of the succession of clock spindles $\ff^{(n)}$, which is the level of the scaffolding, is the
time of the type-2 evolution. The deletions next to each clock spindle are naturally interpreted as \em emigration \em as each family of spindles in an excursion 
above the minimum of the \Stable[\frac32] process $\xi(\fN_*)|_{(T_n^+,T_{n+1}^-]}$ is removed from the evolution and such excursions form a homogeneous Poisson point process up to the level
where the last clock spindle dies. 

\begin{proposition}
 The process constructed in Definition \ref{deftype2} is a type-2 evolution.
\end{proposition}
\begin{proof} 
%
  Consider a data quadruple $\Psi=(\ff_1,\ff_2,\cev{\fN},\fN_\beta)\sim\fP_{a,b,\beta}^2$ and the filtration $(\cF^y,y\ge 0)$
  generated by $(\ff_1(y),\ff_2(y),\cev{\fN}^{\le y},\fN_\beta^{\le y})$. We will use the notation of Definition \ref{deftype2} to inductively
  set up all random variables as needed for Definition \ref{def:type2:v1}, and we will show that Definitions 
  \ref{def:type2:v1} and \ref{deftype2}, in this setup, yield pathwise the same process $((m_1^y,m_2^y,\alpha^y),y\ge 0)$. For the purpose 
  of this proof we will mark all random variables appearing in Definition \ref{def:type2:v1} by an underscore.

  Now, $\underline{\ff}^{(0)}:=\ff_1$ and 
  $(0,\underline{\fm}^{(0)})\star\underline{\gamma}^{(0)}:=\skewerbar(\clade(\ff_2,\cev{\fN})\star\fN_\beta)$ 
  have the appropriate joint distribution and achieve 
    $((\underline{m}_1^y,\underline{m}_2^y,\underline{\alpha}^y),0\le y\le \underline{Y}_1)=((m_1^y,m_2^y,\alpha^y),0\le y\le Y_1)$. 
  Suppose we have defined up to $(\underline{\ff}^{(n-1)},\underline{\fm}^{(n-1)},\underline{\gamma}^{(n-1)})$ and identified 
    $((\underline{m}_1^y,\underline{m}_2^y,\underline{\alpha}^y),0\le y\le \underline{Y}_{n})=((m_1^y,m_2^y,\alpha^y),0\le y\le Y_{n})$ 
  for some $n\ge 1$. Then given $\cF^{Y_{n-1}}$, we apply Lemma \ref{lem:type1Markov}, which is the Markov-like property of the type-1 data triple 
  $\Psi_{n-1}$ at the level $\zeta(\ff^{(n-1)})=Y_n-Y_{n-1}$, to find a post-$\zeta(\ff^{(n-1)})$ data triple $\Psi_{n-1}^{\zeta(\ff^{(n-1)})}$. The first component of this triple is $\ff^{(n)}$ and the last component is $\fN_{\Psi_n}$. Noting that $\ff^{(n)}$ and $\fN_{\Psi_n}$ are 
  conditionally independent given the pre-$\zeta(\ff_{n-1})$ data, indeed given $\cF^{Y_n}$, we proceed as follows. Suppose $n$ is even. First, 
    $\underline{\ff}^{(n)}:=\ff^{(n)}\sim{\tt BESQ}_{m_1^{Y_{n}}}(-1)$, 
  is as appropriate for Definition \ref{def:type2:v1}, since $\underline{Y}_{n}=Y_{n}$. 
  Second, $\Psi_{n}\sim\fP_{\alpha^{Y_{n}}}^1=\fP_{\underline{\alpha}^{\underline{Y}_n}}^1$, which gives rise to a type-1 evolution 
  $(0,\underline{\fm}^{(n)})\star\underline{\gamma}^{(n)}:=\skewerbar(\fN_{\Psi_{n}})\sim\bP_{\underline{\alpha}^{\underline{Y}_n}}^1$,
  as required, since we have $\underline{m}_{2}^{\underline{Y}_n}=\underline{\ff}^{(n-1)}(\zeta(\underline{\ff}^{(n-1)}))=0$. This also implies that 
  for all $y\in[0,Y_{n+1}-Y_n)$
  \begin{align*}(\underline{m}_{1}^{\underline{Y}_n+y},(0,\underline{m}_{2}^{\underline{Y}_n+y})\star\underline{\alpha}^{\underline{Y}_n+y})
   &=(\underline{\ff}^{(n)}(y),(0,\underline{\fm}^{(n)}(y))\star\underline{\gamma}^{(n)}(y))
   =(\ff^{(n)}(y),\skewer(y,\fN_{\Psi_n}))\\
   &=(m_{1}^{Y_n+y},(0,m_{2}^{Y_n+y})\star\alpha^{Y_n+y}),
  \end{align*}
  as required. The same argument applies for $n$ odd, with the roles of 1 and 2 interchanged.
\end{proof}


\begin{lemma}\label{lem:type2_symm}
  If we modify Definition \ref{deftype2} so that we let $\fN_*=\textsc{clade}(\ff_1,\cev\fN)\concat\fN_{\beta}$ and $Y_1=\zeta(\ff_2)$ and accordingly swap the parity in Step 4., we obtain a type-2 evolution that is pathwise the same as in Definition \ref{deftype2}, with identical sets $\{\zeta(\ff_1),\zeta(\ff_2)\}\cup\{Y_n,n\ge 0\}$. In particular, the point measure $\restrict{\cev{\fN}}{(-\infty,T_{\min\{\zeta(\ff_1),\zeta(\ff_2)\})}(\cev{\fN}))\times\cE}$ is redundant for the type-2 evolution.
\end{lemma}

\begin{proof} 
 For the purposes of this proof, we add underscores and write $\underline{Y}_{\,j}$, $\underline{\ff}_{\,i}^{(j)}$, $\underline{\Psi}_{\,j}$, $i=1$ or $i=2$, $j\ge 0$, and  
  $((\underline{m}_{\,1}^y,\underline{m}_{\,2}^y,\underline{\alpha}^y),y\ge 0)$ in the modification of Definition \ref{deftype2}. We remark that the underscores here are unrelated to those in the previous proof. 
  The main aim of this proof is to show the pathwise equality 
  $((\underline{m}_{\,1}^y,\underline{m}_{\,2}^y,\underline{\alpha}^y),y\ge 0)=((m_1^y,m_2^y,\alpha^y),y\ge 0)$. We only discuss the case where $a>0$ and $b>0$. The cases where $a=0$ or $b=0$ can
  then be checked similarly.
  
  On the event $\{\zeta(\ff_1)<\zeta(\ff_2)\}$, we have $\underline{Y}_{\,0}=0=Y_0<Y_1=\zeta(\ff_1)<Y_2=\zeta(\ff_2)=\underline{Y}_{\,1}$, and we see inductively that
  $\underline{Y}_{\,j}=Y_{j+1}$, $\underline{\ff}_{\,i}^{(j)}=\ff_{\,i}^{(j+1)}$ and $\underline{\Psi}_{\,j}=\Psi_{j+1}$ for all $j\ge 1$, $i=1,2$. It is now easy to see that the pathwise equality holds on this event. Similarly, on $\{\zeta(\ff_1)>\zeta(\ff_2)\}$, we have $\underline{Y}_{\,1}=\zeta(\ff_2)$ and 
  $\underline{Y}_{\,j+1}=Y_{j}$ for all $j\ge 1$, and the same argument applies.

  In particular, the sets $\{Y_n,n\ge 0\}$ and $\{\underline{Y}_{\,n},n\ge 0\}$ differ precisely by the omission of either $\zeta(\ff_2)$ from the former or of $\zeta(\ff_1)$ 
  from the latter. The last statement of the lemma follows using the original definition on $\{\zeta(\ff_1\}<\zeta(\ff_2)\}$ and the modified definition on 
  $\{\zeta(\ff_1)>\zeta(\ff_2)\}$. 
\end{proof}   

It is not a priori clear in Definition \ref{def:type2:v1}, nor equivalently in Definition \ref{deftype2}, that clock changes cannot accumulate at a 
finite level $Y_\infty=\sup_{n\ge 0}Y_n<\infty$. This would leave the type-2 evolution undefined for $y\ge Y_\infty$, so we address this point before establishing any further properties. 

\begin{lemma}\label{type2welldef}
 For all $(a,b,\beta)\in\cJ^\circ$, the type-2 evolution as constructed in Definition \ref{deftype2} is such that there is a.s.\ some finite $n\ge 0$ for which $Y_n<Y_{n+1}=\infty$ and as $y$ increases to $Y_n$, the evolution $(m_1^y,m_2^y,\alpha^y)$ approaches $(0,0,\emptyset)$. 
\end{lemma}
\begin{proof}
%
 %
 First, we prove the claimed convergence to $(0,0,\emptyset)$. 
 The events $\{T_1^+ = \infty\}$, $\{Y_2 = \infty\}$, and $\{\zeta(\skewerbar(\fN_*)) < Y_1\}$ are equal up to null sets. 
On these events, $m_1^y$ converges to 0 as $y$ increases to $Y_1$, and $(m_2^y,\alpha^y)$ are already absorbed at $(0,\emptyset)$ prior to that level. We proceed inductively. On the event $\{T_n^+ <\infty\}$, this time $T_n^+$ is when the type-1 scaffolding $\fX_*$ exceeds level $Y_n$. Since this scaffolding eventually dies at level 0, we get $T_{n+1}^- <\infty$ a.s.. Now, on the event $\{T_{n+1}^+ = \infty\}$, we apply the same argument as before to $\shiftrestrict{\fN_*}{(T_{n+1}^-,\infty)\times\cE}$ in place of $\fN_*$, to conclude that $(m_1^y,m_2^y,\alpha^y)$ approaches $(0,0,\emptyset)$ as $y$ increases to $Y_{n+1}$.
 
 It remains to show that $Y_n<Y_{n+1}=\infty$ for some $n\ge 1$. 
  We claim that it suffices to prove the following.
  \begin{center}
   $(*)$\ \ \ \parbox[c]{\textwidth - 2cm}{
  Consider any two spindles of heights $\zeta(f_1)=c_1$ and $\zeta(f_2)=c_2$ with $c_1<c_2$. Apply the construction of Definition \ref{deftype2} to $(\ff_1,\ff_2,\cev{\fN},\fN_\emptyset)=(f_1,f_2,\cev{\fN},0)$ for $(\cev{\fN},0)\sim\fP_\emptyset^0$. Then there is some $n\ge 1$ for which $Y_n < Y_{n+1}=\infty$.}\ \ \ \hphantom{(*)}
  \end{center}
 Indeed, once this is shown, $a+b>0$ in the general case implies $Y_2>0$, and only finitely many clades of $\fN_\beta$ survive to level $Y_2$. 
 We apply $(*)$ to these clades one by one, with $c_1$ as the final clock level of the preceding clades and $c_2$ as the next level after $c_1$ at which the top mass of the next clade vanishes, to see that each clade contributes a finite number of clock change levels.
 
 To prove $(*)$, we note that this can be read as a statement about the
  \Stable[\frac32]\ L\'evy process $X=c_2+\xi\Big(\shiftrestrict{\cev{\fN}}{(T_{c_2}(\cev{\fN}),0)\times\cE}\Big)$.
  Specifically, note that $\zeta\big(\ff_1^{(j)}\big)$ or $\zeta\big(\ff_2^{(j)}\big)$ is the overshoot $Y_{j+1}-Y_j$ of $X$ when first crossing level $Y_j$.

  Now we extend $X$ to a \Stable[\frac32]\ process with infinite lifetime so that $T_j^+<\infty$ for all $j\ge 1$, and we show that $Y_j\rightarrow\infty$. To this end, let $\Delta_n=Y_{n+1}-Y_n$ and $R_n=\Delta_{n+1}/\Delta_n$
  for $n\ge 1$. By the strong Markov property of \Stable[\frac32], the conditional distribution of $\Delta_{n+1}$ given $\Delta_1,\ldots,\Delta_n$ equals the law of the overshoot of a $\Stable[\frac32]$ process
  when first crossing $\Delta_n$, which is the same as the overshoot of its $\Stable[\frac12]$ ladder height subordinator \cite{BertoinLevy}. By stable scaling, for each $n$, $R_n$ is independent of $\Delta_n$ and is 
  distributed like the overshoot of a $\Stable[\frac12]$ subordinator across 1. So the sequence $(R_n,\,n\geq 1)$ is i.i.d.\ and
  $$\Delta_{n+1} = \Delta_1 \cdot \prod_{i=1}^{n} R_{i} \qquad \text{for }n\geq 1.$$
  Thus, $(\log(\Delta_n),\,n\geq 1)$ is a random walk. It suffices to show that the increments $\log(R_n)$, $n\ge 1$, of this walk have non-negative expected value.
 
  We can get at the law of $R_{n}$ by taking advantage of the $\Stable[\frac12]$ inverse local time subordinator associated with one-dimensional Brownian motion, $(B(t),\,t\geq 0)$. In this 
  setting, $R_{n}$ is distributed like $T-1$, where $T$ is the time of the first return of $B$ to zero, after time 1. By a calculation based on the reflection principle, we find
  $\bP(T < t) = \frac{2}{\pi}\arctan(t-1)$.
 Thus,
  $$\bE\left(  \log(R_n) \right) = \int_1^\infty \log(t-1)\bP(T\in dt) = \frac{2}{\pi}\int_1^{\infty} \log(t-1)\frac{1}{(t-1)^2 + 1}dt = 0. \vspace{-0.53cm}$$ 
\end{proof}

\subsection{Type-2 evolutions as Borel right Markov processes}\label{secbrm}

In this section we will prove Theorem \ref{thm:diffusion}, i.e.\ that type-2 evolutions are Borel right Markov processes, and that the 
IP-valued variant is a path-continuous Hunt process. We listed the properties 1.-4.\ that this comprises before Proposition \ref{intro:thm1}.

\begin{proof}[Proof of Theorem \ref{thm:diffusion}]
 1. By Lemma \ref{type2welldef}, type-2 evolutions take values in $\cI^\circ$ or $\cJ^\circ$ of (\ref{type2spaces}), 
    which are Lusin as Borel subsets of products of Lusin spaces (see \cite[Theorem 2.7]{Paper1}). 
 
 2. We first prove the path-continuity of the $\cI^{\circ}$-valued type-2 evolution. For $n\ge 0$, in between $Y_n$ and $Y_{n+1}$, this 
    process is formed by concatenating a \BESQ[-1] block to the left of an $\cI$-valued type-1 evolution, 
    $(0,\ff^{(n)})\concat (0,\fm^{(n)})\concat \gamma^{(n)}$. The 
    \BESQ[-1] process is continuous and, as noted in Proposition \ref{intro:thm1}, so is the type-1 evolution. By \cite[Lemma 2.11]{Paper1}, an interval 
    partition process formed by concatenation of two continuous interval partition processes is again continuous. To see continuity at $Y_n$, 
    first suppose that $n\ge 2$ is even. We note that as $y$ approaches $Y_n$ from below, the $\cI^{\circ}$-valued process approaches 
    $(0,0)\concat (0,m_1^{Y_n})\concat\alpha^{Y_n}$, while for $y$ approaching $Y_n$ from above, it approaches 
    $(0,m_1^{Y_n})\concat(0,0)\concat\alpha^{Y_n}$, by the continuity of \BESQ[-1] and $\cI$-valued type-1 evolution. The argument for odd $n\ge 1$ is the same,
    with $m_2^{Y_n}$ in the place of $m_1^{Y_n}$.
    
    The c\`adl\`ag property of $\cJ^\circ$-valued type-2 evolution follows similarly from the corresponding property of $\cJ^\bullet$-valued type-1 evolution proved
    in Corollary \ref{corpairtype1}. Specifically, continuity at $Y_n$ still holds by the same argument, using the path-continuity at independent 
    random times of $\cJ^\bullet$-valued type-1 evolutions, which follows from the path-continuity at fixed levels, which in turn holds as no excursion of a
    \Stable[\frac32] process below the supremum starts from a fixed level.
 
 
 3. The type of construction undertaken in Definition \ref{def:type2:v1}, in which a right Markov process with finite lifetime is reborn at the end of the 
    lifetime according to a probability kernel, has been studied by Meyer \cite{Mey75}. 
    Type-1 evolutions and \BESQ[-1] processes are Borel right Markov processes (see Corollary \ref{corpairtype1}), and thus so too is the process $\big((m_1^y,m_2^y,\alpha^y,1),\,0\le y\le Y_1\big)$ starting from any $(a,b,\beta)\in\cJ^\circ$ with $a>0$ and killed at $Y_1$.  By swapping the parity as in the statement of Lemma \ref{lem:type2_symm}, we can similarly define $\big((m_1^y,m_2^y,\alpha^y,2),\,0\le y\le Y_1\big)$ starting from
$(a,b,\beta)\in\cJ^\circ$ with $b>0$ and killed at $Y_1$, where the fourth component $I(y)=1$ or $I(y)=2$ records which of the two top blocks is evolving according to \BESQ[-1] and which is forming a type-1 evolution with $\alpha^y$. We define the deterministic kernel $\overline{N}((0,x,\beta,1);\,\cdot\,)=\delta_{(0,x,\beta,2)}$, $\overline{N}((x,0,\beta,2);\,\cdot\,)=\delta_{(x,0,\beta,1)}$. As noted in \cite[Definition 8.1]{Sharpe}, Borel right Markov processes are right Markov processes satisfying the hypoth\`eses droites, in Meyer's sense. Therefore, we can apply \cite[Th\'eor\`eme 1 and Remarque on p.474]{Mey75} to conclude that if we alternate killed processes with $I(y)=1$ and $I(y)=2$, using transitions according to $\overline{N}$ to determine initial states from the previous killing state,
\begin{equation}\label{augment}\mbox{the process $\big((m_1^y,m_2^y,\alpha^y,I(y)),\,y\ge 0\big)$ is a right Markov process,}
\end{equation}
satisfying the strong Markov property. It is not hard to show that the semigroup of this process is Borel, see e.g. the last point in the proof of \cite[Th\'eor\`eme (3.18)]{Bec07}. In Proposition \ref{contini} we strengthen this to continuity. 
 
    Lemma \ref{lem:type2_symm} verifies Dynkin's criterion to show that the $\cJ^{\circ}$-valued type-2 evolution is a right Markov process as well. To see that the $\cI^{\circ}$-valued type-2 evolution is a right Markov process, just note that every state $(0,a)\concat(0,b)\concat\beta\in\cI^\circ$ corresponds to two states $(a,b,\beta,1)$ and $(b,a,\beta,2)$, but that both are based on $\besq_a(-1)$ and type-1 evolution from $(b,\beta)$ and hence construct the same process, apart from maintaining opposite last components $I(y)$. Hence, Dynkin's criterion applies again. 
 
 4. The Hunt property of $\cI^\circ$-valued type-2 evolutions holds since sample paths are continuous.
\end{proof}

In Section \ref{sec:Holder} we prove a H\"older continuity result for type-2 evolutions started from certain initial distributions, with bounds on all moments of the H\"older constants. It is possible to mimic \cite[Proof of Proposition 5.11]{Paper1} and appeal to the construction of Definition \ref{deftype2} to prove H\"older continuity with index $\theta\in (0,\frac14)$ at all times after time zero, from any initial state, but in this setting we could not also give the desired bounds, so we omit such arguments here.

In order to establish continuity of the semigroup of type-2 evolution in the initial condition we require some intermediate results.

\begin{lemma}\label{lemma slightly stronger continuity}
Suppose that $((b_n,\beta_n),\,n\geq 1)$ is a sequence in $(\cJ^\bullet,d_\bullet)$ that converges to $(b,\beta)$ and that $(x_n,\,n\ge 1)$ is a 
sequence of levels converging to $x>0$. Let $((m^y_n, \gamma^y_n), y\geq 0)$ and  $((m^y, \gamma^y), y\geq 0)$ be type-1 evolutions started from $(b_n,\beta_n)$ and $(b,\beta)$ respectively.  If $f\colon\cJ^\bullet\to \bR$ is bounded and continuous, then 
$$\bE\left[ f(m^{x_n}_n, \gamma^{x_n}_n)\right] \to  \bE\left[f(m^x, \gamma^x)\right].$$
\end{lemma}

\begin{proof}
If $g\colon\cI \to \bR$ is bounded and continuous, then the fact that
$$\bE\left[g\left((0,m^{x_n}_n)\concat \gamma^{x_n}_n\right)\right] \to  \bE\left[g\left((0,m^{x})\concat \gamma^x\right)\right].$$
is established in the proof of \cite[Proposition 5.20]{Paper1}.  The slightly stronger version that separates out convergence of the top mass 
follows from the coupling used in that proof. Specifically, that proof reduces the argument to finitely many clades, each of which is composed of 
an initial spindle and an independent $\Stable[\frac32]$ L\'evy process. Furthermore, the ladder height process of a $\Stable[\frac32]$ L\'evy 
process, in which the leftmost spindle at each level can be found, is a $\Stable[\frac12]$ subordinator. The probability that $x$ is in its range 
is zero, so that the evolution of the leftmost mass is continuous around level $x$ with probability one.    
\end{proof}

It will be convenient to augment the type-2 evolution $(\Gamma^y, y\ge 0)$ by the counting process $J(y)=\inf\{j\ge 0\colon Y_{j+1}>y\}$ counting 
its clock changes. This process $((\Gamma^y,J(y)),y\ge0)$ can be constructed as a strong Markov process as in (\ref{augment}) and similarly relates to $(\Gamma^y,y\ge 0)$ by 
Dynkin's criterion. Let $p$ be the parity map sending even numbers to $2$ and odd numbers to $1$. The state space for the evolution 
$((\Gamma^y,J(y)), y\geq 0)$ is the set
$$ \cJ^+ = \{ ((m_1,m_2,\beta), j) \in \cJ^\circ \times \mathbb{N}_0\colon \ m_{p(j+1)}>0\}.$$
In the following lemma, we write $\bE_{\gamma,j}$ to denote the expectation for the augmented process starting from $(\gamma,j)\in\cJ^+$.

\begin{lemma}\label{lemma markovprops}
 Suppose that $(\Gamma^y,y\ge 0)$ is a type-2 evolution with clock change levels $Y_j$. 
 Then
 \begin{enumerate}[label=(\roman*), ref=(\roman*)]
  \item for  all $f\colon\cJ^+ \to \bR$ bounded and continuous
   $$\bE \left[ f(\Gamma^{Y_j+u}, J(Y_j+u)) \middle| \cF^{{Y_j}} \right] = \bE_{\Gamma^{Y_j}, j}\left[f(\Gamma^{u}, J(u))\right],\qquad\bP\mbox{-a.s..}$$
  \item for all $h\colon\cJ^\circ \to \bR$ bounded and continuous and for $\bP$-a.e.\ $\omega$ 
   \begin{align*}  &\bE\left[ h(\Gamma^y) \mathbf{1}\{Y_j\leq y <Y_{j+1}\} \middle| \cF^{Y_j}\right](\omega)\\ 
				&=\mathbf{1}\{Y_{j}(\omega)\leq y\} 
				  \bE_{\Gamma^{Y_j(\omega)}(\omega),j} \left[ h(\Gamma^{y\vee Y_j(\omega)- Y_j(\omega)}) \mathbf{1}\{ y\vee Y_j(\omega)-Y_j(\omega) <Y_{1}\}\right].
   \end{align*}
 \end{enumerate}
\end{lemma}

\begin{proof}
The first claim is immediate from the construction of type-2 evolutions and the second follows from the proof of \cite[Theorem 2.3.3]{chung2006markov} applied to the augmented Markov process $((\Gamma^y,J(y)),y\ge 0)$. The book \cite{chung2006markov} assumes that the Markov process takes place on a locally compact state space, but that is not needed in the proof of Theorem 2.3.3.  The right-continuous dependence of the semigroup on time needed in the proof follows from the right-continuity of sample paths.
\end{proof}

Next we establish weak continuity at clock levels.

\begin{lemma}\label{lemma clockcont}
Suppose that $(a_n,b_n,\beta_n) \rightarrow (a,b,\beta)$ in $(\cJ^\circ,d_\cJ)$ with $a>0$.  Let $(\Gamma^y_n,\,y\geq 0)$ and $(\Gamma^y,\,y\geq 0)$
be type-2 evolutions started from $(a_n,b_n,\beta_n)$ and $(a,b,\beta)$ respectively with respective clock levels $Y^n_k$ and $Y_k$.  Then
 $$\bE\left[f\left(\Gamma_n^{Y^n_j}, Y^n_j\right)\right] \rightarrow \bE\left[f\left(\Gamma^{Y_j}, Y_j\right)\right].$$
\end{lemma}

\begin{proof} We first establish the claim for $j=1$. Let $((\Gamma^y_n, J_n(y)), y\geq 0)$ and $((\Gamma^y,J(y)),y\geq 0)$ be the augmented type-2 evolutions started from $(a_n,b_n,\beta_n,0)$ and $(a,b,\beta,0)$. Let $\ff^{(0)}$ be a $\besq(-1)$ started from $a$, let $(\mathbf{m}^{(0)}, \gamma^{(0)})$ be an independent type-1 evolution started from $(b, \beta)$, and let $(\mathbf{m}^{(0)}_n,\gamma^{(0)}_n)$ be a type-1 evolution, independent from $\ff^{(0)}$, and started from $(b_n, \beta_n)$.  From the construction of type-2 evolutions, we see that 
$$ (\Gamma^y,\,0\leq y\leq Y_1) \stackrel{d}{=}  \left( \left(\ff^{(0)}(y), \mathbf{m}^{(0)}(y), \gamma^{(0)}(y)\right),\,0\leq y\leq Y_1\right)$$
and
\begin{equation} (\Gamma^y_n,\,0\leq y\leq Y^n_1) \stackrel{d}{=} 
                  \left( \left(\frac{a_n}{a}\ff^{(0)}\left(\frac{a}{a_n}y\right), \mathbf{m}_n^{(0)}(y), \gamma^{(0)}_n(y)\right),\,0\leq y\leq Y^n_1\right).
  \label{scalebesq}
\end{equation}
Note that, from this construction, $Y^n_1 = (a_n/a)Y_1$.  Furthermore, from \cite[Equation (13)]{GoinYor03} we see that $Y_1$ is distributed like $a/(2G)$ where $G\sim\GammaDist[\frac32,1]$.  In particular, $Y_1$ has a continuous density $q$ on $(0,\infty)$. Disintegrating based on the value of $Y^n_1$, we see that
$$\bE\left[ f( \ff^{(0)}_{n}(Y^n_1), \mathbf{m}^{(0)}_n(Y^n_1),\gamma^{(0)}_n(Y^n_1), Y^n_1)\right] 
 = \!\int_0^\infty\! \bE\left[ f\left( 0, \mathbf{m}^{(0)}_n\left(x\right), \gamma^{(0)}_n\left(x\right),x\right)\right]\frac{a}{a_n}q\left(\frac{ax}{a_n}\right) dx.
$$
It follows from Lemma \ref{lemma slightly stronger continuity} and a version of the dominated convergence theorem (e.g.\ \cite[Theorem 1.21]{Kallenberg}) that
\begin{equation}\label{base1}\bE_{(a_n,b_n,\beta_n),0}\left[f\left(\Gamma^{Y_1}, Y_1\right)\right] \rightarrow \bE_{(a,b,\beta),0}\left[f\left(\Gamma^{Y_1}, Y_1\right)\right].
\end{equation}
This completes the proof for $j=1$, for all $a>0$, $b\ge 0$ and $\beta\in\cI$. The same proof applied to augmented type-2 evolutions started from $(a_n,b_n,\beta_n,1)$ and $(a,b,\beta,1)$ shows 
\begin{equation}\label{base2}\bE_{(a_n,b_n,\beta_n),1}\left[f\left(\Gamma^{Y_1},Y_1\right)\right] \rightarrow \bE_{(a,b,\beta),1}\left[f\left(\Gamma^{Y_1}, Y_1\right)\right],
\end{equation}
for all $a\ge 0$, $b>0$ and $\beta\in\cI$. The inductive step $j\rightarrow j+1$ follows from the strong Markov property of the augmented type-2 
evolutions at clock levels $Y^n_j$ and $Y_j$, applying (\ref{base2}) for odd $j$ and (\ref{base1}) for even $j$.
\end{proof}

\begin{proposition}\label{contini}
Fix  $y\geq 0$ and define $F_y\colon\cJ^\circ \to \mathcal{P}(\cJ^\circ)$, by letting $F_y(a,b,\beta)$ be the law at level $y$ of a type-2 evolution starting from 
the initial state $(a,b,\beta)\in \cJ^\circ$. Then $(a,b,\beta)\mapsto F_y(a,b,\beta)$ is weakly continuous. Similarly define $G_y(\gamma)$ as the law at level $y$
for the IP-valued variant starting from $\gamma\in\cI^\circ$. Then $\gamma\mapsto G_y(\gamma)$ is weakly continuous on $\cI^\circ$.
\end{proposition}

\begin{proof}
We first prove the $\cJ^\circ$-valued case. Suppose that $(a_n,b_n,\beta_n)\rightarrow(a,b,\beta)$ in $(\cJ^\circ,d_\cJ)$, i.e. $a_n\to a$, 
$b_n\to b$ and $d_\cI(\beta_n,\beta)\to 0$. We may assume without loss of generality that $a>0$. Once the proof is complete for this subcase, we 
can apply Lemma \ref{lem:type2_symm} to deduce the subcase $a=0$, $b>0$; the subcase $a=b=0$, $\beta=\emptyset$ is trivial. Let 
$(\Gamma^y_n,y\ge 0)$ and $(\Gamma^y,y\ge 0)$ be $\cJ^\circ$-valued type-2 evolutions started from $(a_n,b_n,\beta_n)$ and $(a,b,\beta)$, 
respectively, with respective clock levels $(Y^n_j)_{j\ge0}$ and $(Y_j)_{j\ge0}$.  Observe that for all bounded continuous $f\colon\cJ^\circ\rightarrow\bR$
\begin{equation}\label{serieseq}
  \bE\left[ f\left(\Gamma^y_n\right)\right] = \sum_{j=0}^\infty \bE \left[ f\left(\Gamma^y_n\right) \mathbf{1}\{Y^n_j\leq y <Y^n_{j+1}\}\right].
\end{equation}
By Lemma \ref{lemma clockcont} and the Skorohod representation theorem, we may now assume $Y_j^n(\omega)\rightarrow Y_j(\omega)$, 
$d_{\cJ}\left(\Gamma^{Y_j^n(\omega)}_n(\omega),\Gamma^{Y_j(\omega)}(\omega)\right)\rightarrow 0$, and since $\bP(Y_j=y)=0$, also 
$\mathbf{1}\{Y_j^n(\omega)\le y\}\rightarrow\mathbf{1}\{Y_j(\omega)\le y\}$ for $\bP$-a.e.\ $\omega$. 
Recall that $Y_1$ and the initial clock spindle under $\bP_{\gamma,j}$ are associated with the block labeled 1 when $j$ is even and with the block labeled $2$ when $j$ is odd. For $\gamma = \Gamma^{Y_j(\omega)}(\omega)$ or $\Gamma_n^{Y^n_j(\omega)}(\omega)$, in either case this is the non-zero top mass of $\gamma$. 
Recall also from (\ref{scalebesq}) that $\besq(-1)$ 
processes with converging initial states can be coupled to converge uniformly together with their lifetimes. In particular, we 
can use their convergence in distribution together with Lemma \ref{lemma slightly stronger continuity} for the convergence of the second top mass 
and interval partitions at level $y\vee Y_j^n(\omega)-Y_j^n(\omega)\rightarrow y\vee Y_j(\omega)-Y_j(\omega)$ to obtain for $\bP$-a.e.\ $\omega$
\begin{align*}
  &\bE_{\Gamma^{Y^n_j(\omega)}_n(\omega),j} \left[ f\left(\Gamma^{y\vee Y^n_j(\omega)- Y^n_j(\omega)}\right) 
												 \mathbf{1}\left\{ y\vee Y^n_j(\omega)-Y^n_j(\omega) <Y_{1}\right\}\right]\\
  &\rightarrow \bE_{\Gamma^{Y_j(\omega)}(\omega),j} \left[ f\left(\Gamma^{y\vee Y_j(\omega)- Y_j(\omega)}\right) 
											     \mathbf{1}\left\{ y\vee Y_j(\omega)-Y_j(\omega) <Y_{1}\right\}\right].
\end{align*} 
By Lemma \ref{lemma markovprops}(ii) and applying the previous convergences and dominated convergence, we find 
\begin{equation}\label{conveq}
  \bE \left( f(\Gamma^y_n) \mathbf{1}\{Y^n_j\leq y <Y^n_{j+1}\}\right)\rightarrow \bE \left[ f(\Gamma^y) \mathbf{1}\{Y_j\leq y <Y_{j+1}\}\right].
\end{equation}
A further application of the dominated convergence theorem yields $\bE[f(\Gamma^y_n)]\rightarrow \bE[f(\Gamma^y)]$, completing the proof in the $\cJ^\circ$-valued case.

We now consider the $\cI^\circ$-valued case and suppose that $(0,a_n)\concat(0,b_n)\concat\beta_n \rightarrow (0,a)\concat(0,b)\concat\beta$ with 
$a>0$ and $b\geq 0$, the convergence now being with respect to the $d_\cI$-metric. We emphasize that this is weaker than convergence of the triples 
for the $d_\cJ$-metric and we could have, for example, $a_n \to 0$ and $b_n\to 0$ as sequences of real numbers. By Definition \ref{def:IP:metric}, 
there exist sequences $A_n, B_n \in (0,a_n)\concat (0,b_n)\concat \beta_n$, 
$\alpha_n^{(0)}\in\cI^\circ$ and $\alpha_n^{(1)},\alpha_n^{(2)}\in\cI$ such that 
$$(0,a_n)\concat (0,b_n)\concat \beta_n = \alpha_n^{(0)} \concat A_n \concat \alpha_n^{(1)} \concat B_n \concat \alpha_n^{(2)},$$
with ${\rm Leb}(A_n)\to a$, ${\rm Leb}(B_n)\to b$, $d_\cI(\alpha_n^{(i)},\emptyset) \to 0$ for $i\in \{0,1\}$, and $d_\cI(\alpha_n^{(2)},\beta) \to 0$.
Let $\check{\Psi}_n=(\ff^{(0)}_{n,1},\ff^{(0)}_{n,2},\cev{\fN}^{(0)}_n,\fN^{(0)}_n)\sim\fP_{\alpha_n^{(0)}}^2$, 
$(\ff^{(1)}_{n},\cev{\fN}^{(1)}_n,\fN^{(1)}_{n})\sim\fP_{{\rm Leb}(A_n),\alpha_n^{(1)}}^1$, and $(\ff_{n}^{(2)},\cev{\fN}^{(2)}_n,\fN^{(2)}_n) \sim\fP_{{\rm Leb}(B_n), \alpha_n^{(2)}}^1$ be independent.  Observe that
$$\Psi_n:=\left( \ff_{n,1}^{(0)},\ff^{(0)}_{n,2},\cev{\fN}^{(0)}_n,\fN^{(0)}_n\concat \clade( \ff^{(1)}_n,\cev{\fN}^{(1)}_n) \concat \fN^{(1)}_n \concat \clade(\ff^{(2)}_n,\cev{\fN}^{(2)}_n)\concat \fN^{(2)}_n \right) \sim \fP_{a_n,b_n,\beta_n}^2, $$
and
$$\widetilde{\Psi}_n:=\left( \ff^{(1)}_n,\ff^{(2)}_n,\cev{\fN}^{(2)}_n,\fN^{(2)}_n\right) \sim \fP_{{\rm Leb}(A_n),{\rm Leb}(B_n),\alpha^{(2)}_n}^2.$$
Let $(\check{\gamma}_n^y,y\ge 0)$, $(\gamma_n^y,y\ge 0)$ and $(\widetilde{\gamma}^y_n,y\ge 0)$ be the IP-valued type-2 evolutions constructed from 
$\check{\Psi}_n$, $\Psi_n$ and $\widetilde{\Psi}_n$ by deletion clocking as in Definition \ref{deftype2}, concatenating top mass intervals as in Definition 
\ref{def:type2:v1}. Let $\tau_n = \inf\{ y>0\colon\check{\gamma}_n^y=\emptyset\}$. 
Since ${\rm Leb}(A_n)\to a>0$ and $(\|\check{\gamma}_n^y\|,y\ge 0)\sim\besq_{\|\alpha_n^{(0)}\|}(-1)$ with $\|\alpha^{(0)}_n\|\rightarrow 0$, we have
$\mathbb{P}(\tau_n<\zeta(\ff^{(1)}_n))\to 1$. It is clear from the definitions that the first block (taken from the clock spindle straddling level $y$) of 
$\gamma^y_n$ is the first block of $\check{\gamma}_n^y$ for $0\le y<\tau_n$, whereas it is given by $\ff^{(1)}_n(y)$ for $\tau_n\wedge\zeta(\ff^{(1)}_n)\le y<\zeta(\ff^{(1)}_n)$. 
Furthermore, the conditional distribution of $(\gamma_n^y,\;\tau_n\wedge\zeta(\ff^{(1)}_n)\le y<\zeta(\ff^{(1)}_n))$ given $\check{\Psi}$ only depends on $\tau_n$. 
It is the same as the conditional distribution given $\tau_n$ of
$$\left((0,\ff^{(1)}_n(y))\concat\widehat{\alpha}_n^y\concat\alpha_n^y,\;\tau_n\wedge\zeta(\ff^{(1)}_n)\le y<\zeta(\ff^{(1)}_n)\right),$$
where the three processes are independent, $(\alpha_n^y,y\ge 0)$ is a type-1 evolution, and $(\widehat{\alpha}^y_n,y\ge 0)$ is a type-0 evolution up to level 
$\tau_n$ and then continues as a type-1 evolution. In particular, $(\|\widehat{\alpha}_n^y\|,y\ge 0)$ is a $\besq(1)$ starting from 
$\|\alpha^{(1)}_n\|\rightarrow 0$ up to level $\tau_n$ and then continues as $\besq(0)$. We conclude that
for $\tau_n^\prime=\inf\{y>\tau_n\colon\widehat{\alpha}_n^y=\emptyset\}$, we have 
$$\bP(\tau_n\le\tau_n^\prime<\zeta(\ff^{(1)}_n))\to 1.$$
Since $\gamma_n^y=\widetilde{\gamma}_n^y$  for all $y\ge\tau_n^\prime\ge\tau_n$ on the event $\{\tau_n\le\tau_n^\prime<\zeta(\ff_n^{(1)})\}$ and 
$\bP(\tau_n\le\tau_n^\prime\le y)\to 1$, we find 
$\bP(\gamma_n^y=\widetilde{\gamma}_n^y)\to 1$. This reduces the proof to the case when $a_n\rightarrow a$ and $b_n\rightarrow b$. The argument is now similar to the $\cJ^\circ$-valued case. We decompose as in (\ref{serieseq}) and then apply (\ref{conveq}) to functions of the form 
$$f(a,b,\beta)=\left\{\begin{array}{ll}g((0,a)\concat(0,b)\concat\beta)&\quad\mbox{for $j$ even,}\\
									   g((0,b)\concat(0,a)\concat\beta)&\quad\mbox{for $j$ odd,}
               \end{array}\right.$$
which are continuous for all bounded continuous $g\colon\cI^\circ\rightarrow\bR$, by \cite[Lemma 2.11]{Paper1}.               
\end{proof}


\subsection{The total mass process}\label{sectm}

In this section, we prove Theorem \ref{thm:total_mass}, that the total mass process of a type-2 evolution is a $\besq(-1)$. Our approach is to 
use the $\besq(-1)$ processes $\ff^{(n)}$ and type-1 evolutions $(\fm^{(n)},\gamma^{(n)})$ with $\besq(0)$ total mass, $j\ge 0$. Since the type-2
total mass process is built from the sum of these, the following additivity lemma will be useful. This extends the well-known additivity of
\besq\ processes with nonnegative parameters.

\begin{lemma}\label{lmadd}
 Let $X\sim \besq_a(-1)$, $W \sim \besq_b(0)$ and $\underline{Z}\sim \besq_1(-1)$ be independent. Consider the times 
 $T_X=\inf\left\{t\ge 0\colon X_t=0\right\}$, $T_W=\inf\left\{t\ge 0\colon W_t=0  \right\}$ and $\tau=T_X \wedge T_W$.
 Define a process 
 \[
  V_t=\begin{dcases}
   X_t + W_t, \qquad t \le \tau,\\
   Z_{t - \tau}, \qquad\quad\ t> \tau,
  \end{dcases}
 \]
 where $Z_s=(X_\tau + W_{\tau})\underline{Z}_{s/ (X_\tau+ W_\tau)}$, $s\ge 0$. Then $V \sim \besq_{a+b}(-1)$. 
\end{lemma}
\begin{proof}
 Consider a probability space where all three processes $X, W, \underline{Z}$ are supported. On a standard extension of the sample space, there exist two independent Brownian motions $\beta, \underline{\beta}$ such that 
 \begin{equation}\label{eq:2bm}
 \begin{split}
  d(X_t + W_t) &= -dt + 2 \sqrt{X_t+W_t} d \beta_t, \qquad 0\le t \le \tau, \\
  d \underline{Z}_s &= -ds + 2 \sqrt{\underline{Z}_s} d\underline{\beta}_s, \qquad\quad\ \ s\ge 0.
 \end{split}
 \end{equation}
 Consider the process 
 \[
  B_t=\begin{dcases}
   \beta_t, \qquad\qquad\qquad\qquad\qquad\qquad\quad 0\le t \le \tau, \\
   \beta_\tau + \sqrt{X_\tau+W_\tau}\underline{\beta}_{(t - \tau)/ (X_\tau + W_\tau)}, \quad t > \tau.
  \end{dcases}
 \]
 Then, it follows by L\'evy's characterization of Brownian motion that $B$ is a standard one-dimensional Brownian motion. 
 
 On this same probability space consider the strong solution of the stochastic differential equation (SDE)
 \[
  dU_t = - dt + 2\sqrt{U_t} dB_t, \quad U_0=a+b.
 \]
 It is well-known \cite[Chapter XI]{RevuzYor} that the above SDE has a strong solution that is pathwise unique. Obviously, $U\sim \besq_{a+b}(-1)$. However, it is clear from \eqref{eq:2bm} that the process $V$ also satisfies the relation $dV_t = - dt + 2\sqrt{V_t} dB_t$. Hence, by pathwise uniqueness, $V=U$, almost surely. Thus $V\sim \besq_{a+b}(-1)$.
\end{proof}

\begin{proof}[Proof of Theorem \ref{thm:total_mass}]
 Consider a type-2 evolution $((m_1^y,m_2^y,\beta^y),y\ge 0)$ as constructed in Definition \ref{def:type2:v1}, from initial state $(a,b,\beta)\in\cJ^\circ$. If $(a,b,\beta)$ equals $(a,0,\emptyset)$ or $(0,b,\emptyset)$ then the result is trivial from the construction, so assume not. Then by Proposition \ref{type2welldef}, there is a.s.\ some finite $K\ge 0$ for which the evolution dies at time $Y_{K+1}$. During the interval $[Y_{K},Y_{K+1})$, there is a \em degeneration time \em 
$D=\inf\{y\ge 0\colon(m_1^y,\beta^y)=(0,\emptyset)\mbox{ or }(m_2^y,\beta^y)=(0,\emptyset)\}$ when the type-1 evolution $\gamma^{(K)}$ dies while the top block
$\ff^{(K)}$ continues to live until $Y_{K+1}$. 
 
 By the strong Markov property and Definition \ref{def:type2:v1}, after time $D$, the type-2 evolution comprises a single non-zero component $m_i^y$, with $i$ being either 1 or 2, evolving as a \BESQ[-1] until its absorption at zero. Let $\underline{Z}$ denote the $\besq_1(-1)$ process obtained by applying \BESQ\ scaling to normalize mass of this component at degeneration: $\underline{Z}_y := (m_i^{D})^{-1}m_i^{D+m_i^{D}y}$, $y\ge0$. By the strong Markov property, $\underline{Z}$ is independent of the type-2 evolution run up until time $D$.
 
 We define $D_n:=\min\{Y_n,D\}$, $n\ge 0$, so that $D_n=D$ for $n$ sufficiently large, and set
  $$V_y:=m_1^y+m_2^y+\|\beta^y\|,\quad
  V_y^{(n)}:=\left\{\begin{array}{ll}
  		V_y&\mbox{if } y\le D_n,\\
        Z^{(n)}_{y-D_n}&\mbox{if } y>D_n,
    \end{array}\right.
      \quad\mbox{where }Z_s^{(n)}=V_{D_n} \underline{Z}_{s/V_{D_n}},\,s\ge 0.$$
 We will show inductively that all $V^{(n)}$, $n\ge 1$, and hence the a.s.\ limit $V=\lim_{n\rightarrow\infty}V^{(n)}$, are $\besq_{a+b+\|\beta\|}(-1)$.
 
 For $n=1$, we have $V_y=X_y+Y_y$, $0\le y\le D_1$, where $X=\ff^{(0)}\sim\besq_{a}(-1)$ and $Y=\fm^{(0)} + \|\gamma^{(0)}\|$ independent, and with $D_1=\min\{T_X,T_Y\}$ as in Lemma \ref{lmadd}. Since $Y\sim\besq_{b+\|\beta\|}(0)$ by Proposition \ref{type1totalmass}, Lemma \ref{lmadd} yields $V^{(1)}\sim\besq_{a+b+\|\beta\|}(-1)$.

 Now, assume for induction that for some $n\ge1$, $\widehat{V}^{(n)}\sim\besq_{\widehat{a}+\widehat{b}+\|\widehat{\beta}\|}(-1)$ for all type-2 evolutions  $((\widehat{m}_1^y,\widehat{m}_2^y,\widehat{\beta}^y),y\ge 0)$ starting from any $(\widehat{a},\widehat{b},\widehat{\beta})\in\cJ^\circ$.  By the strong Markov property, we can apply the inductive hypothesis to $(\widehat{m}_1^y,\widehat{m}_2^y,\widehat{\beta}^y) := (m_1^{D_1+y},m_2^{D_1+y},\beta^{D_1+y})$, $y\ge 0$, on the event $\{Y_1=D_1\} = \{D>Y_1\}$.  Then $\widehat D_n = D_{n+1}-D_1$ and $\widehat{\underline{Z}}=\underline{Z}$. We see that
 \begin{align*}
  V_y^{(n+1)} &= \left\{\begin{array}{ll}
  		V_y&\mbox{if } y\le D_{n+1},\\
		Z^{(n+1)}_{y-D_{n+1}}&\mbox{if } y>D_{n+1},
	\end{array}\right.\\
  &= \left\{\begin{array}{ll}
   		V_y&\mbox{if } y\le D_1,\\
		\widehat{V}_{y-D_1}&\mbox{if }D_1<y\le D_1+\widehat{D}_n\\
    	\widehat{Z}^{(n)}_{y-D_1-\widehat{D}_n}&\mbox{if } y>D_1+\widehat{D}_n,
    \end{array}\right\}       
    = \left\{\begin{array}{ll}
      	V_y&\mbox{if } y\le D_1,\\
		\widehat{V}^{(n)}_{y-D_1}&\mbox{if } y>D_1.\end{array}\right.
 \end{align*}
 By the inductive hypothesis, $\widehat{V}^{(n)}\sim\besq_{m_1^{D_1}+m_2^{D_1}+\|\beta^{D_1}\|}(-1)$, and by the strong Markov property and \BESQ\ scaling, $((\widehat{V}^{(n)}_0)^{-1}\widehat{V}^{(n)}_{s\widehat{V}^{(n)}_0},s\ge 0)\sim\besq_1(-1)$ is unconditionally independent of $\cF^{D_1}$, and hence of $((X_y,Y_y),0\le y\le D_1)$.
  Then, by the $n=1$ case already established, we conclude that $V^{(n+1)}\sim\besq_{a+b+\|\beta\|}(-1)$, as required.
\end{proof}

\subsection{The Markov-like property of type-2 data quadruples}\label{sec:Markovish}

We can extend the definition of cutoff data from type-1 evolutions, as seen before Lemma \ref{lem:type1Markov}, to type-2 evolutions.

\begin{definition}\label{def:type2:cutoff}
 In the setting of Definition \ref{deftype2}, for $j\ge 0$ even and $y\in [Y_j,Y_{j+1})$,
 \begin{equation*}
  \big(\ff_{\Psi,1}^y,\cev\fN_\Psi^y,\fN_\Psi^y\big) := \left(\ff_{\Psi_{j}}^{y-Y_{j}},\cev\fN_{\Psi_{j}}^{y-Y_{j}},\fN_{\Psi_{j}}^{y-Y_{j}}\right)\!, \quad 
  		\ff_{\Psi,2}^y := \left(\ff^{(j)}(y\!-\!Y_{j}\!+\!z),\,z\!\ge\!0\right)\!.
 \end{equation*}
 We make the same definition for $j\ge1$ odd, but with subscripts `1' and `2' reversed. We also write 
 $\Psi^y=(\ff_{\Psi,1}^y,\ff_{\Psi,2}^y,\cev\fN_{\Psi}^y,\fN_\Psi^y)$ for the cutoff data quadruple.
\end{definition}

Recall notation $J(y)$, denoting the number of clock changes, and $I(y)$, denoting the index of the clock, for $y\ge0$:
\begin{equation}\label{eq:IJ_index_def}
 J(y) = \inf\{j\ge0\colon Y_{j+1}>y\},\qquad I(y) = \text{1 if $J(y)+1$ is odd, or } I(y)=\text{2 if even}.
\end{equation}
In light of the previous definition, \eqref{eq:type_2_def} can be rewritten as
\begin{equation*}
 m_{I(y)}^y = \ff_{\Psi,I(y)}^y(0), \qquad \alpha^y = \skewer\left(0,\fN_{\Psi}^y\right).
\end{equation*}


It should be clear from the independence of the $\besq(-1)$ top mass processes that $T_j^+=\infty$ may happen for any $j\ge 1$. As a consequence of the argument of the proof of Lemma
\ref{type2welldef}, it will, in fact, happen for some random finite $j\ge 1$, in such a way that the $(j-1)^{\text{st}}$ type-1 evolution of Definition \ref{deftype2} vanishes at a level strictly below the last top mass process. We denote these extinction levels by $\zeta_i^+=\inf\{y\ge 0\colon m_i^z=0\mbox{ for all }z\ge y\}$, $i=1,2$. We write $D=\min\{\zeta_1^+,\zeta_2^+\}$ and $\zeta=\max\{\zeta_1^+,\zeta_2^+\}$. We call level $\zeta$ the \emph{lifetime} of the type-2 evolution and level $D$ its \emph{degeneration time}.



Consider a type-2 data quadruple $(\ff_1,\ff_2,\cev{\fN},\fN_\beta)$. Recall the definition above Lemma \ref{lem:type0Markov} of the point process $\fN^{\le y}$ of spindles below level $y$, based on the type-0 data pair 
$(\cev{\fN},\fN_\beta)$. In the context of type-2 data, we denote the right-continuous natural filtration of $(\ff_1(y),\ff_2(y),\fN^{\le y})$, $y\ge 0$ 
by $(\cF^y,y\ge 0)$, again abusing notation to suppress the dependence on type 2. 

The cases when $b=0$ and $\beta=\emptyset$, or when $a=0$ and $\beta=\emptyset$, are one-dimensional since no non-trivial type-1 point data triple is ever formed in Definition \ref{deftype2}. We therefore have 
$(m_1^y,m_2^y,\alpha^y)=(\ff_1(y),0,\emptyset)$, $y\ge 0$, or $(m_1^y,m_2^y,\alpha^y)=(0,\ff_2(y),\emptyset)$, $y\ge 0$, respectively, and this degenerate type-2 evolution inherits the Markov 
property from $\besq(-1)$. For other initial states, we establish a Markov-like property of a form similar to Lemmas \ref{lem:type0Markov} and \ref{lem:type1Markov}.

Throughout this section, $(\ff_1,\ff_2,\cev\fN,\fN_{\beta}) \sim \fP^2_{a,b,\beta}$, where $a,b\ge 0$, $a+b>0$, and $\beta\in\cI$. Following Definition \ref{deftype2}, let 
$\fN_* := \clade(\ff_2,\cev\fN)\concat\fN_{\beta}$ and $\fX_* := \xi(\fN_*)$. From \cite[proof of Proposition 5.11]{Paper1}, the local time process associated with $\fX_*$, denoted by $(\ell^y_{\fX_*}(t);\,y,t\ge 0)$, is a.s.\ continuous in both level and time coordinates. Therefore, for the cutoff processes $\fN_*^y$ and $\fN_*^{\le y}$ and their associated 
scaffolding processes 
$\fX_*^y:=\xi(\fN_*^y)$ and $\fX_*^{\le y}:=\xi(\fN_*^{\le y})$, we can define local times $(\ell^0_{\fX_*^y}(t),\,t\ge0)$ and $(\ell^y_{\fX_*^{\le y}}(t),\,t\ge 0)$ by extending continuously, approaching level $y$ from above and below, respectively. Moreover,
\begin{equation}
 \ell^y_{\fX_*^{\le y}}(\phi(t)) = \ell^y_{\fX_*}(t) = \ell^0_{\fX_*^y}(t-\phi(t)) \quad\text{where} \quad
 \phi(t) := \text{Leb}\{s\le t\colon \fX_*(s)\le y\}.\label{eq:cutoff_LT}
\end{equation}


\begin{proposition}\label{prop:type2Markov} 
  Let $\Psi=(\ff_1,\ff_2,\cev{\fN},\fN_\beta)\sim\fP_{a,b,\beta}^2=\besq_{a}(-1)\otimes\besq_{b}(-1)\otimes\fP_\beta^0$ for some $a,b\in[0,\infty)$ and $\beta\in\cI$, so that at least two of $a$, $b$ and $\|\beta\|$ are strictly positive. For $y\ge 0$, given $\cF^y$, 
  $(\ff^y_{\Psi,1},\ff^y_{\Psi,2},\cev{\fN}^y_\Psi,\fN^y_\Psi)$ has conditional distribution 
  $\fP_{m^y_1,m^y_2,\alpha^y}^2=\besq_{m^y_1}(-1)\otimes\besq_{m^y_2}(-1)\otimes\fP_{\alpha^y}^0$. 
\end{proposition}
We prove this by way of the following.
\begin{lemma}\label{lem:LTs_meas_in_lvl}
 Fix $y>0$. On $\{\zeta>y\}$, let
 \begin{equation*}
  J:=J(y) = \inf\{ j \geq 0\colon Y_{j+1} > y \}, \quad S := T_{J}^+, \quad T := \inf\{t > T_{J+1}^- \colon \fX_*(t)\ge y\}.
 \end{equation*}
 On $\{\zeta\le y\}$, let $S=T=\infty$. Then the local times $\ell^y_{\fX_*}(S)$ and $\ell^y_{\fX_*}(T)$ are measurable in $\cF^y$.
\end{lemma}
\begin{proof}
 Consider the cutoff processes
 $$\Psi^{\le y} := ((\ff_1(z),z\le y),(\ff_2(z),z\le y),\cev\fN^{\le y},\fN_\beta^{\le y}), \qquad \fX_*^{\le y} := \xi\big( \fN_*^{\le y} \big),$$
 where $\fN_\beta^{\le y}$ is associated with the type-1 point measure $\fN_\beta$, and $\cev\fN^{\le y}$ is such that $\cev\fN^{\le y}\star\fN_\beta^{\le y}=\fN^{\le y}$ with $\fN^{\le y}$ 
 associated with the type-0 data $(\cev\fN,\fN_\beta)$ as in Section \ref{prel}. Now, suppose we apply the construction of Definition \ref{deftype2} to $\Psi^{\le y}$ in place of $\Psi$. For clarity, we refer to the times and levels associated with this construction on $\Psi^{\le y}$ as $U_{j}^+$, $U_{j}^-$ and $Z_{j}$, for $j\ge 0$, rather than $T_{j}^+$, $T_{j}^-$ and $Y_{j}$, which are associated with the construction on $\Psi$.
 
 Let $\phi$ be as in \eqref{eq:cutoff_LT}. We now show by induction that: (i) $U_{j-1}^+ = \phi(T_{j-1}^+)$, (ii) $Z_j = \min\{Y_j,y\}$, and (iii) $U_j^- = \phi(T_j^-)$ for all $1\le j \le J+1$. By definition, $ U_1^- = \phi(T_1^-) =U_0^+ = \phi(T_0^+) =0$ and $Z_1=\min\{Y_1,y\}$. If $J = 0$, then this completes the proof. Otherwise, assume the assertion holds up to some index $j\le J$. Then $\shiftrestrict{\fX_*}{[T_j^-,T_j^+)} = \shiftrestrict{\fX_*^{\le y}}{[U_j^-,U_j^+)}$, as, by our hypotheses, $\phi(T_j^-) = U_j^-$ and $\fX_*$ is bounded above by $Y_j\le y$ on this interval. This implies $Z_{j+1} = \min\{Y_{j+1},y\}$ and $U_j^+ = \phi(T_j^+)$. Then $T_{j+1}^-$ is the time of the next return of $\fX_*$ to level $Y_j = Z_j$, so $\phi(T_{j+1}^-)$ is the time of next return of $\fX_*^{\le y}$ to this level, which equals $U_{j+1}^-$, as desired. 
 
 By \eqref{eq:cutoff_LT}, $\ell^y_{\fX_*}(S) = \ell^y_{\fX_*^{\le y}}(\phi(T_{J}^+)) = \ell^y_{\fX_*^{\le y}}(U_{J}^+)$. 
 Finally, $J = \inf\{ j \geq 0\colon Z_{j+1} = y \}$ a.s.. Thus, $\ell^y_{\fX_*}(S)$ is measurable in $\cF^y$. Similarly, $\phi(T) = \inf\{t > U_{J+1}^-\colon \fX_*^{\le y}(t) = y\} =: T'$. Thus, $\ell^y_{\fX_*}(T) = \ell^y_{\fX_*^{\le y}}(T')$ is also measurable in $\cF^y$.
\end{proof}
\begin{proof}[Proof of Proposition \ref{prop:type2Markov}]
 In the case $y=0$, the assertion is trivial. Fix $y>0$. Consider the type-1 data $(\ff_2,\cev\fN,\fN_{\beta})$ and associated point process $\fN_* = \clade(\ff_2,\cev\fN)\concat\fN_{\beta}$. Let $((m_*^z,\alpha_*^z),\,z\ge0)$ denote the resulting type-1 evolution. Let $\fN_*^y$ denote the cutoff process, as in Lemma \ref{lem:type1Markov}.
 We restrict to the a.s.\ event that $\fN_*$ behaves \emph{nicely} about level $y$, in the sense that no two excursions about the level occur at the same local time.
 
 Let $J$, $S$, and $T$ be as in Lemma \ref{lem:LTs_meas_in_lvl}. We first work on the event $\{y < D,\ J \ge 1,\ J\text{ is even}\}\in\cF^y$.  Theorem 37 of \cite{Paper0} asserts that, if a spindle occurs at time $t$ and survives to level $y$, then the corresponding block in the level $y$ skewer occurs at diversity $\ell^y_{\fX_*}(t)$. Thus, $m_1^y$ and $m_2^y$ are respectively the masses of the unique blocks $V_1,V_2\in (0,m_*^y)\concat\alpha_*^y$ for which $\sD_{(0,m_*^y)\concat\alpha_*^y}(V_1) = \ell^y(S)$ and $\sD_{(0,m_*^y)\concat\alpha_*^y}(V_2) = \ell^y(T)$. Finally, $\alpha^y$ corresponds to the set of blocks of $\alpha_*^y$ that have diversity greater than $\ell^y(T)$ to their left:
 $$\alpha_*^y = \{U\in\alpha_*^y\colon \sD_{\alpha_*^y}(U) \le \ell^y(T)\} \concat \alpha^y.$$
 In particular, from Lemma \ref{lem:LTs_meas_in_lvl} and the property that the type-1 evolution is adapted, we find that $(m_1^y,m_2^y,\alpha^y)$ is measurable in $\cF^y$.
 
 Analogously to the discussion of $m_1^y$ and $m_2^y$, since $J$ is even, $\ff_{\Psi,1}^y = (\ff^{(J)}(y-Y_{J}+z),\,z\ge0)$ and $\ff_{\Psi,2}^y$ are cut off the spindles that cross level $y$ at times $S$ and $T$, 
 respectively. In other words, $\ff_{\Psi,1}^y$ and $\ff_{\Psi,2}^y$ are cut off the middle spindles of the excursions about $y$ at local times $\ell^y_{\fX_*}(S)$ and $\ell^y_{\fX_*}(T)$, respectively. Let 
 $R := \inf\{t>T\colon \fX_*(t) = y\}$. Then, by Definition \ref{def:type2:cutoff},
 \begin{equation*}
  \cev\fN^y_\Psi = \RestrictShift{\cev\fN}{(-\infty,T_{\fX_*(T)}(\cev\fN)]}\concat\RestrictShift{\fN^y_*}{(T-\phi(T),R-\phi(R))} \qquad \text{and} \qquad
  \fN^y_\Psi = \ShiftRestrict{\fN^y_*}{[R-\phi(R),\infty)},
 \end{equation*}
 where $\phi$ is as in \eqref{eq:cutoff_LT}. Proposition 5.6 of \cite{Paper1} implies that, given $\cF^y$, the cutoff process $\fN^y_*$ is conditionally distributed as
 $\ConcatIL_{U\in (0,m_*^y)\concat\alpha_*^y}\fN_U$, 
 where each $\fN_U$ is a clade distributed as $\clade(\ff_U,\cev\fN)$, with $\ff_U\sim\besq_{\text{Leb}(U)}(-1)$ independent of $\cev\fN$, and these clades are all conditionally independent given $(0,m_*^y)\concat\alpha_*^y$. Similarly, for $\{y<D,J\ge 1,J\text{ is odd}\}$, the same argument applies, with roles of 1
 and 2 swapped. On $\{J=0\}=\{Y_1>y\}$, we can just apply the Markov-like property for type-1 evolutions and the Markov property of $\besq(-1)$. We conclude from this and the previous two paragraphs that $(\ff_{\Psi,1}^y,\ff_{\Psi,2}^y,\cev\fN_\Psi^y,\fN_\Psi^y)$ has the claimed conditional law.
\end{proof}

\subsection{Type-2 evolutions via interweaving two type-1 point measures}\label{sec:interweaving}

In this section we present another construction of type-2 evolutions from initial states in which the interval partition component is an independent multiple of a $\PDIP[\frac12,\frac12]$ random variable. 
Such interval partitions appear as pseudo-stationary distributions of type-0 and type-1 evolutions, and indeed, we will use this construction to study pseudo-stationarity properties of type-2 evolutions, including projections of type-2 evolutions to three-mass processes that only retain the evolution of the two top masses and the total mass of the interval partition.

Consider independent $A$ and $B$ for which $\bP(A+B>0)=1$. Also consider independent 
$C_1,C_2\sim\GammaDist(\frac12,\gamma)$ and $\overline{\beta}_1,\overline{\beta}_2\sim{\tt PDIP}(\frac12,\frac12)$ independent of $(A,B)$. Let 
$(\ff_1,\cev\fN_1,\fN_{\beta_1})$ and $(\ff_2,\cev\fN_2,\fN_{\beta_2})$ be two independent type-1 data triples with $\ff_1(0)=A$, $\ff_2(0)=B$, 
$\beta_1=C_1\overline{\beta}_1$ and $\beta_2=C_2\overline{\beta}_2$. Let $\fN_1 := \clade(\ff_1,\cev\fN_1)\concat\fN_{\beta_1}$, and 
correspondingly define $\fN_2$. We will combine these to define a process $((\td m_1^y,\td m_2^y,\td\alpha^y),\,y\ge 0)$ that we will show is a 
type-2 evolution. This \emph{interweaving} construction is illustrated in Figure \ref{fig:interweaving}.

\begin{figure}
 \centering
 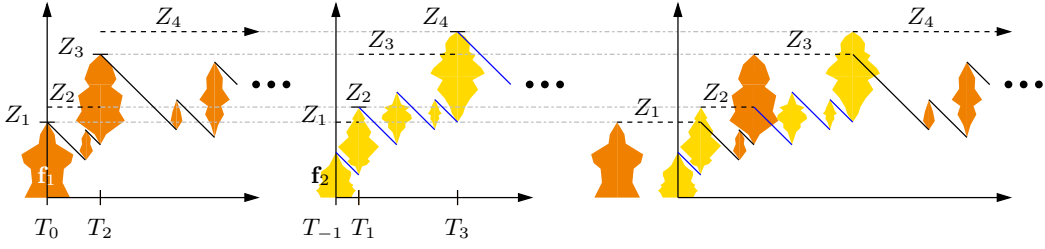
 \caption{Interweaving can be thought of as alternating intervals, $\shiftrestrict{\fN_i}{(T_{j-2},T_j]}$ from two type-1 scaffoldings with spindles, $(\fN_1,\fN_2)$. We begin with a single spindle, $\ff_1$, from $\fN_1$. Then, we include an interval from $\fN_2$ until the time $T_1$ at which its scaffolding exceeds the death level $Z_1 = \zeta(\ff_1)$, reaching some higher level $Z_2$. To this, we concatenate an interval from $\fN_1$ until the time $T_2$ at which its scaffolding exceeds level $Z_2$, reaching some higher level $Z_3$, and so on.\label{fig:interweaving}}
\end{figure}

Let $\fX_1 := \xi(\fN_1)$ and $\fX_2 := \xi(\fN_2)$. We set $T_{-1} := T_0 := 0$, $Z_0 := 0$ and $Z_1 := \zeta(\ff_1)$. For $i\ge 1$ we define
\begin{equation}\label{eq:interweave_vars}
\begin{split}
 T_{2i-1} := \inf\{t\ge T_{2i-3}\colon \fX_2(t) > Z_{2i-1}\}, &\qquad Z_{2i} := \fX_{2}(T_{2i-1}),\\
 T_{2i} := \inf\{t\ge T_{2i-2}\colon \fX_1(t) > Z_{2i}\}, &\qquad Z_{2i+1} := \fX_{1}(T_{2i}),
\end{split}
\end{equation}
with the conventions that $\inf(\emptyset) = \infty$ and $\fX_1(\infty)=\infty$ and $\fX_2(\infty)=\infty$. Also note that this includes setting $T_1=0$ if $\zeta(\ff_2)>\zeta(\ff_1)$. Let $p$ denote the \emph{parity} map, sending even numbers to 2 and odd numbers to 1. For $y\ge 0$ we define
\begin{equation*}
 \td I(y) := p\big(\inf\{j\ge 0\colon Z_{j+1} > y\}\big), \qquad \td J(\infty) := \inf\{j\ge 1\colon T_{j} = \infty\},
\end{equation*}
\begin{equation}\label{eq:interweaving_skewer}
 \left(0,\td m_{3-\td I(y)}^y\right) \concat \left(0,\td m_{\td I(y)}^y\right) \concat\td\alpha^y
 :=\ (0,\ff_1(y))\concat (0,\ff_2(y))\concat\widetilde{\theta}(y)
\end{equation}
where $\displaystyle\widetilde{\theta}(y)=\skewer\!\left(y - \zeta(\ff_2), \Restrict{\fN_{2}}{(0,T_{1}]\times\cE}\right)\concat\Concat_{2\le j\le\td J(\infty)}\! \skewer\!\left(y - Z_{j-1}, \ShiftRestrict{\fN_{p(j+1)}}{(T_{j-2},T_{j}]\times\cE}\right)$.
By this we mean that, (i) if the expression on the right of \eqref{eq:interweaving_skewer} has a leftmost block (note that this equals $(0,\ff_1(y))$ if and only if $y < \zeta(\ff_1)$), then we take $\td m_{3-\td I(y)}^y$ to denote the mass of this block, otherwise setting $\td m_{3-\td I(y)}^y :=0$; and (ii) if said expression has a second-to-leftmost block, then we denote its mass by $\td m_{\td I(y)}^y$, otherwise setting $\td m_{\td I(y)}^y := 0$. Then $\td\alpha^y$ denotes what remains of $\widetilde{\theta}(y)$ after removing leftmost blocks as required to form $\widetilde{m}_1^y$ and $\widetilde{m}_2^y$, and, if necessary, shifting the remaining interval partition down to line up with 0 on its left end.

\begin{proposition}\label{interweaving}
 The process $((\td m_1^y,\td m_2^y,\td\alpha^y),y\ge 0)$ defined in \eqref{eq:interweaving_skewer} is a type-2 evolution with initial state  
 $(\td m_1^0,\td m_2^0,\td\alpha^0)=(A,B,C\overline{\beta})$, where $A$, $B$, $C$ and $\overline{\beta}$ are jointly independent, with 
 $C\sim\GammaDist(\frac12,\gamma)$ and $\overline{\beta}\sim\PDIP(\frac12,\frac12)$.
\end{proposition}

Before we prove this proposition, we recall a simpler construction of pseudo-stationary type-1 data triples that does not require concatenating 
infinitely many clades.

\begin{proposition}[Corollary 4.28(ii) of \cite{Paper1}]\label{prop:type1:pseudo_constr}
 Fix $\gamma>0$.  Let $\fN$ denote a \PRM[\Leb\otimes\mBxc] independent of $S\sim\ExpDist[\gamma]$. We define $T := \inf\{t > 0\colon M^0_{\fN}(t) > S\}$, where $M_{\fN}^0$ is the aggregate mass process of \eqref{eq:skewer_def}. Then $\beta := \skewer\big(0,\restrict{\fN}{[0,T)}\big)$ is a \PDIP[\frac12,\frac12] scaled by an independent \GammaDist[\frac12,\gamma], and, recalling the notation above Lemma \ref{lem:type0Markov}, $\big(\restrict{\fN}{[0,T)}\big)^0$ is a type-1 point measure with initial state $\beta$. Moreover, $\big(\cev\fN,\big(\restrict{\fN}{[0,T)}\big)^0\big)$ is pseudo-stationary type-0 data with \GammaDist[\frac12,\gamma] initial mass.
%
\end{proposition}

In the setting of this construction, we write $\fN_{\beta} := \big(\restrict{\fN}{[0,T)}\big)^0$. Now, let $\ff$ denote a \BESQ[-1] independent of the other objects, with any random initial mass, and define $\fN_* := \clade(\ff,\cev\fN)\concat\fN_{\beta}$. In the special case that $\ff(0)\sim\GammaDist[\frac12,\gamma]$, the measure $\fN_*$ describes a pseudo-stationary type-1 evolution with \ExpDist[\gamma] initial mass, as in Proposition \ref{prop:type1:pseudo_g}. For any distribution of $\ff(0)$, this construction has the following consequence, by way of the Poisson property of $\fN$ and the memorylessness of $S$.

\begin{lemma}[Memorylessness for some type-1 point measures]\label{lem:pseudostat:memoryless}
 Fix $\gamma>0$ and let $\fN_*$ be as above. Let $R$ be a stopping time in the right-continuous time filtration $(\cF_t,\,t\ge0)$ generated by $\fN_*$, i.e.\ the least right-continuous filtration in which $\restrict{\fN_*}{[0,t]}$ is $\cF_t$-measurable for every $t\ge 0$. Given $\restrict{\fN_*}{[0,R]\times\cE}$ with $\xi_{\fN_*}(R) = x$, and further conditioning on 
 $\{\restrict{\fN_*}{(R,\infty)\times\cE}\neq 0\}$, the conditional distribution of $\shiftrestrict{\fN_*}{(R,\infty)\times\cE}$ equals 
 the (unconditioned) distribution of $\ShiftRestrict{\cev\fN}{[T_{x}(\cev\fN),0)\times\cE} \concat \fN_{\beta}$.
\end{lemma}

\begin{proof}[Proof of Proposition \ref{interweaving}]
 Let $(\ff_1,\ff_2,\cev\fN,\fN_{\beta})$ be data for a type-2 evolution $\big((m_1^y,m_2^y,\alpha^y),\,y\ge0)$ starting from the initial distribution as claimed. We follow the notation of Definition \ref{deftype2} and \eqref{eq:IJ_index_def}. Additionally, we define $J(\infty) := \inf\{j\ge 1\colon T_{j}^+ = \infty\}$.
 We prove our assertion by showing
 \begin{gather}
  \!\!\!\!\!\left(\ff_1,\ff_2,\!\left( \ShiftRestrict{\fN_{p(i+1)}}{(T_{i-2},T_{i}]\times\cE}, Z_{i} \right)\!,1\!\le\! i\!\le\!\td J(\infty) \right) \stackrel{d}{=} 
  	\left( \ff_1,\ff_2,\!\left( \ShiftRestrict{\fN_{*}}{(T^-_i,T^+_i]\times\cE}, Y_{i} \right)\!,1\le i\!\le\!J(\infty) \right)\label{eq:inter_clocking_intervals}\\
  \text{and}\quad 
  \left(0,m_{I(y)}^y\right)\concat \left(0,m_{3-I(y)}^y\right) \concat \alpha^y = (0,\ff_1(y))\concat (0,\ff_2(y))\concat \theta(y)\quad\text{for all }y\ge0,\label{eq:alt_deletion_clocking}
 \end{gather}
 where $\displaystyle\theta(y) := \skewer\left(y-\zeta(\ff_2), \Restrict{\fN_{*}}{(0,T^+_1]\times\cE} \right)\concat\Concat_{2\le i\le J^{\infty}} \skewer\left(y-Y_{i-1}, \ShiftRestrict{\fN_{*}}{(T^-_i,T^+_i]\times\cE} \right)$.\\
 These formulas, together with \eqref{eq:interweaving_skewer}, complete the proof.
 
 First, we prove \eqref{eq:inter_clocking_intervals}. For $i\ge 1$, we note the equality of events
 \begin{equation}\label{eq:last_clock_adapted}
  \left\{\td J(\infty)\! =\! i\right\} = \left\{T_{i}\! =\! \infty;\,\td J(\infty)\! \ge\! i\right\} = \left\{\sup\nolimits_t\xi_{\shiftrestrict{\fN_{p(i+1)}}{(T_{i-2},T_{i}]\times\cE}}(t) < Z_i\!-\!Z_{i-1};\,\td J(\infty)\! \ge\! i\right\}\!.
 \end{equation}
 We conclude, by a recursive argument, that the indicator $\mathbf{1}\{\td J(\infty) \le j\}$ is a function of the $1\le i\le j\wedge \td J(\infty)$ terms on the left in \eqref{eq:inter_clocking_intervals}. By a corresponding argument, the indicator $\mathbf{1}\{J(\infty) \le j\}$ is a function of the $1\le i\le j\wedge J(\infty)$ terms on the right.
 
 We now establish the base case for an induction. By definition, $T_{-1} = T_1^- = 0$, $Z_1 = Y_1$, and $\fN_2\stackrel{d}{=}\fN_*$. Recall from \eqref{eq:interweave_vars} that $T_1$ is the time when $\fX_2$ first exceeds $Z_1$, while $T_1^+$ in \eqref{eq:d_clk:levels} is the time when $\fX_*$ exceeds $Y_1$. This proves equality in distribution for the $i=1$ terms of \eqref{eq:inter_clocking_intervals}.
 
 Assume for induction that, for some $j\ge 1$, \eqref{eq:inter_clocking_intervals} holds when we substitute $j\wedge\td J(\infty)$ for the $\td J(\infty)$ bound on the left and substitute $j\wedge J(\infty)$ for $J(\infty)$ on the right. By the argument following \eqref{eq:last_clock_adapted}, $\bP\{\td J(\infty) \le j\} = \bP\{J(\infty) \le j\}$. 
 We now show that the conditional distribution of the $(j+1)^{\text{st}}$ term on the left in \eqref{eq:inter_clocking_intervals}, given the preceding terms and the event $\{j < \td J(\infty)\}$, equals the conditional law of the corresponding term on the right given the preceding terms and the event $\{j < J(\infty)\}$. 
 %
 %
 
 Note that
 \begin{equation*}
 \begin{split}
  Z_{j+1} &= \fX_{p(j+1)}(T_{j}) = Z_{j-1} + \xi_{\shiftrestrict{\fN_{p(j+1)}}{(T_{j-2},T_{j}]\times\cE}}(T_{j}\!-\!T_{j-2})
  	=: G\!\left(\!\left(\shiftrestrict{\fN_{p(i+1)}}{(T_{i-2},T_{i}]\times\cE},Z_i\right)\!,i\!\le\!j\right)
 \end{split}
 \end{equation*}
 and $Y_{j+1} = G\left(\left( \ShiftRestrict{\fN_{*}}{(T^-_i,T^+_i]\times\cE}, Y_{i} \right)\!,i\!\le\! j\right)$.
 Next, observe that $\fX_{p(j+1)}(T_{j-2}) = Z_{j-1}$ while, correspondingly, $\fX_*(T_{j}^-) = Y_{j-1}$. Since we have conditioned on $\{j < \td J(\infty)\}$, which means $T_j < \infty$, we may apply Lemma \ref{lem:pseudostat:memoryless} to $\fN_{p(j+1)}$ at this time. In particular, by the independence of $\fN_1$ and $\fN_2$, and by this lemma, given $Z_{j}$, the restricted process $\shiftrestrict{\fN_{p(j+1)}}{(T_{j-2},\infty)\times\cE}$ is conditionally independent of all preceding terms on the left in \eqref{eq:inter_clocking_intervals}. Correspondingly, $\shiftrestrict{\fN_*}{(T_{j}^-,\infty)\times\cE}$ is conditionally independent of all preceding terms on the right in \eqref{eq:inter_clocking_intervals}, given $Y_j$, and these restricted point processes have the same conditional distribution. Finally, $T_{j}^+-T_j^-$ is the first time that $\xi(\shiftrestrict{\fN_{*}}{(T_{j}^-,\infty)\times\cE})$ exceeds $Y_{j}-Y_{j-1}$, and correspondingly for $\fN_{p(j+1)}$. This completes our induction and proves \eqref{eq:inter_clocking_intervals}.
 
 We now prove \eqref{eq:alt_deletion_clocking}. Recall the deletion clocking construction of the type-2 evolution in Definition \ref{deftype2}.
 
 Case 1: $y < \min\{\zeta(\ff_1),\zeta(\ff_2)\}$. Then $J(y) = 0$, $I(y) = 2$, so the two leftmost blocks on the left hand side of \eqref{eq:alt_deletion_clocking} are $(0,m_1^y)\concat (0,m_2^y)$, which equal $(0,\ff_1(y))\concat (0,\ff_2(y))$, as claimed. By definition, $\fX_*$ is bounded below by $Y_{i}\ge Y_1 > y$ on each interval $(T_{i}^+,T_{i+1}^-]$. Therefore,
 \begin{equation}
  \alpha^y = \skewer\left(y - \zeta(\ff_2),\restrict{\fN_*}{(0,\infty)\times\cE}\right) = \theta(y),
 \end{equation}
 as desired. Indeed, $\theta(y)$, as defined following \eqref{eq:alt_deletion_clocking}, simply skips over certain intervals of $\fN_*$ that cannot contribute to the skewer.
 
 Case 2: $\zeta(\ff_2)\le y < \zeta(\ff_1)$. Then, again, $J(y) = 0$ and $I(y) = 2$. As before, $(0,m_1^y) = (0,\ff_1(y))$, in agreement with \eqref{eq:alt_deletion_clocking}. However, now $\ff_2(y) = 0$. Thus,
 \begin{equation*}
  (0,m_2^y)\concat\alpha^y = \skewer(y,\fN_*) = \skewer\left(y - \zeta(\ff_2),\restrict{\fN_*}{(0,\infty)\times\cE}\right) = \theta(y),
 \end{equation*}
 since, as in Case 1, $\theta(y)$ skips over intervals that do not contribute.
 
 Case 3: $\zeta(\ff_1)\le y < \zeta(\ff_2)$. Then $J(y) = I(y) = 1$ and $T_1^+ = 0$. Then $(0,m_{3-I(y)}^y) = (0,m_2^y) = (0,\ff_2(y))$, while $\ff_1(y)=0$, in agreement with \eqref{eq:alt_deletion_clocking}. Moreover,
 \begin{equation*}
  (0,m_1^y)\concat\alpha^y = \skewer(y-Y_1,\shiftrestrict{\fN_*}{(T_2^-,\infty)\times\cE}).
 \end{equation*}
 In this case, since $T_1^- = T_1^+$, the first term in the formula for $\theta(y)$ is empty. Then, the concatenation of subsequent terms in $\theta(y)$ equals the above expression, since $\fX_*$ is bounded below by $Y_i \ge Y_2 > y$ on each interval $(T_{i}^+,T_{i+1}^-]$ with $i\ge 2$.
 
 Case 4: $\max\{\zeta(\ff_1),\zeta(\ff_2)\} \le y$. Then $J(y)\ge 1$ and $T_{J(y)+1}^- > 0$. Moreover, $\ff_1(y)=\ff_2(y)=0$, so all that remains on the right in \eqref{eq:alt_deletion_clocking} is $\theta(y)$. Note that $\fX_*$ is bounded above by $Y_{J(y)} \le y$ on each interval $(T_i^-,T_i^+]$ with $i < J(y)$, as well as on $(T_{J(y)}^-,T_{J(y)}^+)$. Then $\fX_*$ jumps up across level $y$ at time $T_{J(y)}^+$, giving rise to the broken spindle $\ff_{3-I(y)}^{(J(y))}$. Thus, the terms in $\theta(y)$ with $i < J(y)$ do not contribute, and the $i=J(y)$ term contributes only a single block:
 \begin{equation*}
  \left(0,m_{3-I(y)}^y\right) = \left(0,\ff_{3-I(y)}^{(J(y))}(y-Y_{J(y)})\right) = \skewer\left(y-Y_{J(y)-1},\shiftrestrict{\fN_*}{(T_{J(y)}^-,T_{J(y)}^+]\times\cE}\right).
 \end{equation*}
 Then
 \begin{equation*}
  \left(0,m_{I(y)}^y\right) \concat \alpha^y = \skewer\left(y-Y_{J(y)},\shiftrestrict{\fN_*}{(T_{J(y)+1}^-,\infty)\times\cE}\right),
 \end{equation*}
 which equals the concatenation of terms in $\theta(y)$ over $i>J(y)$, since, similarly to the previous cases, this expression skips over intervals where $\fX_*$ is bounded below by $Y_{J(y)+1}>y$.
\end{proof}

After Definition \ref{deftype2} we interpreted the spindles removed during deletion clocking as emigration. Where is the emigration in the interweaving construction?
The interweaving construction is based on two type-1 evolutions (without emigration). The one with the shorter lifetime is completely incorporated into the type-2
evolution, while the one with the longer lifetime will only be incorporated up to the clock spindle that exceeds that shorter lifetime. Following this clock spindle
is a \Stable[\frac32] process with excursions above the minimum that allow an analogous interpretation of emigration as in deletion clocking.

\section{Stationarity and connection to Wright-Fisher processes}\label{dePoiss}

\noindent In this section, we prove Theorem \ref{thm:stationary}, which describes a stationary variant of the type-2 evolution, constructed by normalizing and time-changing (de-Poissonizing) the type-2 evolution and allowing it to jump back into stationarity (resample) instead of being absorbed in a single-block state at degeneration times.


\subsection{Pseudo-stationarity}

Since type-2 evolutions degenerate to a single block of positive mass before reaching zero total mass, pseudo-stationarity results differ from Proposition \ref{prop:type1:pseudo} for types 0 and 1. Specifically, we obtain results conditionally given that degeneration has not yet happened. We furthermore identify the total mass at degeneration conditionally given the time of degeneration.

\begin{proposition}[Pseudo-stationarity of type-2 evolution]\label{prop:type2:pseudo}
 Consider $(A_1,A_2,A_3)\!\sim\!\DirDist[\frac12,\frac12,\frac12]$ and an independent interval partition $\widebar\beta\sim\PDIP[\frac12,\frac12]$, with $M(0)>0$ an independent random variable. Let $((m_1^y,m_2^y,\alpha^y),\,y\ge0)$ denote a type-2 evolution initially distributed as $(M(0)A_1,M(0)A_2,M(0)A_3\widebar\beta)$. Let $(M(y),\,y\ge0)$ denote its total mass process. For $y>0$, given $\{D>y\}$, $M(y)$ is conditionally independent of $(m_1^y/M(y),m_2^y/M(y),\alpha^y/M(y))$. The latter is conditionally distributed according to the (unconditioned) law of $(A_1,A_2,A_3\widebar\beta)$.
\end{proposition}

In light of this result, we refer to the law of $(M(0)A_1,M(0)A_2,M(0)A_3\widebar\beta)$ above as the \emph{pseudo-stationary law for type-2 evolution} with mass $M(0)$. 
Following the proof in \cite{Paper1} of Proposition \ref{prop:type1:pseudo} above, we first prove this for $M(0)\sim\GammaDist[\frac{3}{2},\gamma]$, $\gamma>0$ and then generalize via Laplace inversion. 

\begin{proposition}\label{prop:type2:pseudo_g}
 Consider a type-2 evolution $((m_1^y,m_2^y,\alpha^y),y\ge 0)$ with initial blocks $(m_1^0,m_2^0)$ independent of $\alpha^0 = M\widebar\beta$,
 where $M\sim \GammaDist[\frac{1}{2},\gamma]$ and $\widebar\beta\sim\PDIP(\frac12,\frac12)$ are independent. 
 Then for $y>0$, given $\{D>y\}$, the interval partition $\alpha^y$ is conditionally independent of $(m_1^y,m_2^y)$, conditionally distributed according to the (unconditional) law of $(2y\gamma+1)M\widebar\beta$.
 
 If, additionally, $m_1^0$ and $m_2^0$ are i.i.d.\ \GammaDist[\frac12,\gamma], then given $\{D>y\}$, $m_1^y$ and $m_2^y$ are conditionally i.i.d.\ \GammaDist[\frac12,\gamma/(2y\gamma+1)].
\end{proposition}

\begin{proof}
 Let $(\ff_1,\ff_2,\cev\fN,\fN_{\beta})$ be type-2 data for this evolution. From Proposition \ref{prop:type1:pseudo_constr}, we may assume $\fN_\beta = \big(\restrict{\fN}{[0,T)}\big)^0$, where $\fN$ is a \PRM[\Leb\otimes\mBxc]\ on $[0,\infty)\times\cE$ and $T$ is the time at which the aggregate mass of spindles crossing level 0, as defined in \eqref{eq:skewer_def}, first exceeds an independent mass threshold $S\sim \ExpDist[\gamma]$. 
 
 We follow the notation of \eqref{eq:IJ_index_def}, in which $I(y)$ denotes the index, 1 or 2, of the clock mass at level $y$. So $m_{I(y)}^y$ is the clock mass at that level, $m_{3-I(y)}^y$ is the non-clock top mass at that level, $\alpha^y$ is the interval partition of remaining, ``spinal'' masses, and $\fN_* := \clade(\ff_2,\cev\fN)\concat\fN_{\beta}$. Let $((m_*^z,\alpha_*^z),\,z\ge0)$ denote the type-1 evolution $\skewerbar(\fN_*)$.
 
 It follows from Definition \ref{deftype2} that, on $\{D>y\}$, the non-clock top mass at level $y$ is the mass of a spindle found in $\restrict{\fN_*}{[0,T)}$ at the stopping time $R = \inf\{t>T_{J(y)+1}^- \colon \xi_{\fN_*}(t)>y\}<T$, and the remaining interval partition $\alpha^y$ equals $\skewer(y - \xi_{\fN_*}(R),\restrict{\fN_*}{(R,T)})$. 
 
 Let $R' := \inf\{t>R\colon \xi_{\fN_*}(t)=y\}$ and $T_y(\fN_*) := \inf\{t\ge0\colon \xi_{\fN_*}(t) = y\}$. 
 By Lemma \ref{lem:pseudostat:memoryless}, the conditional law of $\shiftrestrict{\fN_*}{(R',\infty)}$ given $\{D>y\}$ and $(\ff_1,\restrict{\fN_*}{[0,R']})$ equals the conditional law of $\shiftrestrict{\fN_*}{(T_y(\fN_*),\infty)}$ given $\{m_*^y+\|\alpha_*^y\|>0\}$. Passing to the skewers, the correspondingly conditioned laws of $\alpha^y$ and $\alpha^y_*$ are equal. 
 By Proposition \ref{prop:type1:pseudo_g}, this is an independent \GammaDist[\frac12,\gamma/(2y\gamma+1)]\ multiple of a \PDIP[\frac12,\frac12]. This also implies that $\alpha^y$ is conditionally independent of $(m_1^y,m_2^y)$ given $\{D>y\}$, proving the first assertion of the proposition.
 
 To prove the second assertion, we apply Lemma \ref{interweaving}. In the representation there, we have $D = \min\{\zeta_1,\zeta_2\}$. In particular, conditioning on $\{D>y\}$ is the same as conditioning on $\{\zeta_1>y,\zeta_2>y\}$. By Proposition \ref{prop:type1:pseudo_g} and the independence of the two pseudo-stationary type-1 evolutions in that construction, $m^y_1$ and $m^y_2$ are conditionally independent given $\{\zeta_1>y,\zeta_2>y\}$, with common distribution $\GammaDist[\frac{1}{2},\gamma/(2y\gamma+1)]$.
\end{proof}

\begin{proposition}\label{prop:pseudo_f}
 For $a,b,c>0$ and $\widebar\beta\sim\PDIP[\frac12,\frac12]$, consider a type-2 evolution starting from $(a,b,c\widebar{\beta})$. Let $\widebar\beta'$ be an independent $\PDIP[\frac12,\frac12]$, and let $\td\alpha^y$ denote $\alpha^y/\|\alpha^y\|$ when $\alpha^y\neq\emptyset$ (this holds a.s.\ given $y<D$), or $\widebar\beta'$ otherwise. Then for $y>0$, $\td\alpha^y$ is independent of $(m_1^y,m_2^y,\|\alpha^y\|)$ and has law $\PDIP[\frac12,\frac12]$.
\end{proposition}
\begin{proof}
 For $\gamma>0$, consider $B_{\gamma}\sim \GammaDist[\frac12,\gamma]$ independent of all other objects. By decomposing according to the events $\{D>y\}$ and $\{D<y\}$, and applying the first assertion of Proposition \ref{prop:type2:pseudo_g} in the former case, we see that 
 for all continuous $f\colon \BR^3\to[0,\infty)$ and $g\colon\cI^\circ\to[0,\infty)$,
  \begin{align*}
    &\int_0^\infty \sqrt{\frac{\gamma}{\pi x}}e^{-\gamma x}\fE^2_{a,b,x\widebar\beta}\big[ f(m_1^y,m_2^y,\|\alpha^y\|)g(\td\alpha^y) \big]dx
    = \fE^2_{a,b,B_{\gamma}\widebar\beta}\big[ f(m_1^y,m_2^y,\|\alpha^y\|)g(\td\alpha^y) \big]\\
    &\qquad\qquad = \bE\big[g(\widebar\beta)\big]\int_0^\infty \sqrt{\frac{\gamma}{\pi x}}e^{-\gamma x} \fE^2_{a,b,x\widebar\beta}\big[ f(m_1^y,m_2^y,\|\alpha^y\|) \big]dx.
  \end{align*}
  We cancel factors of $\sqrt{\gamma}$ and appeal to the uniqueness of Laplace transforms to find that
  $$\fE^2_{a,b,x\widebar\beta}\big[ f(m_1^y,m_2^y,\|\alpha^y\|)g(\td\alpha^y) \big] = \fE^2_{a,b,x\widebar\beta}\big[ f(m_1^y,m_2^y,\|\alpha^y\|) \big]\bE\big[g(\widebar\beta)\big]$$
  for a.e.\ $x>0$. By Proposition \ref{contini}, the right hand side is continuous in $x$. 
  Note that $(a,b,\beta)\mapsto f(a,b,\|\beta\|)g(\beta/\|\beta\|)\cf\{\beta\neq\emptyset\}$ is $\bP^2_{a,b,x\bar\beta}$-a.s.\ continuous at $(m_1^y,m_2^y,\alpha^y)$. Thus, the left hand side is continuous in $x$ as well; see e.g.\  \cite[Theorem 4.27]{Kallenberg}. We conclude that the above formula holds for every $x$.
\end{proof}

\begin{proof}[Proof of Proposition \ref{prop:type2:pseudo}]
 Let $((m_1^z,m_2^z,\alpha^z),\,z\ge0)$ be as in the statement of the proposition, and fix $y>0$. The conditional law of $(m_1^y,m_2^y,\alpha^y)$ given $D>y$ can be obtained as a mixture, over the law of the vector $(MA_1,MA_2,MA_3)$ of initial masses, of the conditional laws described in Proposition \ref{prop:pseudo_f}. In particular, conditionally given $\{\alpha^y\neq\emptyset\}$, $\alpha^y/\|\alpha^y\|\sim\PDIP[\frac12,\frac12]$, conditionally independent of $(m_1^y,m_2^y,\|\alpha^y\|)$. To prove that $(m_1^y,m_2^y,\|\alpha^y\|)/M(y)$ then has conditional law \DirDist[\frac12,\frac12,\frac12], we make an argument similar to that in the proof of Proposition \ref{prop:pseudo_f}.
 
 Recall the standard beta-gamma algebra that a \DirDist[x_1,\ldots,x_n] vector, multiplied by an independent \GammaDist[x_1+\cdots+x_n,\gamma] scalar, gives rise to a vector of independent variables, with the $j^{\text{th}}$ having law \GammaDist[x_j,\gamma]. Let $(\widetilde m_1^y,\widetilde m_2^y,\widetilde m_3^y)$ denote $(m_1^y/M(y),m_2^y/M(y),\|\alpha^y\|/M(y))$ when $y<D$ or $(A_1',A_2',A_3')$ otherwise, where the latter is an independent \DirDist[\frac12,\frac12,\frac12]. By the second assertion of Proposition \ref{prop:type2:pseudo_g}, for $\gamma>0$ and measurable $f\colon \BR^3\to[0,\infty)$ we have
  \begin{align*}
    \int_0^\infty 2\sqrt{\frac{x\gamma^3}{\pi}}e^{-\gamma x}\fE^2_{A_1x,A_2x,A_3x\widebar\beta}\big[ f(\widetilde m_1^y,\widetilde m_2^y,\widetilde m_3^y) \big]dx
    = \bE\big[f(A_1,A_2,A_3)\big].
  \end{align*}
  Multiplying the right hand side by $\int_0^\infty 2\sqrt{x\gamma^3/\pi}e^{-\gamma x}dx=1$, canceling factors of $\gamma^{3/2}$, and appealing to uniqueness of Laplace transforms and Proposition \ref{contini}, as in the previous proof, gives the desired result.
\end{proof}

For our next results, we require a scaling invariance property of type-2 evolutions. We recall the scaling map for point processes of spindles from \cite[Equation 4.2]{Paper1}: for $c>0$,
\begin{equation}
 c\scaleH N = \sum_{\text{points }(t,f)\text{ of }N}\delta\left(c^{3/2}t,c\scaleB f\right),\quad\mbox{where }c\scaleB f = (cf(y/c),\, y\in\bR).
\end{equation}
For type-2 data $\Psi = (\ff_1,\ff_2,\cev \fN,\fN_{\beta})$, we will write $c\scaleH\Psi$ to denote $(c\scaleB \ff_1,c\scaleB \ff_2,c\scaleH\cev\fN,c\scaleH\fN_{\beta})$. We adopt the convention that $0\scaleH\Psi = (0,0,\cev\fN,0)$, where the first two zeros on the right denote zero functions, and the last a zero measure on $\Exc\times\bR$.

\begin{lemma}\label{lem:scaling}
 Suppose $\Psi=(\ff_1,\ff_2,\cev{\fN},\fN_\beta)$ is a type-2 data quadruple and $M\ge 0$ is a real-valued random variable, conditionally 
 independent of $\Psi$ given the initial state $(\ff_1(0),\ff_2(0),\beta)$ of the type-2 evolution associated with $\Psi$. Then $M\scaleH\Psi$ is 
 also a type-2 data quadruple.
\end{lemma}

\begin{proof}
 This follows from the well-known scaling invariance of squared Bessel processes (see e.g.\ \cite[A.3]{GoinYor03}), as well as that of clades of spindles \cite[Lemma 4.5]{Paper1} and the law $\mBxc(c\scaleB A) = c^{-3/2}\mBxc(A)$ for $A\subseteq\cE$, from \cite[Lemma 3.9]{Paper1}.
\end{proof}

For type-2 data $\Psi$ with $M_{\Psi}(0)=\ff_1(0)+\ff_2(0)+\|\skewer(0,\fN)\|\neq 0$, we adopt the notation $\widebar\Psi := M_{\Psi}(0)^{-1}\scaleH\Psi$. This is data for a type-2 evolution scaled to have unit initial mass. Going in the other direction, if we begin with data $\widebar\Psi$ for a type-2 evolution with pseudo-stationary initial distribution and unit initial mass, then for any independent random $M\ge0$, the quadruple $M\scaleH\widebar\Psi$ is data for a type-2 evolution with pseudo-stationary initial distribution and initial mass $M$. Denote by $\mu_m$ the pseudo-stationary distribution
on $\cI^\circ$ with total mass $m$. 

We will also denote by $\mu_{a,b,c}$ the distribution on $\cI^\circ$ of $(a,b,c\overline{\beta})$, with $\overline{\beta}\sim\PDIP[\frac12,\frac12]$,
for all $(a,b,c)\in[0,\infty)$ with either $a+b>0$ or $a=b=c=0$. We will use notation $\fP^2_\mu$ for the distribution on 
$\cE^2\times\cev{\cN}\times\cN$ of type-2 data with random $\mu$-distributed initial data, for any distribution on $\cI^\circ$.

\begin{lemma}[Strong pseudo-stationarity]\label{pretotal}
 \begin{enumerate}[label=(\roman*), ref=(\roman*)]
  \item Let $\overline{\beta}\sim\PDIP[\frac12,\frac12]$, and let  
  $(A,B,C)$ be an independent vector for which, with probability 1, at least two components are positive. Consider type-2 data $\Psi$ with 
  initial state $(A,B,C\overline{\beta})$. Denote by $(A(y),B(y),C(y))=(m_1^y,m_2^y,||\alpha^y||)$ the associated 3-mass process and by 
  $(\cF_{\rm 3-mass}^y,y\ge 0)$ the right-continuous filtration it generates. Let $Y$ be a stopping time in this filtration. Then for all 
  $\cF^Y_{\rm 3-mass}$-measurable $\eta\colon\Omega\rightarrow[0,\infty)$ and all measurable 
  $H\colon\cE^2\times\cev{\cN}\times\cN\rightarrow[0,\infty)$,
  $$\bE\left(\eta H(\Psi^Y)\right)=\bE\left(\eta\fP_{\mu_{A(Y),B(Y),C(Y)}}^2[H]\right).$$
  I.e.\ the cutoff data above level $Y$ are type-2 data. Conditionally given $\cF_{\rm 3-mass}^Y$, the initial data of $\Psi^Y$ are distributed as  
  $(A(Y),B(Y),C(Y)\overline{\beta}^\prime)$ for independent $\overline{\beta}^\prime\sim\PDIP[\frac12,\frac12]$.  \label{item:pretotal:3}
  \item  Now consider instead type-2 data of the form $\Psi = M(0)\scaleH\widebar{\Psi}$, where $\widebar{\Psi}$ is data for a pseudo-stationary type-2 evolution with unit initial mass, independent of $M(0)$. Denote by $M(y) = m_1^y+m_2^y+\|\alpha^y\|$, $y\ge 0$, the associated total mass process and by $(\cF^y_{\rm mass},\,y\ge 0)$ the right-continuous filtration it generates. Let $Y$ be a stopping time in this filtration. Then for all $\cF^Y_{\rm mass}$-measurable $\eta\colon\Omega\rightarrow[0,\infty)$ and measurable $H\colon\cE^2\times\cev{\cN}\times\cN\rightarrow[0,\infty)$,
  $$\bE\left(\left.\eta H(M(Y)^{-1}\scaleH\Psi^Y)\,\right|\,D>Y\right)=\bE\left(\left.\eta\,\right|\,D>Y\right)\bE\left(H(\overline{\Psi})\right).\vspace{-0.6cm}$$\label{item:pretotal:1}%
 \end{enumerate}
\end{lemma}
\begin{proof}
  For the first assertion, if we further condition on $\{D\le Y\}$ then the statements follow trivially from Proposition \ref{prop:type2Markov}, as 
  $\|\alpha^Y\|=0$. For fixed $Y=y$, conditional on $\{D>y\}$, the first assertion follows from the pseudo-stationarity of the interval partition in 
  Proposition \ref{prop:pseudo_f} and the Markov-like property of Proposition \ref{prop:type2Markov} for type-2 data, in the same manner 
  as in the proof of \cite[Lemma 6.8]{Paper1} for type-0 and type-1 evolutions. The generalization to stopping times is standard; see e.g.\ the 
  proof of \cite[Theorem 6.9]{Paper1}. The proof of the second assertion is the same, using Proposition \ref{prop:type2:pseudo} instead of 
  Proposition \ref{prop:pseudo_f}.
 
  We point out that the cited results, \cite[Lemma 6.8, Theorem 6.9]{Paper1}, were stated in terms of interval-partition
  evolutions. However, those results, and the methods used to prove them, extend to results like those stated in this lemma, in terms of (type-0, 
  type-1 or) type-2 data.
\end{proof}

\begin{proposition}\label{prop:3mass:Poi}
 Consider a type-2 evolution $((m_1^y,m_2^y,\alpha^y),y\ge 0)$ starting from $(x_1,x_2,x_3\overline{\beta})$ for 
  $\overline{\beta}\sim\PDIP(\frac12,\frac12)$. Then the associated 3-mass process $((m_1^y,m_2^y,\|\alpha^y\|),y\ge 0)$ is a Markov process starting from $(x_1,x_2,x_3)$. 
\end{proposition}
\begin{proof} We check the Rogers--Pitman intertwining criterion \cite[Theorem 2 and Remarks (i)-(ii)]{RogersPitman}. 
  Consider the map $\phi(x_1,x_2,\alpha) = (x_1,x_2,\|\alpha\|)$ and the stochastic kernel $\Lambda((x_1,x_2,x_3),A)=\bP((x_1,x_2,x_3\overline{\beta})\in A)$, where $\overline{\beta}\sim\PDIP(\frac12,\frac12)$. Then clearly $\Lambda((x_1,x_2,x_3),\phi^{-1}(x_1,x_2,x_3))=1$ and
  by Lemma \ref{pretotal}\ref{item:pretotal:3}, we also have 
  $$\bP_\mu\big((m_1^y,m_2^y,\alpha^y)\in A\ \big|\ m_1^y,m_2^y,\|\alpha^y\|\big)=\Lambda((m_1^y,m_2^y,\|\alpha^y\|),A)\qquad a.s.,$$
  for all initial distributions $\mu$ of the form $\Lambda((x_1,x_2,x_3),\,\cdot\,)$, as required.\pagebreak[2] 
\end{proof}

The semi-group of the 3-mass process can be described as ``replace the third component by a scaled $\PDIP(\frac12,\frac12)$, make type-2 evolution transitions, and then project the interval partition onto its mass.''

\subsection{Degeneration in pseudo-stationarity}

In the interweaving construction in pseudo-stationarity, it is easy to describe the degeneration time.

\begin{proposition}\label{prop:degeneration}
 Fix $\gamma>0$. Let $D$ be the degeneration time of a type-2 evolution starting from $(A,B,C\overline{\beta})$, where $A$, $B$, $C$ and 
 $\overline{\beta}$ are jointly independent, with $C\sim\GammaDist(\frac12,\gamma)$ and $\overline{\beta}\sim\PDIP(\frac12,\frac12)$. Then
 $\bP(D>y)=\bP(\zeta_1>y)\bP(\zeta_2>y)$ for all $y>0$, where $\zeta_1$ and $\zeta_2$ are the lifetimes of type-1 evolutions starting from
 $(A,C\overline{\beta})$ and $(B,C\overline{\beta})$, respectively. 

 If also $A,B\sim\GammaDist(\frac12,\gamma)$, then $\bP\{D>y\} = (2y\gamma+1)^{-2}$ for all $y>0$.
\end{proposition}


\begin{proof}
 The interweaving construction is such that on $\{\td J(\infty)\mbox{ even}\}$
 $$\left(0,\widetilde{m}_1^y\right)
	\concat\Concat_{2\le j\le\td J(\infty)\textrm{ even}} 
             \skewer\left(y - Z_{j-1}, \ShiftRestrict{\fN_{p(j+1)}}{(T_{j-2},T_{j}]\times\cE}\right)
   =\skewer\left(y,\fN_1\right),$$
 and on $\{\td J(\infty)\mbox{ odd}\}$, 
 $$\left(0,\widetilde{m}_2^y\right)
	\concat\skewer\left(y - \zeta(\ff_2), \Restrict{\fN_{2}}{(0,T_{1}]\times\cE}\right)
	\concat\Concat_{2\le j\le\td J(\infty)\textrm{ odd}} 
			 \skewer\left(y - Z_{j-1}, \ShiftRestrict{\fN_{p(j+1)}}{(T_{j-2},T_{j}]\times\cE}\right)$$
 equals $\skewer\left(y,\fN_2\right)$, while on the respective opposite event the expression on the LHS yields 
 $\skewer\left(y,\restrict{\fN_1}{[0,T_{\td J(\infty)-1}]}\right)$ and $\skewer\left(y,\restrict{\fN_2}{[0,T_{\td J(\infty)-1}]}\right)$, 
 respectively.
 
 On $\{\td J(\infty)\mbox{ even}\}$, the definitions of $Z_{\td J(\infty)}$ and $T_{\td J(\infty)}$ imply that
 $\zeta_2\ge\zeta=Z_{\td J(\infty)}>\zeta_1=D$, where $\zeta_i$ is the death level of the type-1 evolution $\skewerbar(\fN_i)$,
 $i=1,2$, and $\zeta$ is the death
 level of the type-2 evolution $((\td m_1^y,\td m_2^y,\td\alpha^y),y\ge 0)$. Together with corresponding observations on 
 $\{\td J(\infty)\mbox{ odd}\}$, we see that $D$ is the minimum of the lifetimes $\zeta_1$ and $\zeta_2$ of the two 
 pseudo-stationary type-1 evolutions used in the construction. Hence

 If we apply the interweaving construction to independent $A$ and $B$ with $\GammaDist[\frac12,\gamma]$ distribution, we 
 obtain the type-1 pseudo-stationary initial distribution. As these are i.i.d.\ with \ExpDist[\gamma] initial mass, from 
 \cite[equation (6.3)]{Paper1} they each have lifetime at least $y$ with probability $(2y\gamma+1)^{-1}$. The minimum of two i.i.d.\ variables with 
 this law has probability $(2y\gamma+1)^{-2}$ of exceeding $y$, as claimed.
\end{proof}

\begin{proposition}\label{degdist}
 Consider a type-2 evolution $((m_1^y,m_2^y,\alpha^y),y\ge 0)$ starting from the initial condition of Proposition \ref{prop:type2:pseudo} with $M\sim{\tt Gamma}(\frac{3}{2},\gamma)$, with degeneration time $D$. Let $A = \{I(D)=1\}$; this is the event that $(m_1^y,\,y\ge0)$ is the surviving clock spindle at the time of degeneration. In this event, $(m_1^D,m_2^D,\alpha^D)=(M^D,0,\emptyset)$; in the complementary event, $(m_1^D,m_2^D,\alpha^D)=(0,M^D,\emptyset)$. Then $\Pr(A) = \frac12$, $A$ is independent of $(D,M^D)$, and $\GammaDist[\frac{1}{2},\gamma/(2\gamma y+1)]$ is a regular conditional distribution for $M^D$ given $D=y$.
\end{proposition}
\begin{proof}
 We are interested in the joint distribution of $(D,m_1^D,m_2^D,\alpha^D)$. Using the construction of Lemma \ref{interweaving} from two independent type-1 evolutions with extinction times $\zeta_1$ and $\zeta_2$ and top mass processes $m_1$ and $m_2$, we have $D=\min\{\zeta_1,\zeta_2\}$, $A = \{\zeta_1>\zeta_2\}$, and
  $$(D,M^D)=(\zeta_2,m_1^{\zeta_2})\cf_A+(\zeta_1,m_2^{\zeta_1})\cf_{A^c}.$$
  Under the stated initial conditions, these two type-1 evolutions are in fact i.i.d.. From this, it is clear by symmetry that $\Pr(A)=\frac12$ and $A$ is independent of $(D,M^D)$, as claimed.
  
  For all nonnegative measurable $f$ and $g$ on $\bR$,
  $$\bE(f(D)g(M^D))=\bE(f(\zeta_2)g(m_1^{\zeta_2})\cf_A)+\bE(f(\zeta_1)g(m_2^{\zeta_1})\cf_{A^c}).$$ 
  We use Proposition \ref{prop:type1:pseudo_g} to rewrite the first term on the right hand side as
  \begin{equation*}
  \begin{split}
   &\int_0^\infty f(y)\bE(g(m_1^y)\mathbf{1}{\{\zeta_1>y\}})\bP(\zeta_2\in dy)\\
       &\ \ =\int_0^\infty f(y)\bE(\mathbf{1}{\{\zeta_1>y\}})\int_0^\infty g(x)\frac{1}{\sqrt{\pi x}}\sqrt{\frac{\gamma}{2\gamma y+1}}\exp\left(-\frac{\gamma}{2\gamma y+1}x\right)dx\bP(\zeta_2\in dy)\\
    &\ \ =\bE\left(f(\zeta_2)\mathbf{1}{\{\zeta_1>\zeta_2\}}\int_0^\infty g(x)\frac{1}{\sqrt{\pi x}}\sqrt{\frac{\gamma}{2\gamma \zeta_2+1}}\exp\left(-\frac{\gamma}{2\gamma\zeta_2+1}x\right)dx\right).
  \end{split}
  \end{equation*}
The second term can be written similarly, by symmetry, and together they give
  $$\bE(f(D)g(M^D))=\bE\left(f(D)\int_0^\infty g(x)\frac{1}{\sqrt{\pi x}}\sqrt{\frac{\gamma}{2\gamma D+1}}\exp\left(-\frac{\gamma}{2\gamma D+1}x\right)dx\right).$$
 This proves the claimed regular conditional distribution for $M^D$.
\end{proof}

Note that this result (and proof) formalizes an extension of the second part of Proposition \ref{prop:type2:pseudo_g} to the random time $y=D$, the degeneration time, and yields the same conditional distribution for the surviving top mass as for the surviving top masses when conditioning on $y<D$.

\subsection{De-Poissonization}

Consider a type-2 evolution $\fT=(T^y,y\ge 0)=\big((m_1^y,m_2^y,\alpha^y),y\ge 0\big)$ constructed from independent $\ff_1\sim\besq_a(-1)$ and $\fN_*\sim\fP^1_{(0,b)\concat\beta}$ as in Definition \ref{deftype2}. We now
consider the distribution $\bP^2_{a,b,\beta}$ of $\fT$ on the space $\bD([0,\infty),\cJ^\circ)$ of 
c\`adl\`ag functions from $[0,\infty)$ to $\cJ^\circ$.

For $T=(a,b,\beta)\in\cJ^\circ$, we consider the total mass $\|T\|=a+b+\|\beta\|$. For $\fT=(T^y,y\ge 0)\in\bD([0,\infty),\cJ^\circ)$, we define a time-change function
\begin{equation}\label{eq:dePoi_time_change}
 \rho_{\fT}\colon[0,\infty)\rightarrow[0,\infty],\qquad\rho_{\fT}(u)=\inf\left\{y\ge 0\colon\int_0^y\|T^x\|^{-1}dx>u\right\},\quad u\ge 0,
\end{equation}
which is continuous and strictly increasing until a potential absorption at $\infty$. Recall from Theorem \ref{thm:total_mass} that for a type-2 evolution $\fT$ starting from $(a,b,\beta)\in\cJ^\circ\setminus\{(0,0,\emptyset)\}$, we have $\|\fT\|=(\|T^y\|,y\ge 0)\sim\besq_{a+b+\|\beta\|}(-1)$. By \cite[p.314-5]{GoinYor03}, $\rho_{\fT}$ is bijective onto $[0,\zeta)$ a.s., where $\zeta=\inf\{y\ge 0\colon\|T^y\|=0\}$. Let
\begin{equation}\label{eq:IP_space_1_def}
 \cJ^\circ_1 := \{T\in\cJ^\circ\colon\|T\|=1\}, \qquad \cI^\circ_1 := \{\beta\in\cI^\circ\colon \|\beta\|=1\}.
\end{equation}

\begin{definition}\label{defdePoiss}
 Let $\widebar{\nu}$ be a distribution on $\cJ^\circ_1$. Given a type-2 evolution $\fT\sim\bP^2_{\widebar\nu}$ starting according to $\widebar{\nu}$, we associate the \emph{de-Poissonized type-2 evolution} $\overline{T}^u=T^{\rho_{\fT}(u)} / \|T^{\rho_{\fT}(u)}\|$, $u\ge 0$. We denote its distribution on $\bD([0,\infty),\cJ^\circ_1)$ by $\overline{\bP}^{2,-}_{\overline{\nu}}$.
  
  We define the \emph{IP-valued de-Poissonized type-2 evolution} to be
  \begin{equation*}
   \left\{\left(0,m_{I(\rho_{\fT}(u))}^{\rho_{\fT}(u)}\right)\right\}\concat \left\{\left(0,m_{3-I(\rho_{\fT}(u))}^{\rho_{\fT}(u)}\right)\right\} \concat \alpha^{\rho_{\fT}(u)},\quad u\ge 0.
  \end{equation*}
\end{definition}

\begin{theorem}\label{absorption}
 De-Poissonized type-2 evolutions are Borel right Markov processes absorbed in finite time in either $(1,0,\emptyset)$ or $(0,1,\emptyset)$. The IP-valued variants of these processes are also Borel right Markov processes, but are additionally path-continuous, and are absorbed in finite time in the degenerate interval partition $\{(0,1)\}$ of the interval $(0,1)$.
\end{theorem}

\begin{proof}
 $\cJ^\circ_1$ (likewise $\cI^\circ_1$) is a Borel subset of a Lusin space, and is therefore Lusin. Both continuous time changes and normalization on $\cJ^\circ\setminus\{(0,0,\emptyset)\}$ (respectively, $\cI^\circ\setminus\{\emptyset\}$) preserve the property of sample paths being c\`adl\`ag (resp.\ continuous). The strong Markov property of Theorem \ref{thm:diffusion} and the continuity in the initial state of Proposition \ref{contini} transfer to the de-Poissonized processes as in \cite[Proposition 6.7, proof of Theorem 1.6]{Paper1}. 

  For both processes, absorption occurs at the time $\overline{D}$ that satisfies $\rho_{\fT}(\overline{D}) = D$ since $D<\zeta$ a.s.; note that all states $(m,0,\emptyset)$, $m\in(0,\infty)$, are normalized to $(1,0,\emptyset)$, and similarly for $(0,1,\emptyset)$. For the IP-valued process, all states $\{(0,m)\}$ are normalized to $\{(0,1)\}$. Finally, the time-change is such that $\rho_{\fT}(u)<\zeta$ for all $u\in[0,\infty)$. 
\end{proof}

Pal \cite{Pal11,Pal13} studied Wright--Fisher diffusions with positive and negative real parameters $\theta_1,\ldots,\theta_n$ as de-Poissonized processes associated with vectors of independent $Z_i\sim\besq(2\theta_i)$, $1\le i\le n$. Combining the arguments of \cite[Proposition 11]{Pal11} and \cite[Theorem 4]{Pal13}, we may define generalized Wright--Fisher diffusions (running at 4 times the speed of \cite{Pal11,Pal13}) as either weak solutions to certain systems of stochastic differential equations or, as is relevant for us, as 
\begin{equation}\label{eq:WF_constr}
 \overline{Z}_i(u)=\frac{Z_i(\rho(u))}{Z_+(\rho(u))},\ \ 1\le i\le n,\ \ 0\le u\le\overline{\tau} = \inf\{s\ge 0\colon \exists i\text{ s.t.\ }\overline{Z}_i(s)=0\mbox{ and }\theta_i\le 0\},
\end{equation}
where $Z_+(y) := \sum_{i=1}^nZ_i(y)$ and $\rho(u)$ is as in \eqref{eq:dePoi_time_change}, but with $Z_+(x)$ in place of $\|T^x\|$ inside the integral. 
See also \cite[pp. 60--61]{Paper1}.
 
\begin{proposition}\label{prop:3mass:dePoi}
 Let $\overline{\fT}=((X_1(u),X_2(u),\beta(u)),u\ge 0)\sim\overline{\bP}^{2,-}_{a,b,\beta}$ be a de-Poissonized type-2 evolution starting from any $(a,b,\beta)\in\cJ^\circ_1$. Let $U=\inf\{u\ge 0\colon X_1(u)=0\mbox{ or }X_2(u)=0\}$. Then the 3-mass process $((X_1(u),X_2(u),1-X_1(u)-X_2(u)),0\le u\le U)$ is a generalized Wright--Fisher process with parameter vector $(-\frac{1}{2},-\frac{1}{2},\frac{1}{2})$.
  
  If furthermore the initial state is taken as $\beta=(1-a-b)\overline{\beta}$ for $\overline{\beta}\sim\PDIP(\frac12,\frac12)$, then the 3-mass process $((X_1(u),X_2(u),1-X_1(u)-X_2(u)),u\ge 0)$ is a Markovian extension of the generalized Wright--Fisher process.
\end{proposition}

\begin{proof}
 For the first claim, we assume without loss of generality that $\overline{\fT}$ is constructed as in Definition \ref{defdePoiss} from a type-2 evolution $\fT=(T^y,\,y\ge 0)$ arising from type-2 data $(\ff_1,\ff_2,\cev{\fN},\fN_\beta)\sim\fP_{a,b,\beta}^2$ as in Definition \ref{deftype2}. By Proposition \ref{type1totalmass}, we have $\|\skewerbar(\cev{\fN}\concat\fN_\beta)\|\sim\besq_{\|\beta\|}(1)$. This process, together with $\ff_1\sim\besq_a(-1)$ and $\ff_2\sim\besq_b(-1)$ forms a triple of \besq\ processes, as in the paragraph above the proposition. Thus, we can construct a generalized Wright--Fisher process from $Z_1=\ff_1$, $Z_2=\ff_2$ and $Z_3=\|\skewerbar(\cev{\fN}\concat\fN_\beta)\|$. Since $Z_+(y)=\|T^y\|$ for $0\le y\le\tau:=\inf\{y\ge 0:\exists i\text{ s.t.\ }Z_i(y)=0\}$, we have $\rho(u)=\rho_{\fT}(u)$, and hence $X_i(u)=\overline{Z}_i(u)$, $i=1,2$, and $X_3(u)=1-X_1(u)-X_2(u)=1-\overline{Z}_1(u)-\overline{Z}_2(u)=\overline{Z}_3(u)$, $0\le u\le U=\overline{\tau}$. This completes the proof.   
  
  The second claim follows from Proposition \ref{prop:3mass:Poi} and the observation that the (Poissonized) 3-mass process of that proposition can be de-Poissonized by the same scaling/time-change operation as the type-2 evolution, as the scaling and time change only depend on the common total mass process.
\end{proof}

\subsection{Resampling and stationarity}

As we have seen in Theorem \ref{absorption}, de-Poissonized type-2 evolutions degenerate at a finite random time $\overline{D}<\infty$ in one of the two absorbing states $(1,0,\emptyset)$ and $(0,1,\emptyset)$. In this section we will restart the process instead of entering the absorbing states. Informally and with the conjectured stationary Aldous diffusion in mind, we take the opportunity to sample afresh from the Brownian CRT at each degeneration time. Recall the state space $\cJ_1^*$ of \eqref{eq:IP_spaces}. Formally, we sample from the distribution $\widebar\mu$ on $\cJ^*_1$ of $(X_1,X_2,(1-X_1-X_2)\overline{\beta})$, where $(X_1,X_2,1-X_1-X_2)\sim{\rm Dirichlet}(\frac{1}{2},\frac{1}{2},\frac{1}{2})$ is independent of the interval partition $\widebar{\beta}\sim\PDIP[\frac12,\frac12]$, as discussed in the introduction. 

\begin{definition}\label{def:resampling}
 Let $(a,b,\beta)\in\cJ^*_1$. Let $(\overline{T}^u_{(j)},0\le u<\overline{D}_{(j)})$, $j\ge 0$, be a sequence of 
  independent copies of $(\overline{T}^u_-,0\le u<\overline{D})$, where $\overline{\fT}_-\sim\overline{\bP}_{a,b,\beta}^{2,-}$ for $j=0$ 
  and $\overline{\fT}_-\sim\overline{\bP}_{\overline{\mu}}^{2,-}$ for $j\ge 1$, see Definition \ref{defdePoiss}. Set $V_0=0$ and denote the cumulative 
  lifetimes inductively by $V_j=V_{j-1}+\overline{D}_{(j)}$, $j\ge 1$. Then the concatenation
  $$\overline{T}^{V_j+u}_+=\overline{T}^u_{(j)},\qquad 0\le u<\overline{D}_{(j)},\ j\ge 0,$$
  is called a \emph{(resampling) 2-tree evolution} starting from $(a,b,\beta)$. We denote its distribution on 
  $\bD([0,\infty),\cJ^*_1)$ by $\overline{\bP}^{2,+}_{a,b,\beta}$. For clarity, we continue to use notation $(\overline{T}_+^u,u\ge 0)$ for the 
  canonical process on $\bD([0,\infty),\cJ^*_1)$ when working under $\overline{\bP}^{2,+}_{a,b,\beta}$.
\end{definition}

\begin{proof}[Proof of Theorem \ref{thm:stationary}]
 To confirm that the 2-tree evolution is a Borel right Markov process, we only need to check the strong Markov property. Given this construction, this can be seen in the context of general results about resurrecting Markov processes \cite{Mey75}. However, since our setting is much more elementary than the general theory, we sketch an elementary proof.
Consider a stopping time $Y$ in the right-continuous filtration $(\cF^y,y\ge 0)$ generated by the canonical process on $\bD([0,\infty),\cJ^*_1)$. For $j\ge 1$, let $Y_j=\min\{Y,V_j\}$. 
  Denote by $(\theta_y, y\ge 0)$ the canonical shift operators on $\bD([0,\infty),\cJ^*_1)$. Then for all probability measures $\widebar\nu$ on $\cJ^*_1$, for all 
  $0\le j\le k$, for all  $\cF^{V_1}$-measurable $\eta_i,f\colon\bD([0,\infty),\cJ^*_1)\rightarrow[0,\infty)$, with  
  $\cF^{Y}$-measurable $\eta_j\circ\theta_{V_j}$, we check that
  $$\overline{\bP}^{2,+}_{\widebar\nu}\left[\prod_{i=0}^k\eta_i\circ\theta_{V_i}\mathbf{1}_{\{V_{j}\le Y<V_{j+1}\}}\,f\circ\theta_Y\right]
	  =\overline{\bP}^{2,+}_{\widebar\nu}\Bigg[\prod_{i=0}^{j}\eta_i\circ\theta_{V_i}\mathbf{1}_{\{V_{j}\le Y<V_{j+1}\}}\overline{\bP}^{2,+}_{\widebar T^Y_+}\Bigg[f\prod_{i=j+1}^k\eta_i\circ\theta_{V_{j-i}}\Bigg]\Bigg]$$
  holds, by inductively applying the strong Markov property under $\widebar\bP^{2,+}_{\widebar\nu}$ and $\widebar\bP^{2,+}_{\widebar\mu}$ at $V_i$, which holds by 
  construction, and the strong Markov property under $\widebar\bP^{2,-}_{\widebar\mu}$ (or, for $j=0$ under $\widebar\bP^{2,-}_{\widebar\nu}$) at the $Y-V_j$.  Summing over $j\ge 0$ and applying a monotone class theorem proves the strong Markov property.
  
  Now, we prove that $\widebar\mu$, defined before Definition \ref{def:resampling}, is the unique stationary distribution, and that the process converges to it. Applying Lemma \ref{pretotal}\ref{item:pretotal:1} to the $(\cF^y_{\rm mass},y\ge 0)$-stopping time $Y=\rho_{\fT}(u)$, we find
  \begin{equation}\label{statpre}
   \bE_{\widebar\mu}^2\left(g(T^{\rho_{\fT}(u)}/\|T^{\rho_{\fT}(u)}\|)1_{\{D>\rho_{\fT}(u)\}}\right)=
   \bP_{\widebar\mu}^2(D>\rho_{\fT}(u)){\widebar\mu}(g).
  \end{equation}
  Now consider a resampling 2-tree evolution $(\overline{T}_+^u,u\ge 0)$ with initial distribution ${\widebar\mu}$. We use the notation of Definition \ref{def:resampling}. Let $U\sim\ExpDist(\lambda)$ independent of the resampling 2-tree evolution. Then \eqref{statpre} yields
  $\bE\left(g\left(\overline{T}_+^U\right)1_{\{U<V_1\}}\right)=\bP(U<V_1){\widebar\mu}(g)$. 
  For $m\ge 1$,
  \begin{align*}\bE\left(g\left(\overline{T}_+^U\right)1_{\{V_m\le U<V_{m+1}\}}\right)
	&=\int_0^\infty\lambda e^{-\lambda u}\bE\left(g\left(\overline{T}^u_+\right)1_{\{V_m\le u<V_{m+1}\}}\right)du\\
	&=\bE\left(e^{-\lambda V_m}\int_0^\infty\lambda e^{-\lambda s}g\left(\overline{T}_+^{V_m+s}\right)1_{\{V_m+s<V_{m+1}\}}ds\right)\\
	&=\int_0^\infty\lambda e^{-\lambda s}\bE\left(e^{-\lambda V_m}g\left(\overline{T}_+^{V_m+s}\right)1_{\{V_m+s<V_{m+1}\}}\right)ds\\
	&=\int_0^\infty\lambda e^{-\lambda s}\bE\left(e^{-\lambda V_m}\bE\left(g(\overline{T}_+^s)1_{\{s<V_1\}}\right)\right)ds\\
	&=\bE\left(e^{-\lambda V_m}\bE\left(g\left(\overline{T}_+^U\right)1_{\{U<V_1\}}\right)\right)\\
	&=\bE\left(e^{-\lambda V_m}\right)\bP(U<V_1){\widebar\mu}(g)=\bP(V_m\le U<V_{m+1}){\widebar\mu}(g).
  \end{align*}
  Summing over $m$ and inverting Laplace transforms in $\lambda$, we find that $\overline{\bE}^{2,+}_{\widebar\mu}(g(T^u))={\widebar\mu}(g)$ for all $u\ge 0$, i.e.\ ${\widebar\mu}$ is stationary for $\overline{\fT}_+$. Furthermore, since resampling is according to the stationary distribution ${\widebar\mu}$, we have for any other initial distribution ${\widebar\nu}$ that for all bounded measurable $g\colon\cJ^*_1\rightarrow[0,\infty)$
  $$\overline{\bE}_{\widebar\nu}^{2,+}(g(T^u)) = \overline{\bE}_{\widebar\nu}^{2,+}\left(g(T^u)1_{\{u<V_1\}}\right)+\overline{\bE}_{\widebar\nu}^{2,+}(V_1\le u){\widebar{\mu}}(g)\rightarrow{\widebar{\mu}}(g),$$ 
  since $V_1=\overline{D}$ is finite $\overline{\bP}_{\widebar\nu}^{2,+}$-a.s.. In particular, the stationary distribution is unique.
\end{proof}

As in Propositions \ref{prop:3mass:Poi} and \ref{prop:3mass:dePoi}, we can project a 2-tree evolution $((\widebar m_1^y,\widebar m_2^y,\widebar\beta^y),\,y\ge0)$ down to a \emph{resampling 3-mass process} $\big((\widebar m_1^y,\widebar m_2^y,\widebar\beta^y),\,y\ge0\big)$. 
We can now prove Theorem \ref{thm:wright-fisher}.

\begin{proof}[Proof of Theorem \ref{thm:wright-fisher}]
Proposition \ref{prop:3mass:dePoi} and Theorem \ref{thm:stationary} imply that 
the resampling 3-mass process is a Borel right Markov process that extends the generalized Wright--Fisher process to a recurrent process on the simplex 
$\{(a,b,c)\in [0,1)^3\colon a+b+c=1\}$, which has 
\DirDist[\frac12,\frac12,\frac12] stationary distribution, and converges to stationarity. 
\end{proof}

Note that the Wright--Fisher diffusion with parameters $\left(\frac12,\frac12,\frac12\right)$ has this same invariant law.

Recall the definition in \eqref{eq:diversity} of the diversity $\sD_\beta$ of an interval partition $\beta\in\cI$.

\begin{corollary}
 Under $\overline{\bP}_{\widebar\nu}^{2,+}$, let $(\widebar\beta^u,\,u\ge 0)$ denote the evolution of the interval partition component. Then the total diversity process $(\sD_{\widebar\beta^u}(\infty),\,u\ge 0)$ is continuous except at the resampling times $V_m$, $m\ge 1$.
\end{corollary}
\begin{proof}
 Consider de-Poissonized type-2 evolution, its IP-valued variant as in Definition \ref{defdePoiss}, and the extension of the triple-valued process to a 2-tree evolution via resampling. Up until the first resampling time, the interval partition component of the resampling process differs from the IP-valued evolution by at most two leftmost blocks, removed from the latter to construct the former. As noted in Theorem \ref{absorption}, the latter process is path-continuous under $d_{\cI}$, so its total diversity process is continuous. Removing a finite number of blocks does not change its total diversity. This same argument proves continuity between any two consecutive resampling times.
\end{proof}

\subsection{H\"older estimates}\label{sec:Holder}

Let us denote by $\widetilde\mu$ the distribution of the 
$\cI^\circ$-valued interval partition 
$$(0,\overline{A})\star(0,\overline{B})\star \overline{G}\,\overline{\beta}$$
for independent $(\overline{A},\overline{B},\overline{G})\sim{\tt Dir}(\frac12,\frac12,\frac12)$ and 
$\overline{\beta}\sim{\tt PDIP}(\frac12,\frac12)$. This distribution is not pseudo-stationary in the strong sense that the distribution of an 
$\cI^\circ$-valued type-2 evolution starting from $\widetilde\mu$ has as marginal distributions the distributions of random multiples of this interval 
partition -- intuitively, the leftmost block is stochastically larger than the second block. However, we will be able to appeal to the 
pseudo-stationarity of $\cJ^\circ_1$-valued type-2 evolutions starting from $(\overline{A},\overline{B},\overline{G}\,\overline{\beta})$ in 
situations that treat the two top masses symmetrically.


\begin{proposition}\label{type2holder}
 Let $(\widetilde{\beta}^y,y\ge 0)$ be an $\cI^\circ$-valued type-2 evolution starting according to $\widetilde\mu$. Let $\theta\in(0,1/4)$ and $y>0$. Then 
 there is a random H\"older constant $L=L_{\theta,y}$ with moments of all orders such that
 $$d_\cI(\widetilde{\beta}^a,\widetilde{\beta}^b)\le L|b-a|^\theta\qquad\mbox{for all }0\le a<b\le y.$$ 
\end{proposition} 

The remainder of this subsection is devoted to the proof of this proposition. We begin by some preliminary considerations. Let us first consider the type-2 evolution $\beta^y=M\widetilde{\beta}^{y/M}$, $y\ge0$, with \GammaDist[\frac32,\gamma] initial mass $M$ for some $\gamma>0$, cf.\ Proposition \ref{prop:type2:pseudo_g}. Recall that 
\begin{itemize}[leftmargin=.82cm]
  \item the evolution of $(\beta^y,y\ge 0)$ can be constructed by interweaving two independent pseudo-stationary 
    type-1 evolutions of initial mass ${\tt Exponential}(\gamma)$ (Propositions \ref{interweaving} and \ref{prop:type2:pseudo});
  \item the pseudo-stationary type-1 evolution consists of a type-1 evolution starting from a single 
    interval $(0,A)$ with $A\sim\GammaDist[\frac12,\gamma]$ concatenated left-to-right with an independent type-1 evolution starting from a $\PDIP[\frac12,\frac12]$ scaled by an independent \GammaDist[\frac12,\gamma] mass (Proposition \ref{prop:type1:pseudo_g}); 
  \item a type-1 evolution starting from a single \GammaDist[\frac12,\gamma]-distributed interval $(0,A)$ can be constructed
    from a $\besq_A(-1)$ process, with death level $\zeta$ and an independent \Stable[\frac32] process  
    with $\besq(-1)$ excursions in its jumps and run until it hits $-\zeta$, then shifted up by $\zeta$ to descend from
    $\zeta$ to $0$;
  \item a type-1 evolution starting from $\PDIP[\frac12,\frac12]$ scaled by mass \GammaDist[\frac12,\gamma] can be constructed from a \Stable[\frac32] process $\widetilde{\fX}$ starting from 0, with $\besq(-1)$ excursions in its jumps stopped at a time $\widetilde{T}$, which is the left endpoint of the excursion away from 0 where the mass at level 0 exceeds an independent ${\tt Exponential}(\gamma)$ threshold (Proposition \ref{prop:type1:pseudo_constr});
  \item the local times $(\widetilde{\ell}^y(t),0\le t\le\widetilde{T},y\ge 0)$ of 
    $(\widetilde{\fX}(t),0\le t\le\widetilde{T})$ have the property that for each $a\ge 0$ and $\theta\in(0,1/4)$, the random variable 
    $$\widetilde{D}_\theta^a=\sup_{0\le t\le\widetilde{T},0\le x<y\le a}\frac{|\widetilde{\ell}^x(t)-\widetilde{\ell}^y(t)|}{|y-x|^\theta}$$
    has moments of all orders (\cite[Theorem 3]{Paper0});
  \item 
  for a type-1 evolution $(\alpha^y,\,y\ge0)$ arising from some scaffolding $\fX$ marked by \BESQ[-1] spindles, 
  it is a.s.\ the case that for every $y$ and every block $U\in\alpha^y$, the diversity $\mathscr{D}_{\alpha^y}(U)$ equals the local time $\ell^y(t)$ in $\fX$, up to the time $t$ at which the spindle corresponding to block $U$ arises (\cite[Theorem 1]{Paper0} or \cite[Theorem 4.15]{Paper1}).
\end{itemize}
Consider the \Stable[\frac32] process starting from $\zeta$ obtained by concatenating the descent from $\zeta$ to $0$ before $\td \fX$. Denote this process by $(\widehat{\fX}(t),0\le t\le\widehat{T})$ and its local times by $(\hat{\ell}^y(t),0\le t\le\widehat{T},y\ge 0)$. 
In this context, \cite[Theorem 3]{Paper0} has the following consequence. 

\begin{lemma}\label{lem:LT_Holder} 
 The following random variable has moments of all orders:
  $$\widehat{D}_\theta^a=\sup_{0\le t\le\widehat{T},0\le x<y\le a}\frac{|\hat{\ell}^x(t)-\hat{\ell}^y(t)|}{|y-x|^\theta}.$$
\end{lemma}
\begin{proof}
 Let $H_\zeta=\inf\{t\ge 0\colon\widetilde{\fX}(t)=\zeta\}$. Then the event $\{H_\zeta<\widetilde{T}\}$ has 
  positive probability and by the strong Markov property and by the property of
  Poisson processes of marked excursions, the conditional distribution given $H_\zeta<\widetilde{T}$ of the process
  $(\widetilde{\fX}(H_\zeta+s),0\le s\le\widetilde{T}-H_\zeta)$ is the same as the unconditional distribution of 
  $(\widehat{\fX}(t),0\le t\le T)$; c.f.\ Lemma \ref{lem:pseudostat:memoryless}. But then the associated local times 
  $(\td\ell^y(H_\zeta+s)-\td\ell^y(H_\zeta),0\le s\le\widetilde{T}-H_\zeta,y\ge 0)$ have as their conditional distribution the
  distribution of $(\hat{\ell}^y(t),0\le t\le\widehat{T},y\ge 0)$. By the triangle inequality,
  $$\bE\big[(\widehat{D}_\theta^a)^p\big] \le \bE\big[(2\widetilde{D}_\theta^a)^p\big|H_\zeta<\widetilde{T}\big]<\infty.\vspace{-0.4cm}$$
\end{proof}
This allows us to bound terms (i) and (ii) of Definition \ref{def:IP:metric} of $d_\cI$, which deal with diversity.
\begin{lemma}\label{lem:LT_matching_bd}
  There is a random variable $L_\theta$ with moments of all orders such that uniformly over all those
  matchings $((U_j,U_j^\prime),1\le j\le m)$ of intervals of $\beta^0$ and $\beta^y$ that are taken from the same 
  $\besq(-1)$ excursion, we have
  \begin{equation*}
   \big|\IPLT_{\beta^y}(\infty)-\IPLT_{\beta^0}(\infty)\big| \le L_\theta y^\theta \quad \text{and} \quad \max_{1\le j\le m}\big|\IPLT_{\beta^y}(U'_j)-\IPLT_{\beta^0}(U_j)\big| \le L_\theta y^\theta.
  \end{equation*}
\end{lemma}
\begin{proof}
  Think of $(\beta^y,\,y\ge0)$ as arising from an interweaving construction, as in Section \ref{sec:interweaving}, so for each $y\ge0$, $\beta^y$ is formed as in \eqref{eq:interweaving_skewer}, by concatenating alternating intervals of the skewers of two i.i.d.\ copies $(\widehat\fX_{(1)},\widehat\fX_{(2)})$ of $\widehat\fX$ with jumps marked by \BESQ[-1] spindles. Now, consider a block $U\in\beta^0$; this corresponds to one such spindle, marking a jump at some time $t$ in either $\widehat\fX_1$ or $\widehat\fX_2$. Suppose, for example, that this spindle appears in $\fX_1$ with $t\in [T_2,T_4)$, in the notation of \eqref{eq:interweaving_skewer}. Then by \cite[Theorem 4.15]{Paper1}, $\IPLT_{\beta^0}(U) = \hat\ell^0_{(1)}(t) + \hat\ell^0_{(2)}(T_3)$, and if $U'\in\beta^y$ corresponds to the same spindle, then $\IPLT_{\beta^y}(U') = \hat\ell^y_{(1)}(t) + \hat\ell^y_{(2)}(T_3)$. Such comparisons can be made for spindles coming from any interval $[T_{j-2},T_j)$ in either $\widehat\fX_{(1)}$ or $\widehat\fX_{(2)}$. Thus, the claimed bounds follow from Lemma \ref{lem:LT_Holder} by the triangle inequality, with the $p^{\rm th}$ moment of $L_\theta$ being bounded by twice that of $\widehat{D}_\theta^y$.
\end{proof}
It remains to bound terms (iii) and (iv) in Definition \ref{def:IP:metric}, which deal with mass. Consider a sequence of $m$ distinct size-biased picks among the blocks of $\beta^0$, and match these with the blocks arising from the same spindle at time $y$, $((U_j,U_j^\prime),1\le j\le m)$, allowing that ${\rm Leb}(U_j^\prime)$ may equal zero for some $j$ if the spindle does not survive. We can separately control
\begin{itemize}[leftmargin=.82cm]
  \item total discrepancy between matched beads $\sum_{1\le j\le m}\big|{\rm Leb}(U_j)-{\rm Leb}(U_j^\prime)\big|$,
  \item unmatched level-0 mass $\|\beta^0\|-\sum_{1\le j\le m}{\rm Leb}(U_j)\cf\{{\rm Leb}(U_j')> 0\}$,
  \item and unmatched level-$y$ mass $\|\beta^y\|-\sum_{1\le j\le m}{\rm Leb}(U_j^\prime)$.
\end{itemize}

Denote by $\widetilde\mu_\gamma$ the distribution of $M\widetilde{\beta}^0$ for independent $\widetilde{\beta}^0\sim\widetilde\mu$ and $M\sim\GammaDist[\frac32,\gamma]$.

\begin{lemma}\label{gammaholder}
 Let $(\beta^y,y\ge 0)$ be an $\cI^\circ$-valued type-2 evolution starting according to $\widetilde\mu_\gamma$. Let $\theta\in(0,1/4)$ and $p>0$. Then there is a constant $C=C_{\gamma,\theta,p}$ such that
 \begin{equation}\label{eq:gammaholder}
  \bE\left[(d_\cI(\beta^0,\beta^y))^p\right]\le Cy^{\theta p}\qquad\mbox{for all }0\le y\le 1.
 \end{equation}
\end{lemma}
\begin{proof}
 Consider such a process $(\beta^y,\,y\ge0)$. Its initial state is of the form $\beta^0 = (0,A)\concat (0,B)\concat G\bar\beta$, where $A,B,G$ are i.i.d.\ \GammaDist[\frac12,\gamma] random variables, independent of $\widebar\beta\sim\PDIP[\frac12,\frac12]$. Further let $\beta^*=(\beta^*_1,\beta^*_2,\dots) \in [0,1]^\infty$ denote a size-biased random ordering of the masses of $\bar\beta$. 
 
 To construct the matching, we take the blocks $U_1=A$, $U_2=B$, together with the blocks $U_i$ for $3\leq i\leq m=\lfloor 3y^{-1/4}\rfloor$, where $U_i$ is the block corresponding to $G\beta^*_{i-2}$ and match them with the blocks $U_A',U_B',U_1',\ldots$ that arise from the corresponding spindles at level $y$. Consequently, ${\rm Leb}(U_i') = \mathbf{g}_i(y)$ where $\mathbf{g}_i\sim\besq_{{\rm Leb}(U_i)}(-1)$ given ${\rm Leb}(U_i)$. 
%
 Note that this means that some of our blocks will be matched with empty blocks and should thus be accounted for in the remaining mass component of the metric.  We will handle this later.
 
 Assuming $p \geq 2$, $y\in(0,1]$, and using \cite[Lemma 33]{Paper0} and the fact that $M$ has finite moments of all orders, there are constants $C_1,C_2, C_3,C_4$ and $C_5$, depending only on $p$, such that
 \[ \begin{split}
  \bE \Bigg[ \Bigg(\sum_{1\le j\le m}|{\rm Leb}(U_j)-{\rm Leb}(U_j^\prime)|\Bigg)^{p}\Bigg]
  	&\leq m^{{p}-1} \bE \Bigg[ \sum_{1\le j\le m}|{\rm Leb}(U_j)-{\rm Leb}(U_j^\prime)|^{p}\Bigg]\\
	&\leq m^{{p}-1}y^{{p}/2} \sum_{1\le j\le m}\bE\Big[\Big(C_1 +\sqrt{C_2{\rm Leb}(U_j)+C_3}\Big)^{p}\Big] \\
	&\leq m^{{p}-1}y^{{p}/2} \sum_{1\le j\le m}\bE\Big[\Big(C_1 +\sqrt{C_2{\rm Leb}(M)+C_3}\Big)^{p}\Big] \\
	&\leq C_4 m^{{p}} y^{{p}/2} 
   \leq C_5 y^{p/4},
 \end{split} \]
 absorbing $3^p$ into the constant at the last step.
 
 The unmatched mass at level $0$ is $G\sum_{j=m-1}^\infty \beta^*_i$.  Let $(Y_n)_{n\geq 1}$ be a sequence of independent random variables, also independent of $G$, such that $Y_n$ has ${\tt Beta}(1/2, (n+1)/2)$ distribution.  Using the stick-breaking construction of the Poisson--Dirichlet distribution, we see that for all $y\in(0,1]$
 \[\begin{split}
  \bE\Bigg[ \Bigg(G \sum_{j=m-1}^\infty \beta^*_i\Bigg)^{p}\Bigg]
 	&= \bE[G^p] \bE\Bigg[ \Bigg(1 - \sum_{j=1}^{m-2} \beta^*_i\Bigg)^{p}\Bigg]
	= \bE[G^p]\bE\Bigg[ \Bigg(\prod_{j=1}^{m-2} (1-Y_j)\Bigg)^{p}\Bigg] \\
	&= \bE[G^p]\prod_{j=1}^{m-2}\bE\left[ \left(1-Y_j\right)^{p}\right] 
	=\bE[G^p]\Gamma\left(1+{p} \right)\frac{ \Gamma\left(\frac{m}{2}\right)}{\Gamma\left(\frac{m}{2} +{p}\right)} \\
	&\sim \bE[G^p]\Gamma\left(1+{p} \right)m^{-{p}} 
	= \bE[G^p]\Gamma(1+{p}) \lfloor 3y^{-1/4}\rfloor^{-{p}}
 	\le C_6 y^{p/4},
 \end{split}\]
 for some $C_6\ge C_5$.
 
 We are left with estimating the unmatched mass at time $y$. By the triangle inequality,
 \[ \|\beta^y\| - \sum_{j=1}^{m} {\rm Leb}(U'_j)  \leq \big| \|\beta^y\|-M\big| + \Bigg|G\sum_{j=m-1}^\infty \beta^*_j\Bigg|+ \sum_{1\le j\le m}\big|{\rm Leb}(U_j)-{\rm Leb}(U_j^\prime)\big|.\]
 Furthermore, by Theorem \ref{thm:total_mass}, $(\|\beta^y\|)_{y\geq 0}$ is a \besq$_{M}(-1)$ process to which \cite[Lemma 33]{Paper0} applies, as above. Consequently, we have for some $C_7>0$ that
 \[\begin{split} &\bE \Bigg|\|\beta^y\| - \sum_{j=1}^{m} {\rm Leb}(U'_j)\Bigg|^{p}\\ 
   &\leq 3^{{p}-1}\Bigg(\bE \Big| \|\beta^y\|-M\Big|^{p} + \bE\Bigg|G\sum_{j=m-1}^{\infty} \beta^*_j\Bigg|^{p}+ \bE\Bigg[\Bigg(\sum_{1\le j\le m}|{\rm Leb}(U_j)-{\rm Leb}(U_j^\prime)|\Bigg)^{{p}}\Bigg]\Bigg) \\
   & \leq C_7(y^{{p}/2} +y^{p/4} + y^{p/4})
   \leq 3C_7 y^{{p} /4}.
 \end{split}\]
 To account for the fact that some ${\rm Leb}(U'_i)$ may be $0$, and thus the corresponding $U_i$ should count towards unmatched mass at time $0$, we bound the metric $d_{\mathcal{I}}$ using the correspondence defined above (and bounding the maximum in the definition of $d_{\mathcal{I}}$ by a sum) to see that
 \[\begin{split}
   d_{\mathcal{I}}(\beta^0,\beta^y) & \leq  \sum_{1\le j\le m}|{\rm Leb}(U_j)-{\rm Leb}(U_j^\prime)|\mathbf{1}\{{\rm Leb}(U'_j) >0\}\\
   &\qquad + \|\beta^y\| - \sum_{j=1}^{m} {\rm Leb}(U'_j)\mathbf{1}\{{\rm Leb}(U'_j) >0\} +M- \sum_{j=1}^m {\rm Leb}(U_j)\mathbf{1}\{{\rm Leb}(U'_j) >0\} \\
   &\qquad + |\ell^0(T)-\ell^y(T)|+  \max_{1\le j\le m}|\ell^0(U_j)-\ell^y(U_j^\prime)|\mathbf{1}\{{\rm Leb}(U'_j) >0\} \\
   & =   \sum_{1\le j\le m}|{\rm Leb}(U_j)-{\rm Leb}(U_j^\prime)|+ \|\beta^y\| - \sum_{j=1}^{m} {\rm Leb}(U'_j) +M -\sum_{j=1}^m {\rm Leb}(U_j)\\
   & \qquad + |\ell^0(T)-\ell^y(T)|+  \max_{1\le j\le m}|\ell^0(U_j)-\ell^y(U_j^\prime)|\mathbf{1}\{{\rm Leb}(U'_j) >0\}.
 \end{split}\]
 Dropping the indicator on the last term and combining this with our calculations above and Lemma \ref{lem:LT_matching_bd} shows that for $0<\theta<1/4$ there exists some constant $C_{\gamma,\theta,{p}}$ depending only on $\gamma$, $\theta$ and $p$ that satisfies \eqref{eq:gammaholder}.
\end{proof}

\begin{proof}[Proof of Proposition \ref{type2holder}]
Let $(\beta^y,y\ge 0)$ be a type-2 evolution with initial distribution $\mu_\gamma$. Denote the total mass evolution by 
$Z(y)=\|\beta^y\|$, $y\ge 0$. Let $0\le a<b\le 1$. Then
\begin{align*}
  \bE_{\mu_\gamma}\!\left[(d_\cI(\beta^a,\beta^b))^p\right]
  &=\bE_{\mu_\gamma}\!\left[\cf\{D>a\}\bE_{\beta^a}\!\left[(d_\cI(\beta^0,\beta^{b-a}))^p\right]\right]
    +\bE_{\mu_\gamma}\!\big[|Z(a)-Z(b)|^p\cf\{D<a\}\big].
\end{align*}
For the first term, we condition on $D>a$ and apply the pseudo-stationarity of Proposition \ref{prop:type2:pseudo}. While the distribution of $\beta^a$ given $D>a$ may not be $\mu_{\gamma/(2\gamma a+1)}$, it is $\mu_{\gamma/(2\gamma a+1)}$ up to a potential swap of the two leftmost blocks, and the matching set up in the proof of Lemma \ref{gammaholder} is unaffected by such a swap so that scaling by $2\gamma a+1$ and applying the bound of Lemma \ref{gammaholder} yields the upper bound:
\begin{align*}
  \bE_{\mu_{\gamma/(2\gamma a+1)}}\!\left[\big(d_\cI\big(\beta^0,\beta^{b-a}\big)\big)^p\right]
  &=(2\gamma a+1)^p\bE_{\mu_\gamma}\!\left[\big(d_\cI\big(\beta^0,\beta^{(b-a)/(2\gamma a+1)}\big)\big)^p\right]\\
  &\le (2\gamma a+1)^pC(b-a)^{\theta p}/(2\gamma a+1)^{\theta p}\le (2\gamma+1)^{p(1-\theta)}C|b-a|^{\theta p}.
\end{align*}
For the second term, we apply \cite[Lemma 33]{Paper0} to find the upper bound
$$|b-a|^{p/2}\bE_{\mu_\gamma}\!\left[\left(1+2(p-1)+2\sqrt{p-1}\sqrt{Z(0)+2(p-1)}\right)^p\right],$$
which is easily seen to be a finite multiple of $|b-a|^{p/2}\le|b-a|^{\theta p}$.

By the Kolmogorov-Chentsov theorem \cite[Theorem I.(2.1)]{RevuzYor}, this shows that for all $0<\theta<1/4$ and $p>0$,
$$\bE_{\mu_\gamma}\!\left[\left(\sup_{0\le a<b\le 1}\frac{d_\cI(\beta^a,\beta^b)}{|b-a|^\theta}\right)^p\right]<\infty.$$
We can write the LHS by integrating out the random initial mass. Cancelling $\gamma^{3/2}/\Gamma(3/2)$ gives
$$\int_0^\infty e^{-\gamma x}\sqrt{x}\bE_x\!\left[\left(\sup_{0\le a<b\le 1}\frac{d_\cI(\beta^a,\beta^b)}{|b-a|^\theta}\right)^p\right]dx<\infty.$$
By Fubini's theorem, this yields for a.e.\ $x\in(0,\infty)$ that
$\displaystyle\bE_x\!\left[\left(\sup_{0\le a<b\le 1}\frac{d_\cI(\beta^a,\beta^b)}{|b-a|^\theta}\right)^p\right] < \infty$.\\
But for any $x,y\in(0,\infty)$, we can find $c<1/y$ so that this expectation is finite for initial mass $xc$. By scaling, 
\begin{align*}
 \infty&>\bE_{c x}\!\left[\sup_{0\le a<b\le 1}\left(\frac{d_\cI(\beta^a,\beta^b)}{|b-a|^\theta}\right)^p\right]
       = c^p\bE_{x}\!\left[\sup_{0\le a<b\le 1}\left(\frac{d_\cI(\beta^{a/c},\beta^{b/c})}{|b-a|^\theta}\right)^p\right]\\
       &=c^{p(1-\theta)}\bE_{x}\!\left[\sup_{0\le a^\prime<b^\prime\le 1/c}\left(\frac{d_\cI(\beta^{a^\prime},\beta^{b^\prime})}{|b^\prime-a^\prime|^\theta}\right)^p\right]
       \ge c^{p(1-\theta)}\bE_{x}\!\left[\sup_{0\le a^\prime<b^\prime\le y}\left(\frac{d_\cI(\beta^{a^\prime},\beta^{b^\prime})}{|b^\prime-a^\prime|^\theta}\right)^p\right],
\end{align*}
so the expectation is finite for any initial mass, including unit initial mass $x=1$, and for any $y\in(0,\infty)$.       
\end{proof}

\bibliographystyle{abbrv}
\bibliography{AldousDiffusion4}

\end{document}